\documentclass[a4paper,11pt]{amsart}

\usepackage{amsmath} 
\usepackage{amssymb} 
\usepackage{amscd} 
\usepackage{amsxtra}

\font\rsfs=rsfs10
\newcommand{\marsfs}[1]{\mbox{\rsfs #1}}

\newcommand{\C}{\mathbb C}
\newcommand{\R}{\mathbb R}
\newcommand{\Z}{\mathbb Z}
\newcommand{\N}{\mathbb N} 
\newcommand{\Q}{\mathbb Q}
\newcommand{\T}{\mathbb T}
\newcommand{\F}{\mathbb F}
\newcommand{\Fc}{{\mathbb F}^{\rm c}} 
\newcommand{\rsF}{\marsfs{F}}

\newcommand{\hbFc}{\widehat{\mathbb F}^{\rm c}} 
\newcommand{\bL}{\mathbb L}
 
\newcommand{\Proj}{\mathbb P}

\newcommand{\J}{\mathbb J} 
\newcommand{\Jc}{\mathbb{J}^{\rm c}} 
\newcommand{\G}{\mathbb G} 
 
\newcommand{\K}{\mathbb K} 
\newcommand{\HH}{\mathbb H} 
\newcommand{\U}{\mathbb U} 

\newcommand{\seminf}{$\frac{\infty}{2}$} 

\newcommand{\ev}{\operatorname{ev}}
\newcommand{\age}{\operatorname{age}}
\newcommand{\Hom}{\operatorname{Hom}}
\newcommand{\End}{\operatorname{End}}

\newcommand{\Pic}{\operatorname{Pic}}
\newcommand{\Ker}{\operatorname{Ker}}

\newcommand{\Image}{\operatorname{Im}}
\newcommand{\Boxop}{\operatorname{Box}}

\newcommand{\id}{\operatorname{id}}
\newcommand{\Grading}{\operatorname{\mathsf{Gr}}}
\newcommand{\Res}{\operatorname{Res}}

\newcommand{\unit}{\operatorname{\boldsymbol{1}}}
\newcommand{\Span}{\operatorname{Span}}
\newcommand{\Tr}{\operatorname{Tr}}
\newcommand{\Eff}{\operatorname{Eff}}
\newcommand{\mir}{\operatorname{mir}}  
\newcommand{\Mir}{\operatorname{Mir}} 
\newcommand{\ch}{\operatorname{ch}} 
\newcommand{\tch}{\widetilde{\operatorname{ch}}}

\newcommand{\tTd}{\widetilde{\operatorname{Td}}}
\newcommand{\codim}{\operatorname{codim}} 
\newcommand{\Tor}{\operatorname{Tor}} 
\newcommand{\pr}{\operatorname{pr}} 
\newcommand{\Vol}{\operatorname{Vol}} 
\newcommand{\crit}{\operatorname{cr}} 
\newcommand{\Hess}{\operatorname{Hess}} 
\newcommand{\cl}{\operatorname{cl}}

\newcommand{\Aut}{\operatorname{Aut}} 
\newcommand{\Gr}{\operatorname{Gr}} 
\newcommand{\inv}{\operatorname{inv}}
\newcommand{\Log}{\operatorname{Log}} 
\newcommand{\pt}{{\operatorname{pt}}}

\newcommand{\Spec}{\operatorname{Spec}} 

\newcommand{\bN}{\boldsymbol{N}} 

\newcommand{\bs}{\boldsymbol{s}} 

\newcommand{\sfT}{\mathsf{T}} 
 
\newcommand{\sfp}{\mathsf{p}} 
\newcommand{\sfw}{\mathsf{w}} 
\newcommand{\sfm}{\mathsf{m}} 

\newcommand{\cF}{\mathcal{F}}
\newcommand{\cN}{\mathcal{N}}
\newcommand{\cA}{\mathcal{A}}
\newcommand{\cU}{\mathcal{U}}
\newcommand{\cO}{\mathcal{O}}
\newcommand{\cD}{\mathcal{D}}
\newcommand{\cL}{\mathcal{L}}
\newcommand{\cX}{\mathcal{X}}
\newcommand{\cH}{\mathcal{H}}

\newcommand{\hcH}{\widehat{\mathcal{H}}}
\newcommand{\cM}{\mathcal{M}}

\newcommand{\cR}{\mathcal{R}}
\newcommand{\cZ}{\mathcal{Z}}
\newcommand{\cP}{\mathcal{P}}
\newcommand{\cV}{\mathcal{V}}
  
\newcommand{\cI}{\mathcal{I}}

\newcommand{\cK}{\mathcal{K}}

\newcommand{\cMo}{\mathcal{M}^{\rm o}}
\newcommand{\cMoo}{\mathcal{M}^{\rm oo}} 
\newcommand{\cRz}{\mathcal{R}^{(0)}} 
\newcommand{\tcRz}{\widetilde{\mathcal{R}}^{(0)}}  
 
\newcommand{\Sol}{\mathcal{S}} 
\newcommand{\tSol}{\widetilde{\mathcal{S}}}

\newcommand{\tC}{\widetilde{C}}
\newcommand{\hrho}{\hat{\rho}} 
\newcommand{\hGamma}{{\widehat{\Gamma}}}
\newcommand{\hS}{\widehat{S}}

\newcommand{\frD}{\mathfrak{D}}

\newcommand{\frJ}{\mathfrak{J}} 

\newcommand{\tM}{\widetilde{M}}

\newcommand{\ov}{\overline}

\newcommand{\iu}{\pmb{\mathtt{i}}} 
\newcommand{\Yo}{Y^{\rm o}}

\newcommand{\ceil}[1]{\lceil #1\rceil}
\newcommand{\floor}[1]{\lfloor #1 \rfloor}

\newtheorem{theorem}{Theorem}[section]
\newtheorem{lemma}[theorem]{Lemma}
\newtheorem{proposition}[theorem]{Proposition}
\newtheorem{conjecture}[theorem]{Conjecture}
\newtheorem{corollary}[theorem]{Corollary} 

\theoremstyle{definition}
\newtheorem{definition}[theorem]{Definition}
\newtheorem{remark}[theorem]{Remark}

\newtheorem{assumption}[theorem]{Assumption}
\newtheorem{proposal}[theorem]{Proposal} 

\def\pair#1#2{\left\langle #1,#2\right\rangle}
\def\Pair#1#2{ (\!( #1,#2 )\!)} 
\def\parfrac#1#2{\frac{\partial{#1}}{\partial #2}}
\def\corr#1{\left\langle #1 \right\rangle}

\begin{document}

\title[An integral structure in quantum cohomology]  
{An integral structure in quantum cohomology 
and mirror symmetry for toric orbifolds}
\author{Hiroshi Iritani}
\address{Faculty of Mathematics, Kyushu University, 6-10-1, 
Hakozaki, Higashiku, Fukuoka, 812-8581, Japan.}
\email{iritani@math.kyushu-u.ac.jp}
\address{Department of Mathematics, Imperial College London, 
Huxley Building, 180, Queen's Gate, London, 
SW7 2AZ, United Kingdom.}
\email{h.iritani@imperial.ac.uk}  
\begin{abstract}
We introduce an integral structure in orbifold 
quantum cohomology associated to the $K$-group 
and the $\hGamma$-class.  
In the case of compact toric orbifolds, 
we show that this integral structure matches with 
the natural integral structure for the Landau-Ginzburg model 
under mirror symmetry. 
By assuming the existence of an integral 
structure, we give a natural explanation for 
the specialization to a root of unity in Y. Ruan's 
crepant resolution conjecture \cite{ruan-crc}.   
\end{abstract} 
\maketitle 

\tableofcontents 

\section{Introduction} 
Mirror symmetry for Calabi-Yau manifolds 
can be formulated as an isomorphism of 
\emph{variations of Hodge structures} (VHS for short):  
The A-model VHS \cite{morrison-mathaspects} 
defined by the genus zero 
Gromov-Witten theory of $X$ 
is isomorphic to the B-model VHS 
associated to deformation of complex structures 
of the mirror $Y$. 
As a consequence, one can calculate Gromov-Witten 
invariants of $X$ from Picard-Fuchs equations for $Y$;  
such phenomena have been checked in many examples 
including toric complete intersections 
\cite{givental-mirrorthm-toric, CCLT:wp}. 
However, while the B-model VHS has a natural integral 
local system $H^n(Y,\Z)$, the A-model VHS 
seems to lack an integral structure. 
In this paper, we study the question 
\emph{``What is the integral local system 
in the A-model mirrored from the B-model?''}  
Our calculation on compact toric orbifolds suggests 
that the $K$-group of $X$ should give 
the integral local system in the A-model.

Let us describe our $K$-theory 
integral structure in the A-model. 
The genus zero Gromov-Witten theory 
defines a family of commutative algebras 
$(H^*(X),\circ_\tau)$ on the cohomology group 
parametrized by $\tau\in H^*(X)$, 
called \emph{quantum cohomology}.  
The \emph{quantum $D$-module} is given by 
a flat connection $\nabla$ on 
the trivial bundle $H^*(X)\times H^*(X) \to H^*(X)$  
with a parameter $z\in \C^*$, 
called the \emph{Dubrovin connection}:  
\[
\nabla_X = d_X + \frac{1}{z} X\circ_\tau, \quad 
X\in H^*(X), 
\]
where $\tau$ denotes a point on the base 
and $d_X$ is the directional derivative 
(with respect to the given trivialization). 
We can extend the Dubrovin connection 
in the direction of the parameter $z$ 
(see Definition \ref{def:QDM}) 
and get a flat $H^*(X)$-bundle 
over $H^*(X)\times \C^*$. 
A general solution to the differential equation 
$\nabla s(\tau,z)=0$ is of the form 
$s(\tau,z) = L(\tau,z) z^{-\mu} z^{c_1(X)} \phi$ 
for some $\phi\in H^*(X)$. 
Here $\mu$ is the grading operator (\ref{eq:def_mu}) 
and $L(\tau,z)$ is the fundamental 
solution (\ref{eq:fundamentalsol_L}) 
which is asymptotic to $e^{-\tau/z}$ 
in the large radius limit (\ref{eq:largeradiuslimit}).  
Let $\delta_1,\dots,\delta_n$ be the Chern roots of 
the tangent bundle $TX$ and define a transcendental 
characteristic class $\hGamma(T\cX)$ by 
(see (\ref{eq:hGamma}) for orbifold case) 
\begin{align*}
\hGamma(TX) &:= 
\prod_{i=1}^n \Gamma(1+\delta_i) 
= \exp\biggl(-\gamma c_1(X) + \sum_{k\ge 2} (-1)^k (k-1)! \zeta(k) 
\ch_k(TX) \biggr),  
\end{align*} 
where $\gamma$ is the Euler constant and 
$\zeta(s)$ is the Riemann zeta function. 
For $V\in K(X)$, we define a $\nabla$-flat section  
$\cZ(V)$ to be (see (\ref{eq:Psi})) 
\[
\cZ(V)(\tau,z) := (2\pi)^{-n/2} 
L(\tau,z) z^{-\mu}z^{c_1(X)} 
\left( \hGamma(TX)\cup  
(2\pi\iu)^{\deg/2} \ch(V)\right),   
\]
where $n=\dim X$. 
These flat sections $\cZ(V)$, $V\in K(X)$ 
define an integral lattice 
in the space of $\nabla$-flat sections. 
We call it the \emph{$\hGamma$-integral structure}. 

The mirror of a compact toric orbifold  
is given by a Landau-Ginzburg (LG) model.  
It is a pair of a torus $Y_q=(\C^*)^n$ and 
a Laurent polynomial $W_q \colon Y_q \to \C$ on it 
($q$ is a parameter). 
The LG model defines a \emph{B-model $D$-module} 
which is underlain by a natural integral local 
system generated by \emph{Lefschetz thimbles} 
of $W_q$. 
Under mirror symmetry (Conjecture \ref{conj:mirrorthm}), 
the quantum $D$-module of a toric orbifold is isomorphic 
to the B-model $D$-module (Proposition \ref{prop:Dmoduleiso}).  
Our main theorem is the following: 

\begin{theorem}[Theorem \ref{thm:pulledbackintstr}] 
\label{thm:introd_main} 
Let $\cX$ be a weak Fano projective toric orbifold  
constructed from the initial data satisfying 
$\hrho \in \tC_\cX$ (see Section \ref{subsec:toricorbifolds}).  
Assume that mirror theorem 
(Conjecture \ref{conj:mirrorthm}) 
and Assumption \ref{assump:Ktheory} (c) 
hold for $\cX$. 
The $\hGamma$-integral structure on the quantum 
$D$-module corresponds to the natural 
integral local system of the B-model $D$-module 
under the mirror isomorphism in Proposition \ref{prop:Dmoduleiso}. 
\end{theorem} 

Conjecture \ref{conj:mirrorthm} 
will be proved in joint work \cite{CCIT:toric} 
with Coates, Corti and Tseng. 
In fact, both of the assumptions in the theorem 
are known to be true for toric \emph{manifolds}. 
This theorem follows from the following 
equality of ``central charges". 
We define the \emph{quantum cohomology central charge}
of $V\in K(X)$ to be 
\[
Z(V)(\tau,z) := \frac{(2\pi z)^{n/2}}{(2\pi\iu)^n} 
\int_X \cZ(V)(\tau,z).  
\]
Under Conjecture \ref{conj:mirrorthm}, 
$Z(V)$ is given as a pairing 
of $\ch(V)$ and 
a cohomology-valued hypergeometric series $H(q,z)$  
(see (\ref{eq:cc_byH}) and (\ref{eq:H-series})). 

\begin{theorem}[Theorem \ref{thm:cc_match}]
\label{thm:introd_cc} 
Under the same assumptions  
as Theorem \ref{thm:introd_main}, 
the quantum cohomology central charge of the 
structure sheaf $\cO_\cX$ is given by 
the oscillatory integral over 
the real Lefschetz thimble $\Gamma_\R$: 
\begin{equation}
\label{eq:introd_ccstr}
Z(\cO_\cX)(\tau(q),z) 
= \frac{1}{(2\pi\iu)^n} 
\int_{\Gamma_\R\subset Y_q} 
e^{-W_q(y)/z} \omega_q 
\end{equation} 
where $\tau=\tau(q)$ is a mirror map 
in Lemma \ref{lem:convergence}.  
\end{theorem} 

The relationship between $K$-theory 
and quantum cohomology can be foreseen by 
Kontsevich's \emph{homological mirror symmetry}. 
The integral local system of the 
B-model VHS on a Calabi-Yau $Y$ can be measured  
by integration (period) over a Lagrangian $n$-cycle, 
an object of the Fukaya category of $Y$ (A-type D-brane). 
Therefore, by homological mirror symmetry, 
a coherent sheaf on $X$, 
an object of the derived category of $X$  (B-type D-brane)   
should have a pairing with the quantum $D$-module and 
give a (dual) flat section of the Dubrovin connection.  
The quantum cohomology central charge $Z(V)$ 
can be viewed as a ``period of $V$" and 
the equality (\ref{eq:introd_ccstr})  
should be generalized as 
\[
Z(V)(\tau(q),z)  = \frac{1}{(2\pi\iu)^n} 
\int_{\mir(V)} e^{-W_q/z} \omega_q,  
\]
where $\mir(V)$ is the 
Lefschetz thimble mirror to $V$. 
Theorem \ref{thm:introd_main} shows 
the existence of the map $V \mapsto \mir(V)$ 
on the $K$-group level. 
This shows a $K$-group version of 
Dubrovin's conjecture 
(Corollary \ref{cor:K_Dubrovin}).

In the context of toric mirror symmetry,  
closely related observations have been made   
by Horja \cite{horja1}, Hosono \cite{hosono} 
and Borisov-Horja \cite{borisov-horja-FM}. 
Borisov-Horja \cite{borisov-horja-FM} identified 
the space of solutions to the GKZ system 
(corresponding to a toric Calabi-Yau $\cX$) 
with the $K$-group of $\cX$ and showed 
that the analytic continuation of a solution 
corresponds to a Fourier-Mukai transformation 
between birational $\cX$'s.  
Hosono \cite{hosono} proposed a central charge formula 
for Calabi-Yau complete intersections 
in toric varieties in terms of 
an explicit hypergeometric series. 
Our observation is based on non-Calabi-Yau examples, 
but all of their results can be understood from  
the $\hGamma$-integral structure. 
After the preprint version 
\cite{iritani-realint-preprint} 
of this paper was written, 
a rational structure based 
on the same $\hGamma$-class was proposed by 
Katzarkov-Kontsevich-Pantev \cite{KKP} 
independently.

We hope that an integral structure 
exists globally on the K\"{a}hler moduli space --- 
the (maximal) base space where the quantum cohomology is 
analytically continued. 
A global existence of an integral structure 
is relevant to Yongbin Ruan's 
\emph{crepant resolution conjecture}. 
Roughly speaking, it says that 
for a crepant resolution $Y$ of an orbifold $\cX$, 
quantum cohomology of $\cX$ and $Y$ 
are related by analytic continuation. 
In joint work \cite{CIT:I} with Coates and Tseng, 
we proposed the picture that the \emph{semi-infinite 
variations of Hodge structures} (\seminf VHS) 
associated to quantum cohomology of $\cX$ and $Y$ 
match under a linear symplectic transformation 
$\U \colon \cH^\cX \to \cH^Y$ 
between Givental's symplectic spaces 
$\cH^\cX$, $\cH^Y$ 
(which are loop spaces on cohomology groups, 
see (\ref{eq:Giventalsp})).  
This implies that the quantum $D$-modules 
of $\cX$ and $Y$ are isomorphic after analytic continuation.  
In this paper, we furthermore conjecture that 
\emph{the isomorphism of quantum $D$-modules  
preserves the $K$-theory integral structures 
on the both sides}. 
Then the symplectic transformation $\U$ 
would be induced from an isomorphism 
$\U_K \colon K(\cX) \to K(Y)$ 
of $K$-groups (\ref{eq:CD_UK_U})  
($\Psi$ below involves the $\hGamma$ class,  
see (\ref{eq:Psi})):  
\[ 
\begin{CD}
K(\cX) @>{\U_K}>> K(Y) \\ 
@V{z^{-\mu}z^{\rho}\Psi^\cX}VV  
@VV{z^{-\mu}z^{\rho} \Psi^Y}V  \\
\cH^\cX\otimes_{\cO(\C^*)}\cO(\widetilde{\C^*}) 
 @>{\U}>> 
\cH^Y \otimes_{\cO(\C^*)}\cO(\widetilde{\C^*}).    
\end{CD} 
\] 
In view of Borisov-Horja \cite{borisov-horja-FM}, 
we hope that $\U_K$ is given by 
a geometric correspondence such as 
a Fourier-Mukai transformation. 
This picture (Proposal \ref{propo:crc_int}) 
gives us a natural reason why  
the quantum parameters should be specialized 
to roots of unity at the orbifold 
large radius limit point. 
In some cases, one can predict explicitly  
the specialization value/co-ordinate change 
using the $\hGamma$-class.

This paper is a revision of 
the preprint \cite{iritani-realint-preprint}, 
where we also studied possible \emph{real structures}  
on quantum cohomology \seminf VHS, 
yielding Hertling's TERP structure 
\cite{hertling-tt*, hertling-sevenheck}. 
We showed that the $(p,p)$-part of quantum cohomology 
\seminf VHS is \emph{pure} and \emph{polarized} 
near the large radius limit point 
with respect to the real structure induced 
from the $\hGamma$-integral structure 
\cite[Theorem 3.7]{iritani-realint-preprint}.  
These properties --- pure and polarized --- 
are semi-infinite analogues of the Hodge decomposition 
and Hodge-Riemann bilinear inequality and yield 
\emph{$tt^*$-geometry} 
\cite{cecotti-vafa-top-antitop, hertling-tt*} 
on quantum cohomology.  
The real structure part of the preprint 
\cite{iritani-realint-preprint} 
will appear in a separate paper \cite{iritani-real}. 

The paper is organized as follows. 
In Section 2, we introduce the $\hGamma$-integral 
structure in orbifold quantum cohomology after 
reviewing the basics on orbifold quantum cohomology. 
In Section 3, we introduce Landau-Ginzburg mirrors 
to toric orbifolds and construct 
the B-model $D$-module from the LG model. 
In Section 4, we formulate mirror symmetry 
for toric orbifolds in terms of a $D$-module, 
and prove the main theorem (Theorem \ref{thm:pulledbackintstr}). 
In Section 5, we propose the crepant resolution 
conjecture with an integral structure 
(Proposal \ref{propo:crc_int}) 
and study specialization values of quantum 
parameters using the notion of integral periods.

We assume the convergence of 
quantum cohomology throughout the paper. 
We consider only the even parity 
part of the cohomology, \emph{i.e.} 
$H^*(X)$ means  $\bigoplus_k H^{2k}(X)$. 
We also assume that a smooth Deligne-Mumford 
stack $\cX$ in this paper has the resolution property 
\emph{i.e.} every coherent sheaf is a quotient 
of a vector bundle, so that we can apply the 
orbifold Riemann-Roch (\ref{eq:orbifoldRR}) to $\cX$.  
(A toric orbifold has this property. 
See \cite[Theorem 2.1]{totaro}.) 
Note that the orbifold cohomology $H_{\rm orb}^*(\cX)$ 
is denoted also by $H_{\rm CR}^*(\cX)$ 
in the literature. 

\vspace{5pt}  
\noindent  
{\bf Acknowledgments} 
Thanks are due to 
Tom Coates, 
Alessio Corti, 
Hsian-Hua Tseng 
for many useful discussions and 
their encouragement. 
This project is motivated by joint works 
\cite{CIT:I, CCIT:An, CCIT:toric} with them. 
The author benefited from a number of  
valuable discussions with Martin Guest and Claus Hertling 
and expresses gratitude to them.  
The author thanks valuable comments from Jim Bryan, 
Yongbin Ruan and the referees of the paper. 
This research was supported by 
Inoue Research Award for Young Scientists, 
Grant-in-Aid for Young Scientists (B), 
19740039, 2007 
and EPSRC(EP/E022162/1). 
\vspace{5pt} 
 
\noindent{\bf Notation}
\vspace{5pt}
 
\begin{tabular}{ll}
$\iu$ & imaginary unit $\iu^2=-1$ \\ 
$\cX$ & smooth Deligne-Mumford stack \\
$X$ & coarse moduli space of $\cX$ \\ 
$I\cX$ & inertia stack of $\cX$ \\
$\sfT=\{0\}\cup \sfT'$ & index set of inertia components; \\
$\inv \colon I\cX \to I\cX, \ \sfT \to \sfT$ 
& involution $(x,g)\mapsto (x,g^{-1})$ \\ 
$\iota_v$ & age of inertia component $v\in \sfT$ \\
$n$, $n_v$ & $\dim_\C \cX$, $\dim_\C \cX_v$. 
\end{tabular}

\section{Integral structure in quantum cohomology}
\label{sec:A-model} 

In this section, we review orbifold quantum cohomology 
and introduce the integral structure 
associated to the $K$-group and the $\hGamma$-class. 
Gromov-Witten theory for orbifolds 
has been developed by Chen-Ruan 
\cite{chen-ruan:new_coh_orb, chen-ruan:GW} 
in the symplectic category 
and by Abramovich-Graber-Vistoli \cite{AGV} in the algebraic category. 
The definition of the integral structure  
makes sense for both categories, but 
we work in the algebraic category. 

\subsection{Orbifold quantum cohomology} 
Let $\cX$ be a proper smooth Deligne-Mumford stack over $\C$.  
Let $I\cX$ denote the \emph{inertia stack} of $\cX$, 
defined by the fiber product $\cX\times_{\cX\times \cX} \cX$ 
of the two diagonal morphisms $\Delta\colon \cX\to \cX\times \cX$. 
A point on $I\cX$ is given by a pair $(x,g)$ of a point 
$x\in \cX$ and $g\in \Aut(x)$. 
We call $g$ the \emph{stabilizer} of $(x,g)\in I\cX$. 
Let $\sfT$ be the index set of components of the $I\cX$. 
Let $0\in \sfT$ be the distinguished element 
corresponding to the trivial stabilizer. 
Set $\sfT'=\sfT\setminus\{0\}$. 
We have  
\[
I\cX = \bigsqcup_{v\in \sfT} \cX_v = 
\cX_0 \cup \bigsqcup_{v\in \sfT'} \cX_v, \quad 
\cX_0 = \cX. 
\] 
We associate a rational number $\iota_v$ 
to each connected component $\cX_v$ of $I\cX$. 
This is called \emph{age} or \emph{degree shifting number}.  
Take a point $(x,g)\in \cX_v$ and let 
\[
T_x \cX = \bigoplus_{0\le f<1} (T_x\cX)_f
\]
be the eigenspace decomposition of $T_x\cX$ 
with respect to the stabilizer action, where 
$g$ acts on $(T_x \cX)_f$ by $\exp(2\pi\iu f)$. 
We define 
\[
\iota_v = \sum_{0\le f< 1} f \dim_\C (T_x \cX)_f. 
\]
This is independent of the choice of a point $(x,g)\in \cX_v$. 
The \emph{(even parity) orbifold cohomology group} 
$H_{\rm orb}^*(\cX)$ is defined to be the 
sum of the (even degree) cohomology of $\cX_v$, $v\in \sfT$: 
\[
H_{\rm orb}^k(\cX) = 
\bigoplus_{\substack{v\in \sfT \\ k-2\iota_v \equiv 0 (2)}} 
H^{k-2\iota_v}(\cX_v,\C).  
\]
The degree $k$ of the orbifold cohomology 
can be a fractional number in general. 
Each factor $H^*(\cX_v,\C)$ in the right-hand side 
is same as the cohomology group of $\cX_v$ as a topological space.  
If not otherwise stated, we will use $\C$ 
as the coefficient of cohomology groups. 
We have an involution $\inv \colon I\cX \to I\cX$ defined by 
$\inv(x,g) = (x,g^{-1})$. 
This induces an involution $\inv \colon \sfT \to \sfT$. 
The \emph{orbifold Poincar\'{e} pairing} is defined to be 
\[
(\alpha, \beta)_{\rm orb} : = \int_{I\cX} \alpha \cup \inv^*(\beta) 
= \sum_{v\in \sfT} \int_{\cX_v} \alpha_v \cup \beta_{\inv(v)},   
\]
where $\alpha_v$, $\beta_{v}$ are 
the $v$-components of $\alpha$, $\beta$.  
This pairing is symmetric, non-degenerate over $\C$ 
and of degree $-2n$, where $n=\dim_\C\cX$. 

Now we assume that the coarse moduli space $X$ 
of $\cX$ is projective.  
The \emph{genus zero Gromov-Witten invariants} 
are integrals of the form: 
\begin{equation}
\label{eq:GWcorrelator}
\corr{\alpha_1\psi^{k_1},\dots,\alpha_l \psi^{k_l}}_{0,l,d}^\cX 
= \int_{[\cX_{0,l,d}]^{\rm vir}} \prod_{i=1}^l \ev_i^*(\alpha_i) \psi_i^{k_i}
\end{equation} 
where $\alpha_i \in H_{\rm orb}^*(\cX)$, $d\in H_2(X,\Z)$ 
and $k_i$ is a non-negative integer. 
$[\cX_{0,l,d}]^{\rm vir}$ is the virtual fundamental class 
of the moduli stack $\cX_{0,l,d}$ of 
genus zero, $l$-pointed stable maps to $\cX$ of degree $d$; 
$\ev_i\colon \cX_{0,l,d} \to I\cX$ is the evaluation map\footnote
{The map $\ev_i$ here is defined only 
as a map of topological spaces (not as a map of stacks). 
The evaluation map defined in \cite{AGV} 
is a map of stacks but takes values in the 
\emph{rigidified inertia stack}, 
which is the same as $I\cX$ as a topological space but 
is different as a stack. } 
at the $i$-th marked point; 
$\psi_i$ is the first Chern class 
of the line bundle over $\cX_{0,l,d}$ 
whose fiber at a stable map is the cotangent space 
of the coarse curve at the $i$-th marked point. 
(Our notation is taken from \cite{CCLT:wp}; 
$\cX_{0,l,d}$ is denoted by $\cK_{0,l}(\cX,d)$ in \cite{AGV}.)  
The correlator (\ref{eq:GWcorrelator}) is 
non-zero only when $d$ belongs to 
$\Eff_\cX\subset H_2(X,\Z)$, the semigroup generated 
by effective stable maps, and $\sum_{i=1}^l (\deg \alpha_i + 2k_i) = 
2n + 2\pair{c_1(T\cX)}{d} + 2l -6$.

Let $\{\phi_k\}_{k=1}^N$ and 
$\{\phi^k\}_{k=1}^N$ be bases of $H_{\rm orb}^*(\cX)$ 
which are dual with respect to the orbifold Poincar\'{e} pairing, 
\emph{i.e.} $(\phi_i, \phi^j)_{\rm orb} = \delta_i^j$. 
The \emph{orbifold quantum product} $\circ_\tau$ 
is a formal family of commutative and associative 
products on $H_{\rm orb}^*(\cX)\otimes \C[\![\Eff_\cX]\!]$ 
parametrized by $\tau\in H^*_{\rm orb}(\cX)$. 
This is defined by the formula: 
\begin{equation*} 
\alpha \bullet_\tau \beta = 
\sum_{d\in \Eff_\cX} \sum_{l\ge 0} \sum_{k=1}^N    
\frac{1}{l!} 
\corr{\alpha,\beta,\tau,\dots,\tau, 
\phi_k}_{0,l+3,d}^\cX Q^d \phi^k,   
\end{equation*} 
where $Q^d$ is the element of the group ring $\C[\Eff_\cX]$ 
corresponding to $d\in \Eff_\cX$. 
We decompose $\tau\in H_{\rm orb}^*(\cX)$ as 
\begin{equation}
\label{eq:decomp_tau}
\tau = \tau_{0,2} + \tau', \quad  
\tau_{0,2}\in H^2(\cX), \quad 
\tau' \in \bigoplus_{k\neq 1} H^{2k}(\cX) \oplus 
\bigoplus_{v\in \sfT'} H^*(\cX_v).
\end{equation} 
Using the divisor equation \cite{tseng:QRR, AGV}, 
we find 
\begin{align} 
\label{eq:quantumproduct_divisor}
\alpha \bullet_\tau \beta = \sum_{d\in \Eff_\cX} \sum_{l\ge 0} 
\sum_{k=1}^N 
\frac{1}{l!} 
\corr{\alpha,\beta,\tau',\dots,\tau',
\phi_k}_{0,l+3,d}^{\cX} 
e^{\pair{\tau_{0,2}}{d}} Q^d \phi^k. 
\end{align} 
Therefore, the quantum product can be viewed 
as a formal power series in $e^{\tau_{0,2}}Q$ and $\tau'$. 
When this is a convergent power series, we can put $Q=1$ 
and define
\[
\circ_\tau := \bullet_\tau|_{Q=1}.  
\]
Under the following convergence assumption, 
the product $\circ_\tau$ defines an analytic family of 
commutative rings $(H_{\rm orb}(\cX), \circ_\tau)$ 
over $U$: 
\begin{assumption} 
\label{assump:converge}
The orbifold quantum product $\circ_\tau$ 
is convergent over 
an open set $U\subset H_{\rm orb}^*(\cX)$ of the form: 
\[
U = \left\{ \tau\in H^*_{\rm orb}(\cX)\;;\; 
\Re \pair{\tau_{0,2}}{d} \le -M, \ 
\forall d\in \Eff_\cX\setminus\{0\}, \ 
\|\tau'\| \le e^{-M} \right\} 
\]
for a sufficiently big $M>0$, 
where $\tau = \tau_{0,2}+ \tau'$ is the decomposition 
in (\ref{eq:decomp_tau}) and $\|\cdot\|$ is some norm 
on $H_{\rm orb}(\cX)$.  
\end{assumption} 

The domain $U$ here contains the following limit direction:  
\begin{equation}
\label{eq:largeradiuslimit}
\Re\pair{\tau_{0,2}}{d} \to -\infty, \quad 
\forall d\in \Eff_\cX\setminus\{0\}, 
\quad \tau' \to 0.  
\end{equation} 
This is called the \emph{large radius limit}. 
In the large radius limit, $\circ_\tau$ goes to the 
orbifold cup product $\cup_{\rm orb}$ 
due to Chen-Ruan \cite{chen-ruan:new_coh_orb}. 
(For manifolds, $\cup_{\rm orb}$ is the same as 
the cup product.)

\subsection{Quantum $D$-modules and Galois action}
\label{subsec:QuantumDmod} 
We associate a meromorphic flat connection 
(quantum $D$-module) to the orbifold quantum cohomology. 
We introduce certain automorphisms of 
the quantum $D$-module, which we call \emph{Galois actions}. 

Take a homogeneous basis 
$\{\phi_k\}_{k=1}^N$ of $H_{\rm orb}^*(\cX)$ 
and let $\{t^k\}_{k=1}^N$ be the linear 
co-ordinate system on $H_{\rm orb}^*(\cX)$ 
dual to $\{\phi_k\}_{k=1}^N$. 
Let $\tau = \sum_{k=1}^N t^k \phi_k$ be a 
general point on $U\subset H^*_{\rm orb}(\cX)$. 
Let $(\tau,z)$ be a general point on 
$U\times \C$ and $(-)\colon U\times \C\to U\times \C$ 
be the map sending $(\tau,z)$ to $(\tau,-z)$. 

\begin{definition} 
\label{def:QDM} 
The \emph{quantum $D$-module} $QDM(\cX)$ 
or \emph{A-model $D$-module} 
is the tuple $(F,\nabla,(\cdot,\cdot)_F)$ 
consisting of the trivial holomorphic vector bundle 
$F := H^*(\cX)\times (U\times \C) \to U\times \C$,  
the meromorphic flat connection $\nabla$  
\begin{align*} 
&\nabla_k = \nabla_{\parfrac{}{t^k}} 
= \parfrac{}{t^k} + \frac{1}{z} \phi_k \circ_\tau \\ 
&\nabla_{z\partial_z}  
= z \parfrac{}{z} - \frac{1}{z} E\circ_\tau + \mu,  
\end{align*} 
and the $\nabla$-flat pairing $(\cdot,\cdot)_F$: 
\[
(\cdot,\cdot)_F \colon (-)^* \cO(F) \otimes \cO(F) \to \cO_{U\times \C} 
\]
induced from the orbifold Poincar\'{e} pairing 
$F_{(\tau,-z)}\times F_{(\tau,z)}  = H_{\rm orb}^*(\cX)
\times H_{\rm orb}^*(\cX) \to \C$. 
Here $E\in \cO(F)$ is the \emph{Euler vector field}  
\begin{equation}
\label{eq:def_E}
E :=c_1(T\cX) + 
\sum_{k=1}^N (1- \frac{1}{2}\deg \phi_k) t^k \phi_k   
\end{equation} 
and $\mu\in \End(H_{\rm orb}^*(\cX))$ 
is the \emph{Hodge grading operator} 
\begin{equation}
\label{eq:def_mu} 
\mu(\phi_k): = (\frac{1}{2}\deg\phi_k -\frac{n}{2}) \phi_k. 
\end{equation} 
The flat connection $\nabla$ is called the \emph{Dubrovin connection} 
or \emph{the first structure connection}. 
The standard argument (as in \cite{cox-katz,manin}) 
and the WDVV equation in orbifold Gromov-Witten theory \cite{AGV} 
show that the Dubrovin connection is flat.  
\qed
\end{definition} 

Note that the connection $\nabla$ defines a map: 
\[
\nabla \colon \cO(F) \to \cO(F)(U\times \{0\}) 
\otimes_{\cO_{U\times \C}}  
(\pi^*\Omega_U^1 \oplus \cO_{U\times \C} \frac{dz}{z}),  
\]
where $\pi\colon U\times \C \to U$ is the projection. 
By identifying  $\phi_i$ with the vector field 
$\partial/\partial t^i$, one can regard $E$ 
as the vector field over $U$: 
\begin{equation}
\label{eq:E_vectorfield}
E = \sum_{k=1}^N r_k \parfrac{}{t^k} + \sum_{k=1}^N 
(1-\frac{1}{2} \deg \phi_k) t^k\parfrac{}{t^k}, 
\end{equation} 
where we put $c_1(\cX) = \sum_{k=1}^N r_k \phi_k$. 
The Euler vector field satisfies the property: 
\begin{equation} 
\label{eq:Euler_reg} 
\Grading := 2 ( \nabla_{z\partial_z} + \nabla_E + \frac{n}{2})  
\quad \text{is regular at $z=0$}. 
\end{equation} 
The operator $\Grading\colon \cO(F) \to \cO(F)$ 
defines the grading for sections of $F$.

Let $H^2(\cX,\Z)$ denote the cohomology of the constant sheaf $\Z$ 
on the topological \emph{stack} $\cX$ 
(not on the topological \emph{space}). 
This group is the set of isomorphism classes of 
topological orbifold line bundles on $\cX$. 
Let $L_\xi \to \cX$ be the orbifold line bundle corresponding to 
$\xi\in H^2(\cX,\Z)$. 
Let $0\le f_v(\xi)<1$ be the rational number such that 
the stabilizer of $\cX_v$ ($v\in \sfT$) acts on $L_\xi|_{\cX_v}$ by 
a complex number $\exp(2\pi \iu f_v(\xi))$. 
This number $f_v(\xi)$ is called the \emph{age}  
of $L_\xi$ along $\cX_v$. 

\begin{proposition}
\label{prop:Galois} 
For $\xi\in H^2(\cX,\Z)$, 
the bundle isomorphism of $F$ defined by 
\begin{align*} 
H^*_{\rm orb}(\cX) \times (U\times \C) 
& \longrightarrow 
H^*_{\rm orb}(\cX) \times (U\times \C) \\ 
(\alpha, \tau,z) & 
\longmapsto ( dG(\xi) \alpha, G(\xi)\tau, z)  
\end{align*}
gives an automorphism of the quantum $D$-module, 
\emph{i.e.} 
preserves the flat connection $\nabla$ and 
the pairing $(\cdot,\cdot)_F$.  
Here $G(\xi), dG(\xi) \colon 
H^*_{\rm orb}(\cX)\to H^*_{\rm orb}(\cX)$ 
are defined by 
\begin{align}
\label{eq:Galois}
\begin{split}  
G(\xi)(\tau_0 \oplus \bigoplus_{v \in \sfT'} \tau_v) 
&=(\tau_0- 2\pi \iu \xi_0) \oplus \bigoplus_{v\in \sfT'} 
e^{2\pi \iu f_v(\xi)} \tau_v,  \\ 
dG(\xi)(\tau_0 \oplus \bigoplus_{v \in \sfT'} \tau_v) 
&=\tau_0 \oplus \bigoplus_{v\in \sfT'} 
e^{2\pi \iu f_v(\xi)} \tau_v,   
\end{split} 
\end{align}
where $\tau_v \in H^*( \cX_v)$ and $\xi_0$ is the image of 
$\xi$ in $H^2( \cX,\Q)$. 
We call this \emph{Galois action} 
of $H^2(\cX,\Z)$ on $QDM(\cX)$.  
\end{proposition} 
\begin{proof} 
For $\alpha_1,\dots,\alpha_l\in H_{\rm orb}^*(\cX)$, 
we claim that 
\[
\langle \alpha_1,\alpha_2,\dots,\alpha_l \rangle_{0,l,d} 
= e^{- 2 \pi \iu \pair{\xi_0}{d}}  
\langle dG(\xi)\alpha_1, dG(\xi)\alpha_2, \dots, 
dG(\xi)\alpha_l \rangle_{0,l,d}.  
\]
If there exists an orbifold stable map 
$f\colon (C,x_1,\dots,x_l) \to \cX$ of degree $d$, 
we have an orbifold line bundle $f^*L_\xi$ on $C$ 
such that the monodromy at $x_k$ 
equals $\exp(2\pi \iu f_{v_k}(\xi))$ 
where $\ev_k(f)\in \cX_{v_k}$. Then we must have 
\[
\deg f^*L_\xi - \sum_{k=1}^l f_{v_k} \in \Z, \quad \emph{i.e. } 
e^{-2\pi \iu \pair{\xi_0}{d}} 
\prod_{i=1}^l e^{ 2 \pi \iu f_{v_i}(\xi)} =1.  
\]
The claim follows from this. The lemma follows from this claim 
and (\ref{eq:quantumproduct_divisor}). 
\end{proof} 

Without loss of generality, 
we can assume that $U$ is invariant under the Galois action. 
By the Galois action, the quantum $D$-module 
descends to the quotient 
$F/H^2(\cX,\Z) \to (U/H^2(\cX,\Z))\times \C$. 
We refer to this flat connection over 
$(U/H^2(\cX,\Z))\times \C$ also as 
the \emph{quantum $D$-module}.

\subsection{The space of solutions to 
the quantum differential equation}

The equation $\nabla s =0$ for a section 
$s$ of $F$ is called 
the \emph{quantum differential equation}. 
A fundamental solution $L(\tau,z)$ to the quantum 
differential equation can be given by 
gravitational descendants. 
Let $\pr\colon I\cX \to \cX$ be the natural projection.  
We define the action of a class $\tau_0 \in H^*(\cX)$ on 
$H_{\rm orb}^*(\cX)$ by 
\[
\tau_0 \cdot \alpha = \pr^*(\tau_0) \cup \alpha, \quad 
\alpha \in H_{\rm orb}^*(\cX),  
\]
where the right-hand side is the cup product on $I\cX$.
We define 
\begin{equation}
\label{eq:fundamentalsol_L}
L(\tau,z) \alpha := e^{-\tau_{0,2}/z} \alpha - 
\sum_{\substack{(d,l)\neq (0,0) \\ d\in \Eff_\cX, 1\le k\le N}}
\frac{\phi^k}{l!} \corr{\phi_k, \tau',\dots,\tau', 
\frac{e^{-\tau_{0,2}/z} \alpha}{z+\psi}}_{0,l+2,d}^\cX 
e^{\pair{\tau_{0,2}}{d}},  
\end{equation} 
where $\tau=\tau_{0,2} + \tau'$ 
is the decomposition in (\ref{eq:decomp_tau}) 
and 
$1/(z+\psi)$ in the correlator should be expanded in the series 
$\sum_{k=0}^\infty (-1)^k z^{-k-1}\psi^k$.  
The following proposition is well-known for manifolds 
\cite{pandharipande, cox-katz}. 

\begin{proposition}
\label{prop:fundamentalsol_A} 
$L(\tau,z)$ satisfies the following differential equations: 
\begin{align}
\label{eq:diffeq_L}
\nabla_k L(\tau,z) \alpha =0, \quad 
\nabla_{z\partial_z} L(\tau,z) \alpha = 
L(\tau,z) (\mu \alpha -\frac{\rho}{z}\alpha),  
\end{align}
where $\alpha\in H^*_{\rm orb}(\cX)$, 
$\rho := c_1(T\cX)\in H^2(\cX)$ and 
$\mu$ is the grading operator (\ref{eq:def_mu}). 
The flat section $L(\tau,z)\alpha$ 
(flat in the $\tau$-direction) is characterized 
by the asymptotic initial condition: 
\begin{equation}
\label{eq:asymptotic_initial}
L(\tau,z) \alpha \sim e^{-\tau_{0,2}/z} \alpha 
\end{equation} 
in the large radius limit (\ref{eq:largeradiuslimit}) 
with $\tau'=0$.  
Set 
\[
z^{-\mu} z^\rho := \exp(-\mu \log z) \exp(\rho \log z).   
\] 
Then we have 
\begin{gather}
\label{eq:diffeq_L_zmuzrho} 
\nabla_k (L(\tau,z) z^{-\mu}z^\rho \alpha) = 0, \quad 
\nabla_{z\partial_z} (L(\tau,z) z^{-\mu}z^\rho \alpha) 
= 0, \\
\label{eq:unitarity_L} 
(L(\tau,-z) \alpha, L(\tau,z)\beta)_{\rm orb}  
= (\alpha,\beta)_{\rm orb},  \\
\label{eq:Galois_L} 
dG(\xi) L(G(\xi)^{-1} \tau, z) \alpha   
= L(\tau,z) e^{-2\pi\iu \xi_0/z} e^{2\pi\iu f_v(\xi)} 
\alpha,\  \text{if} \ \alpha\in H^*(\cX_v),  
\end{gather}
where $dG(\xi), G(\xi)$ are the Galois actions  
for $\xi\in H^2(\cX,\Z)$ in Section \ref{subsec:QuantumDmod}.  
\end{proposition} 
\begin{proof} 
The first equation of (\ref{eq:diffeq_L})   
follows from the topological recursion relation 
\cite[2.5.5]{tseng:QRR} in orbifold Gromov-Witten theory. 
The proof for the case of manifolds can be found in 
\cite[Proposition 2]{pandharipande}, \cite[Chapter 10]{cox-katz} 
and the proof for orbifolds is completely parallel. 

For the second equation of (\ref{eq:diffeq_L}), 
note that we can decompose $L$ as 
$L(\tau,z) = S(\tau,z) \circ e^{-\tau_{0,2}/z}$
for some $\End(H_{\rm orb}^*(\cX))$-valued function $S(\tau,z)$. 
The homogeneity of Gromov-Witten invariants 
shows that $S$ preserves the degree, \emph{i.e.} 
$(z\partial_z +E + \mu) \circ S(\tau,z) 
= S(\tau,z) \circ (z\partial_z + E + \mu)$, where $E$ 
is regarded as the vector field (\ref{eq:E_vectorfield}).  
Therefore, $(z\partial_z +E+ \mu) \circ 
L(\tau,z) = L(\tau,z) \circ (z\partial_z + E + \mu - \rho/z)$. 
The second equation of (\ref{eq:diffeq_L}) 
follows from this and the first equation. 

The asymptotic initial condition (\ref{eq:asymptotic_initial}) 
is obvious from the definition (\ref{eq:fundamentalsol_L}). 

The equation (\ref{eq:diffeq_L_zmuzrho}) follows from 
(\ref{eq:diffeq_L}) and the fact that $z^{-\mu}z^{\rho}\alpha$ 
satisfies the differential equation 
$(z\partial_z + \mu -\rho/z) (z^{-\mu}z^\rho \alpha) =0$, 
which follows easily from the commutation relation 
$[\mu, \rho] = \rho$.  

To show the equation (\ref{eq:unitarity_L}), 
put $\bs' =L(\tau,-z)\alpha$ and $\bs = L(\tau,z)\beta$. 
By using (\ref{eq:diffeq_L}) and the Frobenius 
property $(\alpha\circ_\tau \beta, \gamma)_{\rm orb}=
(\alpha, \beta\circ_\tau \gamma)_{\rm orb}$, we have 
\[
\parfrac{}{t^k}(\bs',\bs)_{\rm orb} = 
\frac{1}{z}(\phi_k\circ_\tau \bs', \bs)_{\rm orb}
- \frac{1}{z}(\bs', \phi_k\circ_\tau \bs)_{\rm orb} =0.  
\] 
Hence $(\bs',\bs)_{\rm orb}$ is constant in $\tau$. 
Using the asymptotics 
$\bs' \sim e^{\tau_{0,2}/z} \alpha$ and 
$\bs \sim e^{-\tau_{0,2}/z} \beta$, 
we have 
\[(\bs',\bs)_{\rm orb} \sim 
(e^{-\tau_{0,2}/z} \alpha, e^{\tau_{0,2}/z}\beta)_{\rm orb} = 
(\alpha,\beta)_{\rm orb}
\]
and the equation (\ref{eq:unitarity_L}) follows. 

Since the Galois action preserves $\nabla$, 
it follows that 
$dG(\xi) L(G(\xi)^{-1} \tau,z)\alpha$ 
is flat in the $\tau$-direction. 
The equation (\ref{eq:Galois_L}) follows from the 
characterization (\ref{eq:asymptotic_initial}) 
and the asymptotics 
$dG(\xi) L( G(\xi)^{-1} \tau, z) \alpha 
\sim e^{-\tau_{0,2}/z} e^{-2\pi\iu \xi_0/z} e^{2\pi \iu f_v(\xi)} \alpha$.  
\end{proof} 

Although the convergence of $L(\tau,z)$ is not a priori clear, 
we know from the differential equations above 
and the convergence assumption of $\circ_\tau$ 
that $L(\tau,z)$ is convergent on $(\tau,z)\in U\times \C^*$. 

\begin{definition} 
The space $\Sol(\cX)$ of multi-valued $\nabla$-flat 
sections of the quantum 
$D$-module $(F,\nabla,(\cdot,\cdot)_F)$ 
is defined to be 
\[
\Sol(\cX) := 
\{s\in \Gamma(U\times \widetilde{\C^*}, \cO(F)) 
\;;\; \nabla s = 0\}, 
\]
where $\widetilde{\C^*}$ is the universal cover of $\C^*$. 
This is a finite dimensional $\C$-vector space 
with $\dim \Sol(\cX) = \dim H_{\rm orb}^*(\cX)$. 
The pairing $(\cdot,\cdot)_\Sol$ on $\Sol(\cX)$ is given by 
\begin{equation}
\label{eq:pairing_Sol}
(s_1,s_2)_{\Sol} :=  (s_1(\tau, e^{\pi\iu} z), s_2(\tau, z))_{\rm orb} 
\in \C, 
\end{equation} 
where $s_1(\tau,e^{\pi\iu}z)$ is the parallel translate 
of $s_1(\tau,z)$ along the counter-clockwise path 
$[0,1]\ni \theta \mapsto e^{\iu\pi\theta} z$. 
Note that the right-hand side is a complex number 
which does not depend on $(\tau,z)$. 
The Galois action in Proposition \ref{prop:Galois} 
defines an automorphism of $\Sol(\cX)$ 
for $\xi\in H^2(\cX,\Z)$: 
\begin{equation}
\label{eq:Galois_Sol}
G^{\Sol}(\xi) \colon \Sol(\cX) \to \Sol(\cX), \quad 
s(\tau,z) \mapsto dG(\xi) s(G(\xi)^{-1}\tau,z).  
\end{equation} 
Using the fundamental solution in Proposition 
\ref{prop:fundamentalsol_A}, 
we define the \emph{cohomology framing} $\cZ_{\rm coh} 
\colon H^*_{\rm orb}(\cX) \to \Sol(\cX)$ of $\Sol(\cX)$ 
by 
\begin{equation}
\label{eq:coh_framing}
\cZ_{\rm coh}(\alpha) := L(\tau,z) z^{-\mu} z^\rho \alpha.   
\end{equation} 
The pairing and the Galois action on $\Sol(\cX)$ 
can be written in terms of the cohomology framing as 
\begin{align} 
\label{eq:cohfr_property}
\begin{split}  
(\cZ_{\rm coh}(\alpha), \cZ_{\rm coh}(\beta))_{\Sol} 
& = (e^{\pi\iu \rho} \alpha, e^{\pi\iu \mu} \beta)_{\rm orb}, \\  
G^\Sol(\xi) (\cZ_{\rm coh}(\alpha)) &= \cZ_{\rm coh} 
( ( \bigoplus_{v\in \sfT} e^{-2\pi\iu\xi_0} e^{2\pi\iu f_v(\xi)}
) \alpha ).
\end{split}  
\end{align} 
Here $\xi_0\in H^2(\cX,\Q)$ and $f_v(\xi) \in [0,1) \cap \Q$ 
are introduced  before Proposition \ref{prop:Galois}. 
The first equation follows from (\ref{eq:unitarity_L}) 
and the second equation follows from (\ref{eq:Galois_L}). 
\qed 
\end{definition}

The Galois actions on $\Sol(\cX)$ 
can be viewed as the monodromy transformations of 
the flat bundle $F/H^2(\cX,\Z)\to (U/H^2(\cX,\Z))\times \C^*$
in the $\tau$-direction. 
The monodromy with respect to $z$ is given by 
\begin{equation}
\label{eq:z-monodromy_V}
\left [
\cZ_{\rm coh}(\alpha) \right]_{z\mapsto e^{2\pi\iu} z }  
= \cZ_{\rm coh} 
(e^{-2\pi \iu \mu} e^{2\pi \iu \rho} \alpha) 
\end{equation} 
This coincides with the Galois action 
$(-1)^n G^\Sol([K_\cX])$ 
and also corresponds to the Serre functor of 
the derived category $D(\cX)$. 
Here, $[K_\cX]$ is the class of 
the canonical line bundle. 
When $\cX$ is Calabi-Yau, 
\emph{i.e.} $K_\cX$ is trivial,  
the pairing $(\cdot,\cdot)_{\Sol}$ 
is either symmetric or anti-symmetric depending on
whether $n$ is even or odd.   
In general, this pairing is neither symmetric nor anti-symmetric.

\subsection{$\hGamma$-integral structure} 
\label{subsec:hGamma_intstr}
By an \emph{integral structure} in quantum cohomology 
we mean a $\Z$-local system $F_\Z \to U\times \C^*$ 
underlying the flat bundle $(F,\nabla)|_{U\times \C^*}$. 
This is given by an integral lattice $\Sol(\cX)_\Z$  
in the space $\Sol(\cX)$ 
of multi-valued flat sections of $QDM(\cX)$. 
There are a priori many choices of integral 
lattices in $\Sol(\cX)$. 
We introduce the $\hGamma$-integral structure 
which has several nice properties.  

Let $K(\cX)$ denote the Grothendieck group of topological 
orbifold vector bundles on $\cX$. 
See \emph{e.g.} \cite{adem-ruan, moerdijk} 
for vector bundles on orbifolds. 
For an orbifold vector bundle $\widetilde{V}$ 
on the inertia stack $I\cX$, 
we have an eigenbundle decomposition of $\widetilde{V}|_{\cX_v}$
\[
\widetilde{V}|_{\cX_v} = \bigoplus_{0\le f<1} 
\widetilde{V}_{v,f} 
\] 
with respect to the action of the stabilizer of $\cX_v$.    
Here, the stabilizer acts on 
$\widetilde{V}_{v,f}$ by $\exp(2\pi\iu f) \in \C$. 
Let $\pr \colon I\cX \to \cX$ be the projection. 
The Chern character $\tch \colon K(\cX) \to H^*(I\cX)$ 
is defined for an orbifold vector bundle $V$ on $\cX$ by 
\[
\tch(V) := \bigoplus_{v\in \sfT} \sum_{0\le f<1} e^{2\pi\iu f}
\ch((\pr^*V)_{v,f}) 
\]
where $\ch$ is the ordinary Chern character. 
For an orbifold vector bundle $V$ on $\cX$, 
let $\delta_{v,f,i}$, $i=1,\dots,l_{v,f}$ be the Chern roots of 
$(\pr^*V)_{v,f}$. 
The Todd class $\tTd\colon K(\cX) \to H^*(I\cX)$ is defined by  
\[
\tTd(V) = \bigoplus_{v\in \sfT} 
\prod_{0<f<1,1\le i\le l_{v,f}}\frac{1}{1-e^{-2\pi\iu f}e^{-\delta_{v,f,i}}}
\prod_{f=0,1\le i\le l_{v,0}} \frac{\delta_{v,0,i}}{1-e^{-\delta_{v,0,i}}}
\]
These characteristic classes appear in the following theorem. 

\begin{theorem}[Orbifold Riemann-Roch \cite{kawasaki-rr,toen}] 
Assume that $\cX$ has the resolution property 
(see \emph{e.g.} \cite{totaro}). 
For a holomorphic orbifold vector bundle $V$ on $\cX$,   
the Euler characteristic 
$\chi(V)$ is given by 
\begin{equation}
\label{eq:orbifoldRR}
\chi(V) := \sum_{i=0}^{\dim \cX} (-1)^i \dim H^i(\cX,V)
= \int_{I\cX} \tch(V)\cup \tTd(T\cX).  
\end{equation} 
\end{theorem}

Define a multiplicative 
characteristic class $\hGamma \colon K(\cX) \to H^*(I\cX)$ 
by 
\begin{equation}
\label{eq:hGamma}
\hGamma(V) := \bigoplus_{v\in \sfT} 
\prod_{0\le f<1} \prod_{i=1}^{l_{v,f}} 
\Gamma(1- f + \delta_{v,f,i}) 
\in H^*(I\cX),  
\end{equation} 
where $\delta_{v,f,i}$ is the same as above.  
The Gamma function on the right-hand side 
should be expanded in series at $1-f>0$. 
We assume the following conditions.

\begin{assumption} 
\label{assump:Ktheory} 
(a) The map $\tch\colon K(\cX) \to H^*(I\cX)$ 
becomes an isomorphism after tensored with $\C$. 

(b) The right-hand side of the orbifold 
Riemann-Roch formula (\ref{eq:orbifoldRR}) 
takes values in $\Z$ for any 
(not necessarily holomorphic)  
complex orbifold vector bundle $V$ on $\cX$. 
Define $\chi(V)$ to be the value of 
the right-hand side of (\ref{eq:orbifoldRR})
for any orbifold vector bundle $V$. 

(c) The pairing $(V_1,V_2)\mapsto 
\chi(V_1\otimes V_2)$ on $K(\cX)$ 
induces a surjective map $K(\cX) \to \Hom(K(\cX),\Z)$. 
\end{assumption}

\begin{remark} 
(i) When $\cX$ can be presented as a quotient 
$[Y/G]$ as a topological orbifold,  
where $Y$ is a compact manifold and $G$ is a compact Lie group 
acting on $Y$ with at most finite stabilizers,  
Part (a) of the assumption 
follows from Adem-Ruan's decomposition 
theorem \cite[Theorem 5.1]{adem-ruan}. 
Note that an orbifold without generic stabilizers 
can be presented as a quotient orbifold $[Y/G]$ 
(see \emph{e.g.} \cite{adem-ruan}).

(ii) When $\cX$ is again a quotient orbifold $[Y/G]$,  
Part (b) follows from Kawasaki's index theorem \cite{kawasaki-Vind} 
for elliptic operators on orbifolds (whose proof uses 
the $G$-equivariant index).   
The right-hand side of (\ref{eq:orbifoldRR}) 
becomes the index of a certain elliptic operator 
$\ov{\partial}+\ov{\partial}^*\colon 
V\otimes \Omega^{0,{\rm even}}_{\cX} \to 
V\otimes \Omega^{0,{\rm odd}}_{\cX}$, 
where $\ov{\partial}$ is a 
not necessarily integrable $(0,1)$ connection 
and $\ov{\partial}^*$ is its adjoint. 
The author does not know a purely topological proof. 

(iii) Part (c) would follow from a universal coefficient theorem and 
Poincar\'{e} duality for 
orbifold $K$-theory (which are true for manifolds), 
but the author does not know a proof nor a reference.  
\qed 
\end{remark}

\begin{definition} 
\label{def:A-model_int} 
We define the \emph{$K$-group framing} 
$\cZ_K\colon K(\cX) \to \Sol(\cX)$ 
of the space $\Sol(\cX)$ of multi-valued flat sections 
of the quantum $D$-module by the formula: 
\begin{align}
\label{eq:Psi} 
\begin{split} 
& \cZ_K(V) : = \cZ_{\rm coh} (\Psi (V)) 
= L(\tau,z) z^{-\mu} z^\rho \Psi(V), \\  
& \text{where} \quad \Psi(V) := (2\pi)^{-n/2} 
\hGamma(T\cX) \cup (2\pi\iu)^{\deg/2} \inv^* (\tch(V)).  
\end{split} 
\end{align} 
Here $\deg\colon H^*(I\cX)\to H^*(I\cX)$ is a grading operator 
on $H^*(I\cX)$ defined by $\deg = 2k$ on $H^{2k}(I\cX)$\footnote
{Note that $\deg$ is the degree of the cohomology class 
as an element of $H^*(I\cX)$, 
\emph{not} as an element of $H_{\rm orb}^*(\cX)$.} 
and $\cup$ is the cup product in $H^*(I\cX)$. 
We call the image $\Sol(\cX)_\Z := \cZ_K(K(\cX))$ of the 
$K$-group framing the \emph{$\hGamma$-integral structure}. 
\qed 
\end{definition}

\begin{proposition} 
\label{prop:A-model_int} 
Assume Assumption \ref{assump:Ktheory}.  
The $\hGamma$-integral structure $\Sol(\cX)_\Z$ 
satisfies the following properties.  

(i) By Part (a) of the assumption, 
$\Sol(\cX)_\Z$ is a $\Z$-lattice in $\Sol(\cX)$: 
\[
\Sol(\cX) = \Sol(\cX)_\Z \otimes_\Z \C.
\]

(ii) 
The Galois action $G^\Sol(\xi)$ on $\Sol(\cX)$ 
in (\ref{eq:Galois_Sol}) 
preserves the lattice $\Sol(\cX)_\Z$ 
and corresponds to the tensor by the line bundle 
$L_\xi^\vee$ in $K(\cX)$: 
\[
\cZ_K(L_\xi^\vee \otimes V) = G^{\Sol}(\xi)(\cZ_K(V)).
\]
where $L_\xi$ is the line bunlde corresponding to 
$\xi\in H^2(\cX,\Z)$. 

(iii)  
The pairing $(\cdot,\cdot)_{\Sol}$ on $\Sol(\cX)$ 
in (\ref{eq:pairing_Sol}) 
corresponds to the Mukai pairing on $K(\cX)$ defined by 
$(V_1,V_2)_{K(\cX)} := \chi(V_2^\vee \otimes V_1)$:  
\[
(\cZ_K(V_1),\cZ_K(V_2))_{\Sol} = (V_1,V_2)_{K(\cX)}. 
\]
In particular, the pairing 
$(\cdot,\cdot)_{\Sol(\cX)}$ restricted 
on $\Sol(\cX)_\Z$ takes values in $\Z$ by Part (b) 
of the assumption and is unimodular by Part (c). 
\end{proposition} 

\begin{proof} 
Because $\hGamma_\cX\cup$ and $(2\pi\iu)^{\deg/2}$ 
are invertible operators over $\C$, 
Part (a) of Assumption \ref{assump:Ktheory} 
implies (i). 
It is easy to check the second statement (ii).  
For (iii), we calculate 
\begin{align*}
(\cZ_K(V_1), \cZ_K(V_2))_{\Sol} 
& = (e^{\pi \iu \rho}\Psi(V_1), 
e^{\pi \iu\mu} \Psi(V_2))_{\rm orb} \quad 
\text{by (\ref{eq:cohfr_property})}  \\ 
= \frac{1}{(2\pi)^n} 
\sum_{v\in \sfT} &
\int_{\cX_v} 
(e^{\pi\iu\rho} \hGamma(T\cX)_{\inv(v)} 
(2\pi\iu )^{\frac{\deg}{2}} \tch(V_1)_v ) \\  
& \qquad \cup 
(e^{\pi\iu(\iota_v - \frac{n}{2} + \frac{\deg}{2})} 
\hGamma(T\cX)_v 
(2\pi\iu )^{\frac{\deg}{2}}\tch(V_2)_{\inv(v)} ) 
\quad \text{by (\ref{eq:Psi})} \\ 
 = \frac{1}{(2\pi)^n} \sum_{v\in \sfT} & 
(2\pi\iu)^{\dim \cX_v} \times \\ 
\int_{\cX_v}  
\prod_{f,i} 
\Gamma(1-\ov{f}+\tfrac{\delta_{v,f,i}}{2\pi\iu}) & 
\Gamma(1-f- \tfrac{\delta_{v,f,i}}{2\pi\iu})
\cdot e^{\frac{\rho}{2}} \tch(V_1)_v \cdot
e^{\pi\iu(\iota_v-\frac{n}{2}+\frac{\deg}{2})}\tch(V_2)_{\inv(v)},  
\end{align*} 
where $\alpha_v$ denotes 
the $v$-component of $\alpha\in H_{\rm orb}^*(\cX)$. 
We used the fact that 
$\mu|_{H^*(\cX_v)} = \iota_v -\frac{n}{2} + \frac{\deg}{2}$ 
in the second step 
and that 
$\int_{\cX_v}((2\pi\iu)^{\frac{\deg}{2}} \alpha) = 
(2\pi\iu)^{\dim \cX_v} \int_{\cX_v} \alpha$ 
in the third step. 
We also used the fact that 
$\{\delta_{\inv(v),f,i} \}_{i}
= \{\delta_{v,\ov{f},i}\}_{i}$, 
where 
\[
\ov{f} := 
\begin{cases} 
1-f & \text{if $0<f<1$,} \\
0   & \text{if $f=0$.} 
\end{cases} 
\]
Using $\Gamma(1-z)\Gamma(z) = \pi/\sin(\pi z)$ 
and $\sum_{f,i}\delta_{v,f,i} =\pr^*\rho|_{\cX_v}$,   
we calculate    
\[
\prod_{f,i} 
\Gamma(1-\ov{f}+\tfrac{\delta_{v,f,i}}{2\pi\iu}) 
\Gamma(1-f-\tfrac{\delta_{v,f,i}}{2\pi\iu})
= (2\pi\iu)^{n-\dim \cX_v} 
e^{-\frac{\rho}{2}} e^{-\pi\iu \iota_v} \tTd(T\cX)_v.  
\]
The conclusion follows from 
the orbifold-Riemann-Roch (\ref{eq:orbifoldRR}). 
\end{proof} 

The lattice $\Sol(\cX)_\Z\subset \Sol(\cX)$ 
defines a $\Z$-local system $F_\Z \to U\times \C^*$ 
underlying the flat vector bundle $(F|_{U\times \C^*}, \nabla)$. 
Because $\Sol(\cX)_\Z$ is invariant under 
the Galois action, the local system $F_\Z\to U\times \C^*$ 
descends to a local system over 
$(U/H^2(\cX,\Z))\times \C^*$.

\begin{remark}
When we consider the \emph{algebraic part} 
of quantum cohomology, 
we can instead use the $K$-group 
of algebraic vector bundles or coherent sheaves
to define an integral structure.  
Let $A^*(\cX)_\C$ denote the Chow ring of $\cX$ over $\C$. 
We set $\HH^*(\cX_v) := \Image(A^*(\cX_v)_\C \to H^*(\cX_v))$ and 
define $\HH_{\rm orb}^*(\cX) := \bigoplus_{v\in \sfT} \HH^*(\cX_v)$. 
Under Assumption \ref{assump:converge},  
the algebraic quantum $D$-module is defined to be 
the holomorphic vector bundle 
\[
\HH^*_{\rm orb}(\cX) \times (U'\times \C) \to (U'\times \C), \quad 
U' = U \cap \HH^*_{\rm orb}(\cX)  
\]
endowed with the restriction of 
the Dubrovin connection to $U'$ and 
the orbifold Poincar\'{e} pairing.  
The Galois action on it is given by an element of $\Pic(\cX)$. 
Here we used the fact that 
the quantum product among classes 
in $\HH_{\rm orb}^*(\cX)$ again belongs to $\HH_{\rm orb}^*(\cX)$; 
this follows from the algebraic construction 
of orbifold Gromov-Witten theory \cite{AGV}.  
When we assume Hodge conjecture for all $\cX_v$, 
each $\HH^*(\cX_v)$ has Poincar\'{e} duality and
the orbifold Poincar\'{e} pairing is non-degenerate 
on $\HH_{\rm orb}^*(\cX)$. 
Definition \ref{def:A-model_int} 
applies to this algebraic quantum $D$-module 
with $K(\cX)$ being the algebraic $K$-group.  
\qed 
\end{remark} 

We introduce the \emph{quantum cohomology central charge} 
of $V\in K(\cX)$ associated to the $\hGamma$-class 
to be the function: 
\begin{equation}
\label{eq:qc_centralcharge}
Z(V)(\tau,z) := c(z) 
\int_{\cX} \cZ_K(V)(\tau,z) 
= c(z) (\unit,\cZ_K(V)(\tau,z))_{\rm orb} 
\end{equation} 
where $c(z) = (2\pi z)^{n/2}/(2\pi\iu)^n$ 
is a normalization factor, \emph{cf.}  
Hosono's central charge formula 
\cite[Definition 2.1]{hosono} for a Calabi-Yau $\cX$ 
given in terms of periods of the mirror.  
For Calabi-Yau 3-folds, the author hopes 
that our $Z(V)$ gives the physics central 
charge of the B-type $D$-brane in the class $V$. 
This plays an important 
role in the Douglas-Bridgeland stability 
on derived categories \cite{douglas, bridgeland}.

\subsection{Givental's symplectic space, 
\seminf VHS and $J$-function} 
\label{subsec:Jfunct} 
Givental's symplectic space 
\cite{givental-quadratic, coates-givental} 
is the loop space on $H_{\rm orb}^*(\cX)$ 
with a loop parameter $z$. 
This is identified with the space of sections 
of $QDM(\cX)$ which are flat only in the $\tau$-direction. 
In the Givental space, $QDM(\cX)$ 
can be realized as moving semi-infinite subspaces. 
This is an example of \emph{semi-infinite variation 
of Hodge structure} (\seminf VHS for short) 
due to Barannikov \cite{barannikov-qpI, barannikov-proj}. 
The $J$-function is the image of the unit section $\unit$ 
in this realization. 
The notion of \seminf VHS will be used only 
in Section \ref{sec:integralperiods}.

\begin{definition} 
Let $\cO(\C^*)$ denote the space of holomorphic functions 
on $\C^*$ with the co-ordinate $z$. 
The \emph{Givental space} $\cH$ 
is defined to be the free $\cO(\C^*)$-module: 
\begin{equation}
\label{eq:Giventalsp}
\cH = H_{\rm orb}^*(\cX) \otimes \cO(\C^*) 
\end{equation}
endowed with the pairing 
$(\cdot,\cdot)_{\cH} \colon \cH \times \cH \to \cO(\C^*)$ 
\begin{equation}
\label{eq:pairing_cH}
(\alpha(z), \beta(z))_{\cH} := (\alpha(-z), \beta(z))_{\rm orb}.    
\end{equation} 
and the symplectic form 
$\Omega(\alpha(z),\beta(z)) 
= \Res_{z=0} dz (\alpha(z),\beta(z))_{\cH}$. 
Using the fundamental solution $L(\tau,z)$, we 
identify $\cH$ with the space of sections of 
$QDM(\cX)$ which are flat in the $\tau$-direction.  
\begin{equation}
\label{eq:Giv_flat}
\cH\ni \alpha \longmapsto L(\tau,z)\alpha 
\in \Gamma(U\times \C^*,\cO(F)).  
\end{equation} 
Note that under this identification, 
$(\cdot,\cdot)_{\cH}$ corresponds to  
$(\cdot,\cdot)_F$ by (\ref{eq:unitarity_L}). 
The Galois action on flat sections 
(\ref{eq:Galois_Sol}) induces 
a map $G^\cH(\xi)\colon \cH\to \cH$: 
\begin{equation}
\label{eq:GaloisH}
G^{\cH}(\xi)(\tau_0 \oplus \bigoplus_{v\in \sfT'} \tau_v) 
= e^{-2\pi \iu \xi_0/z}\tau_0 \oplus \bigoplus_{v\in \sfT'} 
e^{-2\pi \iu \xi_0/z} e^{2\pi \iu f_v(\xi)} \tau_v,  
\end{equation} 
by (\ref{eq:Galois_L}). 
Here we used the decomposition 
$\cH^{\cX} = \bigoplus_{v\in \sfT} 
H^*( {\cX_v})\otimes \cO(\C^*)$. 
\qed 
\end{definition} 

We introduce the \seminf VHS 
associated to quantum cohomology.   
Let $\pi\colon U\times \C \to U$ be the 
natural projection. 
Under the identification (\ref{eq:Giv_flat}), 
the fiber $(\pi_*\cO(F))_\tau$ at $\tau\in U$ 
is identified with 
the semi-infinite subspace $\F_\tau$ of $\cH$: 
\[
\F_\tau := \J_\tau  
(H_{\rm orb}^*(\cX)\otimes \cO(\C)), 
\quad 
\J_\tau  := L(\tau,z)^{-1}. 
\] 
We call $\F_\tau$ the \emph{semi-infinite Hodge structure}. 
This satisfies the following properties: 
\begin{itemize}
\item $X \F_\tau  \subset z^{-1} \F_\tau$ 
for a tangent vector $X\in T_\tau U$;  
\item $\F_\tau$ is isotropic with respect 
to $\Omega$, \emph{i.e.} 
$(\F_\tau,\F_\tau)_{\cH} \subset \cO(\C)$;   
\item $(2 E + \nabla_{z\partial_z}) \F_\tau \subset \F_\tau$.   
\end{itemize} 
Here we regard $\tau \mapsto \F_\tau$ 
as a holomorphic map from $U$ to the Segal-Wilson Grassmannian 
(see \emph{e.g.} \cite{pressley-segal}).  
Also $\nabla_{z\partial_z}$ denotes 
the operator on $\cH$ induced from $\nabla_{z\partial_z}$. 
We call the family $\tau\mapsto \F_\tau$ 
\emph{(a moving subspace realization of) a \seminf VHS}. 
The first property is an analogue of 
Griffith transversality and the second
is the Hodge-Riemann bilinear relation. 
We refer the reader to \cite[Section 2]{CIT:I}, 
\cite[Section 2]{iritani-realint-preprint} 
for the details. 

\begin{remark} 
The \seminf VHS defines a 
Lagrangian cone $\cL$ in $\cH$: 
\begin{equation}
\label{eq:Giventalcone}
\cL := \bigcup_{\tau\in U} z \F_\tau.  
\end{equation} 
This plays an important role in Givental's theory. 
The cone $\cL$ can be written as the graph 
of the differential $d\cF_0$ of the genus 
zero descendant potential $\cF_0$ (with a dilaton shift). 
We refer the reader to \cite{coates-givental} 
for this connection.  
\end{remark} 

Using the fact that $L(\tau,z)^{-1}$ 
is the adjoint of $L(\tau,-z)$ with respect to 
the orbifold Poincar\'{e} pairing (see (\ref{eq:unitarity_L})),  
we can calculate the embedding $\J_\tau = 
L(\tau,z)^{-1} \colon 
(\pi_*\cO(F))_\tau \hookrightarrow \cH$ 
explicitly as follows: 
\begin{align}
\label{eq:Linv}
\J_\tau \alpha = e^{\tau_{0,2}/z}
\biggl(\alpha + 
\sum_{\substack{(d,l)\neq (0,0) \\ d\in \Eff_\cX}} \sum_{i=1}^N 
\frac{1}{l!} 
\corr{\alpha,\tau',\dots,\tau', 
\frac{\phi_i }{z-\psi}}_{0,l+2,d}^\cX 
e^{\pair{\tau_{0,2}}{d}} \phi^i\biggr).   
\end{align} 

\begin{definition}
The \emph{$J$-function} 
\cite{givental-mirrorthm-toric, cox-katz, coates-givental} 
is the image of the unit section $\unit$ under the embedding 
$\J_\tau \colon (\pi_*\cO(F))_\tau \hookrightarrow \cH$: 
$J(\tau,z) := \J_\tau \unit = L(\tau,z)^{-1} \unit$. 
Because the unit section $\unit$ is invariant 
under the Galois action, we have 
\begin{equation}
\label{eq:Galois_J}
J(G(\xi)\tau,z) = G^{\cH}(\xi) J(\tau,z)    
\end{equation} 
which follows from (\ref{eq:Galois_L}). 
\qed  
\end{definition} 

The $J$-function is the unit section $\unit$ 
expressed in the $\tau$-flat frame $L(\tau,z)$. 
The \emph{$H$-function} $H_K(\tau,z)$ is defined 
to be the $K(\cX)\otimes \C$-valued function which 
expresses $\unit$ in terms of the $K$-group 
framing (\ref{eq:Psi}):     
\begin{equation}
\label{eq:H-funct}
H_K(\tau,z) := c(e^{-\pi\iu} z)\cdot  
\Psi^{-1}(z^{-\rho}z^{\mu} L(\tau,z)^{-1} \unit),      
\end{equation} 
\emph{i.e.} $c(e^{-\pi\iu} z) \unit = \cZ_K(H_K(\tau,z))(\tau,z)$. 
Here $c(e^{-\pi\iu}z) := (2\pi z)^{n/2}/(-2\pi)^n$ 
is a normalization factor. 
We also use $H^*(I\cX)$-valued function 
$H(\tau,z) := \tch(H_K(\tau,z))$. 
The quantum cohomology central charge 
(\ref{eq:qc_centralcharge}) can be written as 
(\emph{cf.} \cite[Eqn. (2.3)]{hosono}):  
\begin{equation}
\label{eq:cc_byH} 
Z(V)(\tau,z) 
=\chi(H_K(\tau,e^{\pi\iu}z)\otimes V^\vee)  
= \int_{I\cX} H(\tau,e^{\pi\iu}z) \cup \tch(V^\vee) 
\cup \tTd(T\cX).  
\end{equation} 
\begin{proof} 
We have 
$Z(V)(\tau,z)  
=(\cZ_K(H_K(\tau,e^{\pi\iu}z))(\tau,e^{\pi\iu}z), 
\cZ_K(V)(\tau,z))_{\rm orb}$. 
The formulas follows from this,  
Proposition \ref{prop:A-model_int}, (iii)  
and orbifold Riemann-Roch (\ref{eq:orbifoldRR}). 
\end{proof} 

\section{Landau-Ginzburg mirror of toric orbifolds}
\label{sec:LGmodel}

In this section, 
we describe the Landau-Ginzburg (LG) models 
which are mirror to compact toric orbifolds. 
The LG mirrors 
for toric manifolds have been proposed by Givental 
\cite{givental-ICM, givental-mirrorthm-toric} 
and Hori-Vafa \cite{hori-vafa} 
and they are easily adapted to the case of toric orbifolds. 
We also construct a meromorphic flat connection 
(B-model $D$-module) over 
the product of $\C$ with 
the parameter space $\cM$ of the LG models.  
The B-model $D$-module has been studied in singularity 
theory as the Brieskorn lattice. 
We give an analytical construction based on 
oscillatory integrals. 
See Sabbah \cite{sabbah-hypergeometric} 
for an algebraic construction (for a tame function 
on an algebraic variety) 
using the Fourier-Laplace transform 
of the algebraic Gau\ss-Manin system 
(see also \cite{saitoM, douai-sabbah-I}).

\subsection{Toric orbifolds} 
\label{subsec:toricorbifolds} 
To fix the notation, we give the definition 
of toric orbifolds and collect several facts.  
By a toric orbifold, we mean 
a toric Deligne-Mumford stack in the sense of 
Borisov-Chen-Smith \cite{borisov-chen-smith}. 
We only deal with a compact toric orbifold 
with a projective coarse moduli space 
and define a toric orbifold as a quotient of 
$\C^m$ by an algebraic torus $\T \cong (\C^*)^r$. 
The basic references for toric varieties (orbifolds)  
are made to \cite{oda, fulton, audin, borisov-chen-smith}. 

\subsubsection{Definition}  
\label{subsubsec:def_toricorbifolds}
We begin with the following data: 
\begin{itemize}
\item an $r$-dimensional algebraic torus $\T\cong (\C^*)^r$; we set 
$\bL:=\Hom(\C^*,\T)$; 
\item $m$ elements $D_1,\dots,D_m \in \bL^\vee = \Hom(\T,\C^*)$ 
such that $\bL^\vee\otimes \R= \sum_{i=1}^m \R D_i$; 
\item a vector $\eta\in \bL^\vee \otimes \R$.  
\end{itemize} 
The elements $D_1,\dots, D_m$ define a 
homomorphism $\T\rightarrow (\C^*)^m$. 
Let $\T$ act on $\C^m$ via this homomorphism. 
The vector $\eta$ defines a stability condition 
of this torus action. 
Set 
\[
\cA := \{ I\subset \{1,\dots,m\} \;;\; 
\sum_{i\in I} \R_{>0} D_i \ni \eta\}.   
\]
A toric orbifold $\cX$ is defined to be the quotient stack   
\[
\cX = [\cU_\eta/\T], \quad 
\cU_\eta := \C^m \setminus \bigcup_{I\notin \cA} \C^I,  
\]
where
$\C^I := \{(z_1,\dots,z_m)\in \C^m \;;\; z_i= 0 \text{ for } i\notin I\}$.
Under the following conditions, $\cX$ is a smooth Deligne-Mumford 
stack with a projective coarse moduli space:   
\begin{itemize}
\item[(A)] $\{1,\dots,m\} \in \cA$. 
\item[(B)] $\sum_{i\in I} \R D_i = \bL^\vee \otimes \R$ for $I\in \cA$. 
\item[(C)] $\{(c_1,\dots,c_m)\in \R_{\ge 0}^m \;;\; 
\sum_{i=1}^m c_i D_i =0\} = \{0\}$. 
\end{itemize} 
The conditions (A), (B) and (C) ensure 
that $\cX$ is non-empty, that the stabilizer is finite 
and that $\cX$ is compact respectively. 
The generic stabilizer of $\cX$ is given by the kernel of 
$\T \to (\C^*)^m$ and $\dim_\C\cX=n:=m-r$.

We can also construct $\cX$ as a symplectic quotient as follows 
(see also \cite{audin}).  
Let $\T_\R$ denote the maximal compact subgroup of $\T$ 
isomorphic to $(S^1)^r$. 
Let $\mathfrak{h} \colon \C^m \to \bL^\vee\otimes\R$ be the moment 
map for the $\T_\R$-action on $\C^m$: 
\[
\mathfrak{h}(z_1,\dots,z_m) = \sum_{i=1}^m |z_i|^2 D_i.  
\]
The $\T_\R$-action on the level set $\mathfrak{h}^{-1}(\eta)$ 
has only finite stabilizers and we have an isomorphism of 
symplectic orbifolds: 
\begin{equation}
\label{eq:X_symplecticquot}
\cX \cong \mathfrak{h}^{-1}(\eta) / \T_\R.  
\end{equation} 

By renumbering the indices if necessary, we can assume that 
\[
\{1,\dots,m\} \setminus \{i\} \in \cA \quad \text{ if and only if } 
\quad 1\le i\le m' 
\]
where $m'$ is less than or equal to $m$. 
We can easily check that 
$I\supset \{m'+1,\dots,m\}$ for any $I\in \cA$ 
and $D_{m'+1},\dots,D_m$ are linearly independent 
over $\R$. 
The elements $D_1,\dots,D_m$ define the following exact sequence 
\begin{equation}
\label{eq:exactsequence_toric}
\begin{CD}
0 @>>> \bL @>{(D_1,\dots,D_m)}>> \Z^m @>{\beta}>> \bN @>>> 0,  
\end{CD}  
\end{equation} 
where $\bN$ is a finitely generated abelian group. 
By the long exact sequence associated with 
the functor $\Tor_\bullet(-,\C^*)$, 
we find that the torsion part 
$\bN_{\rm tor}=\Tor_1(\bN,\C^*)$ of $\bN$ is isomorphic to  
the generic stabilizer $\Ker(\T\to (\C^*)^m)$. 
The free part 
$\bN_{\rm free}=\bN/\bN_{\rm tor}$ is of rank $n=\dim_\C\cX$. 
Let $b_1,\dots,b_m$ be the images in $\bN$ 
of the standard basis of $\Z^m$ under $\beta$. 
The {\it stacky fan} of $\cX$, in the sense of 
Borisov-Chen-Smith \cite{borisov-chen-smith}, is given by 
the following data: 
\begin{itemize}
\item vectors $b_1,\dots, b_{m'}$ in $\bN$;   
\item a complete simplicial fan $\Sigma$ in $\bN\otimes \R$ such that 

(i) the set of one dimensional cones 
is $\{\R_{\ge 0} b_1, \dots, \R_{\ge 0} b_{m'}\}$; 

(ii) 
$\sigma_I = \sum_{i\notin I} \R_{\ge 0} b_i$ defines a cone 
of $\Sigma$ if and only if $I \in \cA$. 
\end{itemize} 
The toric variety defined by the fan $\Sigma$ is 
the coarse moduli space of $\cX$. 
The conditions (B) and (C) correspond to that $\Sigma$ is simplicial 
and that $\Sigma$ is complete, 
\emph{i.e.} the union of all cones in $\Sigma$ 
is $\bN\otimes \R$. 
An element of $\cA$ may be referred to 
as an ``anticone".

\begin{remark}
Borisov-Chen-Smith \cite{borisov-chen-smith} 
defined a toric Deligne-Mumford stack 
starting from data of a stacky fan. 
Our construction can give every toric Deligne-Mumford stack 
in their sense which has a projective coarse moduli space. 
Note that the vectors $b_{m'+1},\dots, b_{m}$ do not appear 
as data of a stacky fan. 
The stacky fan together with these extra vectors 
gives an \emph{extended stacky fan} in the sense of 
Jiang \cite{jiang-toricstackbundle}.  
When we start from a stacky fan, our initial 
data can be given as the kernel of the map $\beta$ 
by \emph{choosing} extra vectors 
$b_{m'+1},\dots, b_m \in \bN$ such that  
$\beta$ is surjective. 
These redundant data allows us to 
define $\cX$ as a quotient 
by a \emph{connected} torus $\T$. 
\end{remark}

\subsubsection{K\"{a}hler cone and a choice of a nef basis}
\label{subsubsec:KC_nefbasis}
Since every element of $\cA$ contains $\{m'+1,\dots,m\}$, 
it is convenient to put 
\[
\cA' = \{ I'\subset \{1,\dots,m'\}\;
;\; I'\cup\{m'+1,\dots,m\}\in \cA\}. 
\]
We can easily see that $\cU_\eta$ factors as 
\[
\cU_\eta = \cU'_\eta \times (\C^*)^{m-m'}, \quad 
\cU'_\eta = \C^{m'} \setminus \bigcup_{I'\notin \cA'} \C^{I'}. 
\]
Thus we can write 
\[
\cX = [\cU'_\eta/\G], \quad 
\G:=\Ker(\T\to (\C^*)^m \to (\C^*)^{\{m'+1,\dots,m\}}).
\] 
Note that $\G$ is isomorphic to $(\C^*)^{r'}$ 
times a finite abelian group 
for $r':=r-(m-m')$.  
Every character $\xi \colon \G \to \C^*$ of $\G$ 
defines an orbifold line bundle 
$L_\xi:=\cU'_\eta\times_{\G,\xi} \C \to \cX$. 
Under this correspondence between $\xi$ and $L_\xi$, 
the Picard group $\Pic(\cX)$ is identified 
with the character group $\Hom(\G,\C^*)$ 
and also with $H^2(\cX,\Z)$ (via $c_1$): 
\[
\Pic(\cX) \cong \Hom(\G,\C^*) 
\cong \bL^\vee/\textstyle\sum_{i=m'+1}^m \Z D_i 
\cong H^2(\cX,\Z).  
\]
The image $\ov{D}_i$ of $D_i$ in $H^2(\cX,\R)$ 
is the Poincar\'{e} dual of 
the toric divisor $\{z_i=0\}\subset \cX$ for 
$1\le i\le m'$. 
Over rational numbers, we have 
\begin{align*}
H^2(\cX,\Q) & \cong  
\bL^\vee\otimes \Q/\textstyle \sum_{i=m'+1}^m \Q D_i,  \\
H_2(\cX,\Q) & \cong
\Ker((D_{m'+1},\dots, D_m)\colon \bL\otimes \Q \to \Q^{m-m'}) 
\subset \bL\otimes \Q. 
\end{align*} 
Now we introduce a canonical splitting (over $\Q$) 
of the surjection $\bL^\vee \otimes \Q \to H^2(\cX,\Q)$.   
For $m'<j\le m$,  $b_j$ is contained in some cone in $\Sigma$ 
since $\Sigma$ is complete. Namely, 
\begin{equation}
\label{eq:bjcontainedinacone}
b_j = \sum_{i\notin I_j} c_{ji} b_i, \quad \text{ in } \bN\otimes \Q, 
\quad c_{ji}\ge 0, \quad \exists I_j \in \cA, 
\end{equation} 
where $I_j$ is the ``anticone" of  
the cone containing $b_j$.  
By the exact sequence 
(\ref{eq:exactsequence_toric}) tensored with $\Q$, 
we can find $D^\vee_j \in \bL\otimes \Q$ such that 
\[
\pair{D_i}{D^\vee_j} = 
\begin{cases} 
1 \quad & i=j \\
-c_{ji} \quad & i\notin I_j \\
0 \quad & i\in I_j\setminus \{j\}.  
\end{cases}
\]
Note that $D^\vee_j$ is uniquely determined by these conditions. 
These vectors $D^\vee_j$ define a decomposition 
\begin{align}
\label{eq:decomp_LQ} 
\bL^\vee \otimes \Q  &= 
\Ker((D_{m'+1}^\vee,\dots,D_{m}^\vee)\colon 
\bL^\vee\otimes \Q\to \Q^{m-m'}) 
\oplus \bigoplus_{j=m'+1}^m \Q D_j.  
\end{align} 
The first factor $\Ker(D_{m'+1}^\vee,\dots,D_{m}^\vee)$ 
is identified with $H^2(\cX,\Q)$ under the surjection 
$\bL^\vee\otimes \Q \to H^2(\cX,\Q)$. 
Via this decomposition, we henceforth 
regard $H^2(\cX,\Q)$ as a subspace of $\bL^\vee\otimes \Q$. 
We define an {\it extended K\"{a}hler cone} $\tC_\cX$ as  
\[
\tC_\cX  = \bigcap_{I\in \cA} (\sum_{i\in I} \R_{>0} D_i) 
\subset \bL^\vee \otimes \R. 
\] 
Then $\eta\in \tC_\cX$ and the image of $\eta$ in $H^2(\cX,\R)$ 
is the class of the reduced symplectic form. 
The set $\tC_\cX$ is the connected component of the set of 
regular values of the moment map 
$\mathfrak{h}\colon \C^m \to \bL^\vee\otimes \R$, 
which contains $\eta$. 
The extended K\"{a}hler cone depends not only on $\cX$ 
but also on the choice of our initial data. 
The genuine {\it K\"{a}hler cone} $C_\cX$ of $\cX$ is 
the image of $\tC_\cX$ 
under $\bL^\vee \otimes \R \to H^2(\cX,\R)$: 
\[
C_\cX = \bigcap_{I'\in \cA'} (\sum_{i\in I'} \R_{>0} \ov{D}_i) \subset 
H^2(\cX,\R) = H^{1,1}(\cX,\R) 
\]
where $\ov{D}_i$ is the image of $D_i$ in $H^2(\cX,\R)$.  
The next lemma means that the extended K\"{a}hler cone also ``splits". 
\begin{lemma}
\label{lem:decomp_extendedKaehler} 
$\tC_\cX = C_\cX + \sum_{j=m'+1}^m \R_{>0} D_j$ 
in $\bL^\vee \otimes \R \cong 
H^2(\cX,\R)\oplus \bigoplus_{j=m'+1}^m \R D_j$. 
\end{lemma} 
\begin{proof}
First note that for $1\le i\le m'$, 
$\ov{D}_i = D_i + \sum_{j>m'} c_{ji} D_j$, 
where $c_{ji} = -\pair{D_i}{D_j^\vee}\ge 0$. 
Take $I'\in \cA'$ and put $I=I'\cup \{m'+1,\dots,m\}$. 
It is easy to check that 
\[
\sum_{i\in I'} \R_{>0} \ov{D}_i + \sum_{j=m'+1}^{m} \R_{>0}D_j 
 = \sum_{k\in I} \R_{>0} D_k \cap \bigcap_{j=m'+1}^m 
\{D_j^\vee>0\},  
\]
where we regard $D_j^\vee$ 
as a linear function on $\bL^\vee\otimes \R$. 
Thus $C_\cX+\sum_{j>m'} \R_{>0} D_j = \tC_\cX \cap 
\bigcap_{j=m'+1}^m \{D_j^\vee>0\}$. 
For $j>m'$, take $I_j\in \cA$ appearing in 
(\ref{eq:bjcontainedinacone}). 
Then $\tC_\cX \subset \sum_{k\in I_j} \R_{>0} D_k\subset 
\{D_j^\vee>0\}$. The conclusion follows. 
\end{proof} 

We choose an integral basis 
$\{p_1,\dots,p_r\}$ 
of $\bL^\vee$ such that 
$p_a$ is in the closure $\cl(\tC_\cX)$ of 
$\tC_\cX$ for all $a$ 
and $p_{r'+1},\dots,p_r$ are in $\sum_{i=m'+1}^m \R_{\ge 0} D_i$. 
Since the decomposition (\ref{eq:decomp_LQ}) 
is defined over $\Q$, it is not always possible to choose 
$p_1,\dots,p_{r'}$ from $\cl(C_\cX)$. 
The images $\ov{p}_1,\dots, \ov{p}_{r'}$ 
of $p_1,\dots,p_{r'}$ 
in $H^{2}(\cX,\R)$ are nef and those of 
$p_{r'+1},\dots, p_r$ are zero. 
Define a matrix $(\sfm_{ia})$ by  
\begin{equation}
\label{eq:Dprel}
D_i = \sum_{a=1}^r \sfm_{ia} p_a, 
\quad \sfm_{ia} \in \Z.  
\end{equation} 
Then the class $\ov{D}_i$ of the toric divisor 
$\{z_i=0\}$ is given by 
\begin{equation}
\label{eq:ovDprel}
\ov{D}_i = \sum_{a=1}^{r'} \sfm_{ia} \ov{p}_a.  
\end{equation} 
Then $\ov{D}_j =0$ for $m'< j\le m$. 

\subsubsection{Inertia components and orbifold cohomology}
We introduce subsets $\K$, $\K_{\rm eff}$ of $\bL\otimes \Q$ by 
\begin{align*}
\K &= \{d\in \bL\otimes \Q \;;\; \{i\in \{1,\dots,m\} 
\;;\; \pair{D_i}{d}\in \Z \}\in \cA\}, \\
\K_{\rm eff} & = \{ d\in \bL\otimes \Q \; ; \; 
\{i \in \{1,\dots,m\}  \; ; 
\; \pair{D_i}{d} \in \Z_{\ge 0} \} \in \cA\}.    
\end{align*} 
The sets $\K$ and $\K_{\rm eff}$ are not closed 
under addition, but $\bL$ acts on $\K$. 
The set $\K_{\rm eff}\cap H_2(\cX,\R)$ consists 
of classes of stable maps from 
$\Proj(1,a)$ to $\cX$ for some $a\in \N$. 
It follows from the definition that the 
$\K_{\rm eff}$ pairs with $\tC_\cX$ positively. 
The index set $\sfT$ of components of the inertia stack 
$I\cX$ is given by $\Boxop$ \cite{borisov-chen-smith}: 
\[
\Boxop := \Big \{ v\in \bN \;;\; v = \sum_{k\notin I} c_k b_k \text{ in } 
\bN \otimes \Q, \ c_k\in [0,1), \ I \in \cA  \Big \}. 
\]
For a real number $r$, let $\ceil{r}$, $\floor{r}$ 
and $\{r\}$ denote the ceiling, floor and fractional part 
of $r$ respectively. 
For $d\in \K$, we define $v(d)\in \Boxop$ by 
\[
v(d) := \sum_{i=1}^m \ceil{\pair{D_i}{d}} b_i \in \bN. 
\]
Note that $v(d)$ belongs to $\Boxop$ because 
\[
v(d) = \sum_{i=1}^m (\{-\pair{D_i}{d}\} + \pair{D_i}{d}) b_i 
= \sum_{i=1}^m \{-\pair{D_i}{d}\} b_i \quad \text{in $\bN\otimes \Q$} 
\]
by the exact sequence (\ref{eq:exactsequence_toric}).  
This map $d\mapsto v(d)$ factors through $\K \to \K/\bL$ 
and identifies $\K/\bL$ with $\Boxop$. 
The corresponding inertia component\footnote
{When $d\in \K_{\rm eff}\cap H_2(\cX,\Q)$, 
the evaluation image of a stable map 
$\Proj(1,a)\to \cX$ of degree $d$ 
at the stacky marked point $\Proj(a)\in \Proj(1,a)$ 
lies in $\cX_{\inv(v(d))}$.} $\cX_{v(d)}$ 
is defined by 
\[
\cX_{v(d)} := \{[z_1,\dots,z_m] \in \cX \; ;\; z_i = 0 \text{ if } 
\pair{D_i}{d} \notin \Z \}.  
\]
The stabilizer along $\cX_{v(d)}$ is defined to be   
$\exp(-2\pi\sqrt{-1}d) \in 
\bL\otimes \C^*\cong \T$, which acts on $\C^m$ by 
\[
(e^{-2\pi\iu\pair{D_1}{d}}, \cdots, e^{-2\pi\iu\pair{D_m}{d}}).  
\] 
It is easy to check that $\cX_{v(d)}$ depends only on 
the element $v(d) \in \Boxop$. 
The age of $\cX_{v(d)}$ is given by 
\begin{equation} 
\label{eq:ageofX_v}
\iota_{v(d)} = \sum_{i=1}^m \{ - \pair{D_i}{d} \} 
= \sum_{i=1}^{m'} \{ -\pair{D_i}{d} \}.   
\end{equation} 
The inertia stack and orbifold cohomology are given by 
\begin{equation}
\label{eq:toricorbcoh} 
I\cX = \bigsqcup_{v\in \Boxop} \cX_v, \quad 
H_{\rm orb}^i(\cX) = \bigoplus_{v\in \Boxop} H^{i - 2 \iota_v} (\cX_v).  
\end{equation} 
Denote by $\unit_v$ the unit class of $H^*(\cX_v)$. 
Each inertia component $\cX_v$ is again a toric orbifold 
and its cohomology ring is generated by the 
degree two classes $\ov{p}_1,\dots,\ov{p}_{r'}$: 
\begin{align}
\label{eq:coh_presentation} 
\begin{split} 
&H^*(\cX_{v(d)}) = 
\C[\ov{p}_1,\dots,\ov{p}_{r'}] \unit_v 
\cong 
\C[\ov{p}_1,\dots,\ov{p}_{r'}]/\frJ_{v(d)}, \\ 
& \text{where} \quad 
\frJ_{v(d)} 
:= \left\langle \textstyle\prod_{i\in I} \ov{D}_i\;;\;  
\{1\le i \le m\;;\; \pair{D_i}{d}\in \Z\} 
\setminus I \notin \cA \right\rangle.  
\end{split} 
\end{align} 
Here we regard 
$\ov{D}_i$ as a linear form  (\ref{eq:ovDprel}) 
in $\ov{p}_a$. 
For $\xi\in \bL^\vee$, let $[\xi]$ be the image of $\xi$ 
in $\bL^\vee/\sum_{j=m'+1}^m \Z D_j \cong H^2(\cX,\Z)$. 
The age $f_v(\xi) = f_v([\xi]) \in [0,1)$ of the line bundle 
$L_\xi$ (see Section \ref{subsec:QuantumDmod}) is given by 
\begin{equation}
\label{eq:toric_fv}
f_{v(d)}(\xi) = \{-\pair{\xi}{d}\}, \quad d\in \K. 
\end{equation}

\subsubsection{Weak Fano condition} 
\label{subsubsec:weakFano} 
The first Chern class $\rho=c_1(T\cX)\in H^2(\cX,\Q)$ of 
$\cX$ is the image of the vector $\hrho \in \bL^\vee$: 
\[
\hrho := D_1 +\cdots + D_m = \sum_{a=1}^r \rho_a p_a, \quad 
\rho_a := \sum_{i=1}^m \sfm_{ia}.  
\] 
We call $\cX$ \emph{weak Fano} if $\rho$ is in the closure $\cl(C_\cX)$ 
of the K\"{a}hler cone $C_\cX$. 
We shall need a little stronger condition 
$\hrho\in \cl(\tC_\cX)$. 
This condition $\hrho \in\cl(\tC_\cX)$ 
depends not only on $\cX$ but also 
on our initial data in Section 
\ref{subsubsec:def_toricorbifolds}, \emph{i.e.} 
the choice of the vectors $b_{m'+1},\dots,b_m\in \bN$.  
\begin{lemma}
\label{lem:weakFano} 
We have $\hrho \in \cl(\tC_\cX)$ if and only if 
$\rho \in \cl(C_\cX)$ (i.e. $\cX$ is weak Fano) 
and $\age(b_j) := \sum_{i\notin I_j} c_{ji} \le 1$ for all $j>m'$.
If $b_j\in \Boxop$, $\age(b_j)$ coincides 
with $\iota_{b_j}$ in (\ref{eq:ageofX_v});  
See (\ref{eq:bjcontainedinacone}) for the definition 
of $I_j$ and $c_{ji}$.  
\end{lemma} 
\begin{proof}
From $\ov{D}_i = D_i + \sum_{j>m'} c_{ji}D_j$, we have 
\[
\hrho = \rho + \sum_{j>m'} (1-\age (b_j)) D_j  
\]
The conclusion follows from Lemma \ref{lem:decomp_extendedKaehler}. 
\end{proof} 

When $\hrho\in \cl(\tC_\cX)$, we can choose a basis 
$p_1,\dots,p_r\in \cl(\tC_\cX)$ so that $\hrho$ is in the 
cone generated by $p_a$'s. 
Thus \emph{in this case, we can assume 
$\rho_a\ge 0$ without loss of generality}. 

\begin{remark} 
The condition $\hrho \in \cl(\tC_\cX)$ depends on 
the choice of our initial data. 
This can be achieved  
if $\cX$ is weak Fano and if in addition 
its stacky fan satisfies 
\[
\{v\in \Boxop \;;\; \age(v)\le 1\} \cup \{b_1,\dots,b_{m'}\} 
\text{ generates $\bN$ over $\Z$.} 
\]
If this holds, we can choose $b_{m'+1},\dots,b_m \in \Boxop$ 
so that $\{b_1,\dots,b_m\}$ generates $\bN$ and 
$\age(b_j)\le 1$ for $m'<j\le m$.    
Then the exact sequence (\ref{eq:exactsequence_toric}) 
determines $D_1,\dots,D_m$ and 
$\hrho=D_1+\dots+D_m\in \cl(\tC_\cX)$ holds. 
If $\cX$ is simply-connected 
in the sense of orbifold ($\pi_1^{\rm orb}(\cX)=\unit$), 
$\bN$ is generated by $b_1,\dots,b_{m'}$. 
\end{remark} 

\begin{remark} 
\label{rem:meaning_Dj}
The vectors $D_j$, $m' < j \le m$ in $\bL^\vee$ 
correspond to the following elements 
in the twisted sector:  
\begin{equation}
\label{eq:frD}
\frD_j = \prod_{i\notin I_j} 
\ov{D}_i^{\floor{c_{ji}}} \unit_{v(D^\vee_j)} 
\in H_{\rm orb}^{*}(\cX), \quad \text{where } 
v(D^\vee_j) = b_j + \sum_{i\notin I_j} \ceil{-c_{ji}} b_i. 
\end{equation} 
This correspondence can be seen from 
the expansion (\ref{eq:mirrormap_exp}) 
of the mirror map $\tau(q)$ below. 
We have $\frD_j=\unit_{b_j}$ when $b_j\in \Boxop$. 
Therefore, if $\hrho\in \cl(\tC_\cX)$ and 
$b_{m'+1},\dots,b_{m}$ are mutually different elements in $\Boxop$, 
we can identify $\bL^\vee\otimes \C$ 
with the subspace $H^2(\cX) \oplus \bigoplus_{j>m'} H^0(\cX_{b_j})$ 
of $H^{\le 2}_{\rm orb}(\cX)$. 
\end{remark} 

\subsection{Landau-Ginzburg model} 
\label{subsec:LGmodel} 
We introduce the Landau-Ginzburg (LG) model mirror 
to compact toric orbifolds. 
We use the notation from 
Section \ref{subsec:toricorbifolds}. 

\subsubsection{Definition} 
\label{subsubsec:LGmodel_def} 
By applying the exact functor $\Hom(-,\C^*)$ to 
the short exact sequence (\ref{eq:exactsequence_toric}), 
we have 
\begin{equation}
\label{eq:fibration_LG}
\begin{CD}
\unit @>>> \Hom(\bN,\C^*) 
@>>> Y:= (\C^*)^m @>{\pr}>> \cM:= \Hom(\bL,\C^*) @>>> \unit.  
\end{CD} 
\end{equation} 
The {\it Landau-Ginzburg model} (LG model for short) 
associated to a toric orbifold 
is the family $\pr \colon Y \to \cM$ of affine varieties 
given by the third arrow and a fiberwise Laurent polynomial 
$W\colon Y \to \C$, called potential, given by 
\[
W=w_1+\cdots+w_m 
\] 
where $w_1,\dots,w_m$ are the standard 
$\C^*$-valued co-ordinates on $Y=(\C^*)^m$. 
Roughly speaking, 
the base space $\cM = \bL^\vee \otimes \C^*$ 
corresponds to the extended (and complexified) 
K\"{a}hler moduli space $H^{\le 2}_{\rm orb}(\cX)$ 
of $\cX$ under mirror symmetry (see Remark \ref{rem:meaning_Dj}). 
The basis of $\bL$ dual to $p_1,\dots,p_r\in \bL^\vee$ 
in Section \ref{subsubsec:KC_nefbasis} defines 
$\C^*$-valued co-ordinates 
$q_1,\dots,q_r$ on $\cM=\Hom(\bL,\C^*)$. 
Then the projection is given by (see (\ref{eq:Dprel})) 
\begin{equation}
\label{eq:fibration_LG_formula}
\pr(w_1,\dots,w_m) 
= (q_1,\dots,q_r), \quad q_a = \prod_{i=1}^m w_i^{\sfm_{ia}}.  
\end{equation}
Let $Y_q := \pr^{-1}(q)$ be the fiber at $q\in \cM$ 
and set $W_q:=W|_{Y_q}$. 
Note that $Y_q$ has $|\bN_{\rm tor}|$ connected components and 
each connected component is isomorphic to 
$\Hom(\bN_{\rm free},\C^*) \cong (\C^*)^n$. 
Let $e_1,\dots,e_n$ be an arbitrary basis of $\bN_{\rm free}$ and 
$y_1,\dots,y_n$ be the corresponding 
$\C^*$-valued co-ordinate on $\Hom(\bN_{\rm free},\C^*)$. 
We choose a splitting of the exact sequence dual to  
(\ref{eq:exactsequence_toric}) over rational numbers. 
Namely, we take a matrix $(\ell_{ia})_{1\le i\le m, 1\le a\le r}$ with 
$\ell_{ia}\in \Q$ such that $p_a = \sum_{i=1}^m D_i \ell_{ia}$. 
This splitting defines a multi-valued section of $\pr:Y\to \cM$ 
and identifies $Y_q$ with $\Hom(\bN,\C^*)$. 
Under this identification, $y_1,\dots,y_n$ give 
co-ordinates on each connected component of $Y_q$ and we have 
\begin{equation}
\label{eq:W_q}
W|_{Y_q} = W_q = q^{\ell_1} y^{b_1} + \cdots + q^{\ell_m} y^{b_m}, \quad
q^{\ell_i} = \prod_{a=1}^r q_a^{\ell_{ia}}, \quad  
y^{b_i} = \prod_{j=1}^n y_j^{b_{ij}},  
\end{equation} 
where $b_i = \sum_{j=1}^n b_{ij} e_j$ in $\bN_{\rm free}$. 
Here, the choice of the branches of fractional powers of $q_a$ 
appearing in $q^{\ell_i}$ depends on a connected component of $Y_q$. 

\subsubsection{Kouchnirenko's condition} 
When constructing the B-model $D$-module, 
we shall need to restrict the parameter 
$q\in \cM$ to some Zariski open subset $\cMo\subset \cM$ 
so that $W_q$ satisfies the  ``non-degeneracy condition at infinity" 
due to Kouchnirenko \cite[1.19]{kouchnirenko}. 
\begin{definition}
Let $\hS$ denote the convex hull of 
$b_1,\dots,b_m\in \bN\otimes \R$. 
We call the Laurent polynomial $W_q(y)$ of the form (\ref{eq:W_q}) 
{\it non-degenerate at infinity} 
if for every face 
$\Delta$ of $\hS$ (where $0\le \dim\Delta \le n-1$),   
$W_{q,\Delta}(y):=\sum_{b_i\in \Delta} q^{\ell_i} y^{b_i}$ 
does not have critical points on $y\in (\C^*)^n$. 
Let $\cMo$ be the subset of $\cM$ consisting of $q$ 
for which $W_q$ is non-degenerate at infinity.   
\end{definition} 

\begin{proposition}
\label{prop:kouchnirenko} 
(i) Under the condition (C) 
in Section \ref{subsubsec:def_toricorbifolds}, 
$0\in \bN\otimes \R$ is in the interior of $\hS$. 
Therefore, the Laurent polynomial $W_q$ is 
\emph{convenient} in the sense of 
Kouchnirenko \cite[1.5]{kouchnirenko}. 

(ii) $\cMo$ is an open and dense subset of $\cM$ in Zariski topology.  

(iii) For $q\in \cMo$, $W_q(y)$ has   
$|\bN_{\rm tor}|\times n!\Vol(\hS)$ critical points 
on $Y_q$ (counted with multiplicities). 
\end{proposition} 
\begin{proof}
The condition (C) implies that 
there exists $d\in \bL$ such that $c_i:=\pair{D_i}{d}>0$. 
Then by the exact sequence (\ref{eq:exactsequence_toric}),  
we have $\sum_{i=1}^m c_i b_i=0$. This proves (i). 
The statements (ii) and (iii) are due to Kouchnirenko. 
(ii) follows from (i) 
and the same argument as in \cite[6.3]{kouchnirenko}. 
One of main theorems in \cite[1.16]{kouchnirenko} 
states that $W_q(y)$ has $n!\Vol(\hS)$ number of critical points 
on each connected component of $Y_q$. 
(iii) follows from this and $|\pi_0(Y_q)| = |\bN_{\rm tor}|$. 
\end{proof} 

The following lemma shows that the 
Kouchnirenko's condition holds on 
a certain ``cylindrical end" of $\cM$.   
A proof is given in the Appendix \ref{subsec:qsmall}. 

\begin{lemma}
\label{lem:qsmall_kouchnirenko} 
Let $q_1,\dots,q_r$ be the co-ordinates on $\cM$ 
dual to the basis $p_1,\dots,p_r\in \cl(\tC_\cX)$ 
chosen in Section \ref{subsubsec:KC_nefbasis}. 
There exists $\epsilon>0$ such that $q\in \cMo$ 
if $0<|q_a|<\epsilon$ for all $a$. 
\end{lemma} 

\begin{lemma}
\label{lem:rankmatch} 
Assume that $\hrho \in \cl(\tC_\cX)$. 
Then $\hS$ is the union of simplices 
$\hS(\sigma) := \{\sum_{b_i\in \sigma} c_i b_i\;;\; 
c_i \in [0,1], \sum_{b_i\in \sigma} c_i \le 1\}$ 
over maximal ($n$-dimensional) 
cones $\sigma$ of the fan $\Sigma$ 
of $\cX$. 
Moreover, we have 
$|\bN_{\rm tor}|\times n!\Vol(\hS)= \dim H_{\rm orb}^*(\cX)$.  
\end{lemma} 
\begin{proof} 
By Lemma \ref{lem:weakFano}, $\rho= c_1(\cX)$ is nef. 
This implies that the piecewise linear function 
$h\colon \bN\otimes \R \to \R$ on the fan $\Sigma$ 
(linear on each maximal cone in $\Sigma$) 
defined by $h(b_i) = 1$ for $1\le i\le m'$ is convex (see \cite{oda}). 
Therefore, $\bigcup_{\sigma:\dim \sigma = n} 
\hS(\sigma) = h^{-1}((-\infty,1])$ is convex. 
Because $b_j$, $j>m'$ is contained in this by 
Lemma \ref{lem:weakFano}, we have 
$\hS=\bigcup_{\sigma: \dim\sigma =n} \hS(\sigma)$. 

Because odd cohomology groups of 
$\cX_v$ vanish, 
$\dim H^*(\cX_v)$ is equal to 
the Euler number of $\cX_v$, so is equal to 
the number of torus fixed points on $\cX_v$ 
(for the natural torus action) by Poincar\'{e}-Hopf. 
Torus fixed points on $\cX_v$ are parametrized 
by maximal cones $\sigma$ in the fan $\Sigma$ 
such that $\sigma$ contains 
the image of $v\in \Boxop$ in $\bN\otimes \R$. 
Hence,  
\[
\sum_{v\in \Boxop} \dim H^*(\cX_v) 
= \sum_{\sigma: \dim \sigma = n} \sharp\{v\in \Boxop\;;\; v\in \sigma\} 
= \sum_{\sigma: \dim \sigma =n} |\bN_{\rm tor}| \times 
n! \Vol(\hS(\sigma)).    
\] 
The conclusion follows. 
\end{proof}

\subsubsection{Jacobi ring and Batyrev ring}
The \emph{Jacobi ring} $J(W)$ is  
the ring of functions on the (fiberwise) critical set of $W$:   
\[
J(W) := \C[w_1^\pm,\dots,w_m^\pm]\Big/
\left \langle 
y_1\parfrac{W}{y_1},\dots, y_n\parfrac{W}{y_n} \right \rangle. 
\]
Note that $J(W)$ 
is a $\C[q^\pm]:=\C[q_1^\pm,\dots,q_r^{\pm}]$-algebra. 
Denote by  $J(W_q) = J(W) \otimes_{\C[q^\pm]} \C_q$ 
the fiber of $J(W)$ at $q\in \cM=\Spec\C[q^\pm]$. 
By Proposition \ref{prop:kouchnirenko}, 
$J(W_q)$ is of dimension 
$|\bN_{\rm tor}| \times n! \Vol(\hS)$ 
when $q\in \cMo$.  
The \emph{Batyrev ring} is defined by 
\[
\textstyle 
B(\cX) := 
\C[q^\pm][\sfp_1,\dots,\sfp_r] 
\Big/ 
\left\langle q^d 
\prod_{i : \pair{D_i}{d}<0} 
\sfw_i^{-\pair{D_i}{d}} - 
\prod_{i: \pair{D_i}{d}>0} 
\sfw_i^{\pair{D_i}{d}} \; ;\;  
d\in \bL \right\rangle 
\] 
where $q^d := \prod_{a=1}^r q_a^{\pair{p_a}{d}}$ 
and $\sfw_i :=\sum_{a=1}^r \sfm_{ia} \sfp_a$. 
By the condition (C) in 
Section \ref{subsubsec:def_toricorbifolds},  
there exists $d\in \bL$ such that 
$c_i := \pair{D_i}{d}>0$ for all $i$. 
Hence $\prod_{i=1}^m \sfw_i^{c_i} = q^d$ 
holds in $B(\cX)$ and 
therefore $\sfw_i$ is invertible in $B(\cX)$. 
With this fact in mind, Batyrev ring is given 
by the simple relations (note that $\sfm_{ia}$ can be negative) 
\begin{equation}
\label{eq:char_var}
q_a = \prod_{i=1}^m \sfw_i^{\sfm_{ia}} 
= \prod_{i=1}^m \left(
\sum_{b=1}^r \sfm_{ib} \sfp_b\right)^{\sfm_{ia}}, \quad 
1\le a\le r.   
\end{equation} 
The following was shown in 
\cite[Lemma 5.10, Proposition 5.11]{iritani-coLef} 
for toric manifolds. 

\begin{proposition}
\label{prop:Jac-Bat} 
(i) The map $B(\cX) \to J(W)$, 
$\sfp_a \mapsto [q_a (\partial W_q/ \partial q_a)]$ 
defines an isomorphism of $\C[q^\pm]$-algebras.  

(ii) Let $\cMoo$ be the subset of $\cMo$ 
consisting of $q\in \cMo$ such that 
all the critical points of $W_q$ are non-degenerate. 
Then $\cMoo$ is open and dense in $\cMo$.  
\end{proposition}
\begin{proof} 
(i) Since $\pr_*(w_i (\partial/\partial w_i) ) 
= \sum_{a=1}^r \sfm_{ia} q_a (\partial/\partial q_a)$, 
we have 
\[
\left[\sum_{a=1}^r \sfm_{ia} q_a \parfrac{W_q}{q_a} \right] = 
\left[w_i \parfrac{W}{w_i} \right ]= 
[w_i] \quad \text{ in } J(W). 
\]
This shows that $\sfw_i$ maps to an invertible element 
$[w_i]\in J(W)$ satisfying  
$\prod_{i=1}^m [w_i]^{\sfm_{ia}} = q_a$. 
Thus the map $B(\cX)\to J(W)$ is well-defined. 
The inverse map, sending $[w_i]$ to $\sfw_i$, 
is also well-defined. The details are left to the reader.  

(ii) The isomorphism in (i) induces an isomorphism  
$\Spec B(\cX) \cong \Spec J(W)$ over $\cM$. 
Since $\Spec B(\cX)$ can be written as the graph 
of the map $\sfp \mapsto q$ (\ref{eq:char_var}), 
it suffices to show that this map is 
a local isomorphism at generic $\sfp$. 
This follows from the fact that 
the Jacobian $\partial \log q_a/\partial \sfp_b 
= \sum_{i=1}^m \sfm_{ia} \sfw_i^{-1} \sfm_{ib}$ of the map 
(\ref{eq:char_var}) is positive definite when $\sfw_i>0$. 
(Note that we can choose $\sfp_b\in \R$ 
so that $\sfw_i = \sum_{b=1}^r \sfm_{ib} \sfp_b >0$ 
for all $i$ again by the condition (C)). 
\end{proof}

\subsection{B-model $D$-module} 
Here we describe the B-model $D$-module in two steps. 
First we construct a local system over 
$\cMo\times \C^*$ 
using the Morse theory for $\Re(W_q/z)$. 
Then we extend the local system to a meromorphic 
flat connection over $\cMo\times \C$ 
using de Rham forms and oscillatory integrals. 

\subsubsection{Local system of Lefschetz thimbles} 
Let $f_{q,z}\colon Y_q \to \R$ be the real part of the function 
$y\mapsto W_q(y)/z$. 
The following lemma allows us to use Morse theory 
for the improper function $f_{q,z}(y)$. 
\begin{lemma}
\label{lem:PScond}
For each $\epsilon>0$, the family of topological spaces 
\[ 
\bigcup_{(q,z)\in \cMo\times \C^*}
\{y\in Y_q \;;\; \|df_{q,z}(y)\|\le \epsilon \} \to \cMo\times \C^*
\]
is proper, i.e. pull-back of a compact set is compact.  
Here the norm $\|df_{q,z}(y)\|$ is taken with respect to 
the complete K\"{a}hler metric 
$\frac{1}{\iu} \sum_{i=1}^n d\log y_i \wedge d\ov{\log y_i}$ 
on $Y_q$. 
\end{lemma} 

A similar result for polynomial functions can be found 
in \cite[Proposition 2.2 and Remarque]{pham_GM} 
and this lemma may be well-known to specialists. 
A proof is given in the Appendix \ref{subsec:proof_PS}
since the author was not able to find a suitable reference. 
Lemma \ref{lem:PScond} implies that $f_{q,z}$ 
satisfies the Palais-Smale condition, 
so that usual Morse theory applies to $f_{q,z}$ 
(see \emph{e.g.} \cite{palais}). 
Take $(q,z)\in \cMo\times \C^*$. 
Since  the set $\{y\in Y_q\; ;\; \|df_{q,z}(y)\|<\epsilon\}$ 
is compact, we can choose $M\ll 0$ so that 
this set is contained in $\{y\in Y_q\;;\;f_{q,z}(y)>M\}$. 
Then the relative homology group 
$H_n(Y_q, \{y\in Y_q\;;\; f_{q,z}(y)\le M\};\Z)$ 
is independent of the choice of such $M$ and 
we denote this by 
\begin{equation}
\label{eq:relative_hom}
R_{\Z,(q,z)}^\vee = 
H_n(Y_q, \{y\in Y_q\;;\; f_{q,z}(y) \ll 0\};\Z), \quad 
(q,z) \in \cMo \times \C^*.    
\end{equation} 
The number of critical points of $f_{q,z}(y)$ is 
$N:=|\bN_{\rm tor}|\times n!\Vol(\hS)$ 
by Proposition \ref{prop:kouchnirenko}. 
If all the critical points of $W_q(y)$ 
are non-degenerate, by the standard argument in Morse theory,  
we know that $Y_q$ is obtained from $\{f_{q,z}(y)\le M\}$ by attaching 
$N$ $n$-handles and so 
$R_{\Z,(q,z)}^\vee$ is a free abelian group of rank $N$. 
If $W_{q}(y)$ has a critical point $y_0$ 
of multiplicity $\mu_0>1$, one can find\footnote{
We can find $\tilde{f}_{q,z}$ in the following way: 
Let $\rho\colon \R_{\ge 0} \to [0,1]$ be a $C^\infty$-function 
such that $\rho(r)=1$ for $0\le r\le 1/2$ and $\rho(r)=0$ for $r\ge 1$. 
Let $U_0$ be an $\epsilon$-neighborhood of $y_0$ 
(in the above K\"{a}hler metric) 
which does not contain other critical points. 
Let $t=(t_1,\dots,t_n)$ be co-ordinates given by $y_i=y_{0,i} e^{t_i}$. 
For $a=(a_1,\dots,a_n)\in \C^n$, put 
$f^a_{q,z}(y) = f_{q,z}(y) +  \rho(|t|/\epsilon)\Re(at)$. 
Then for a generic, sufficiently small $a$, $\tilde{f}_{q,z}=f^a_{q,z}$ 
satisfies the conditions above (here,  
new critical points are all in $|t|<\epsilon/2$).} 
a small $C^\infty$-perturbation $\tilde{f}_{q,z}$ of 
$f_{q,z}$ on a small neighborhood $U_0$ of $y_0$ such that 
$\tilde{f}_{q,z}$ has just $\mu_0$ 
non-degenerate critical points in $U_0$ 
with Morse index $n$. 
By considering such a perturbation and 
Morse theory for $f_{q,z}$ in families (parametrized by $q$ and $z$), 
we obtain the following. 

\begin{proposition}
\label{prop:locsys_Lefschetz} 
The relative homology groups 
$R^\vee_{\Z,(q,z)}$ in (\ref{eq:relative_hom}) 
form a local system of rank $|\bN_{\rm tor}|\times 
n! \Vol(\hS)$ over $\cMo\times \C^*$. 
\end{proposition} 

When all the critical points $\crit_1,\dots,\crit_N$ 
of $W_q\colon Y_q \to \C$ are non-degenerate, 
a basis of the local system $R^\vee_\Z$ is given by a set of 
{\it Lefschetz thimbles} $\Gamma_1,\dots, \Gamma_N$:  
the image of $\Gamma_i$ under $W_q/z$ is given by 
a curve $\gamma_i:[0,\infty) \to \C$ such that 
$\gamma(0)=W_q(\crit_i)/z$, that 
$\Re\gamma_i(t)$ decreases monotonically 
to $-\infty$ as $t\to \infty$ 
and that $\gamma_i$ does not pass through critical values 
other than $W_q(\crit_i)/z$;   
$\Gamma_i$ is the union of cycles in $W_q^{-1}(z\gamma_i(t))$ 
collapsing to $\crit_i$ along the path $\gamma_i(t)$ as $t\to 0$.  
When the imaginary parts 
$\Im(W_q(\crit_1)/z),\dots,\Im(W_q(\crit_N)/z)$ 
are mutually different, $\Gamma_i$ can be taken to be 
the union of downward gradient flowlines 
of $f_{q,z}(y)$ emanating from $\crit_i$. 
(Note that the gradient 
flow of $f_{q,z}=\Re(W_q/z)$ with respect to a K\"{a}hler metric 
coincides with the Hamiltonian flow  
generated by $\Im(W_q/z)$.) 
Then $\gamma_i$ becomes a half-line parallel to the real axis. 
The intersection pairing defines a unimodular pairing:
\begin{equation}
\label{eq:pairing_Rvee}
R_{\Z,(q,-z)}^\vee \times R_{\Z,(q,z)}^\vee \to\Z.  
\end{equation} 
Let $R_\Z \to \cMo \times \C^*$ 
be the local system dual to $R_\Z^\vee$ 
and $\cR := R_\Z \otimes \cO_{\cMo\times \C^*}$ be the associated 
locally free sheaf on $\cMo\times \C^*$. 
The sheaf $\cR$ is equipped with 
the Gau\ss-Manin connection $\nabla\colon 
\cR \to \cR \otimes \Omega^1_{\cMo\times \C^*}$ 
and the pairing 
$(\cdot,\cdot)_{\cR}\colon 
((-)^*\cR) \otimes \cR \to \cO_{\cMo\times \C^*}$ 
induced from the local system $R_\Z^\vee$. 

\subsubsection{The extension across 
$z=0$ via de Rham forms} 

Let $\omega_1$ be the following holomorphic volume form 
on $Y_1 = \Hom(\bN,\C^*)$: 
\[
\omega_1= 
\frac{1}{|\bN_{\rm tor}|}
\frac{dy_1\cdots dy_n}{y_1\cdots y_n} \quad 
\text{on each connected component.}
\]
This is characterized as a unique translation-invariant 
holomorphic $n$-form $\omega_1$ satisfying 
$\int_{\Hom(\bN,S^1)} \omega_1 = (2\pi \iu)^n$. 
By translation, $\omega_1$ defines a holomorphic volume form 
$\omega_q$ on each fiber $Y_q$. 
Let $\pr\colon \Yo\to \cMo$ be the restriction of the family 
$\pr\colon Y\to \cM$ to $\cMo$. 
Consider a relative holomorphic $n$-form $\varphi$ 
of $\Yo\times \C^* \to \cMo \times \C^*$ 
of the form 
\begin{equation}
\label{eq:relative_nform}
\varphi = f(q,z,y) e^{W_q(y)/z} \omega_q, \quad 
f(q,z,y) \in \cO_{\cMo \times \C^*}[y_1^\pm,\dots,y_n^\pm]   
\end{equation} 
where $\cO_{\cMo\times \C^*}$ is the analytic structure sheaf. 
This relative $n$-form 
gives a holomorphic section $[\varphi]$ of $\cR$ via the integration 
over Lefschetz thimbles: 
\begin{equation}
\label{eq:pairing_R_Rvee}
\pair{[\varphi]}{\Gamma} = \frac{1}{(-2\pi z)^{n/2}} \int_{\Gamma} 
f(q,z,y) e^{W_q(y)/z} \omega_q
\in \cO_{\cMo \times \C^*}  
\end{equation} 
The convergence of this integral is ensured by the fact that 
$f(q,z,y)$ has at most polynomial growth in $y$ 
and that $\Re(W_q(y)/z)$ goes to $-\infty$ at the end of $\Gamma$. 
More technically, as done in \cite{pham_GM}, 
one may prove the convergence of the integral 
by replacing the end of $\Gamma$ with a semi-algebraic chain.  

\begin{definition}
A section of $\cR$ on an open set 
$U\times \{0<|z|<\epsilon\}\subset \cMo\times \C^*$ 
is defined to be \emph{extendible to $z=0$}  
if it is the image of a relative $n$-form $\varphi$ 
of the form (\ref{eq:relative_nform}) such that 
$f(q,z,y)$ in (\ref{eq:relative_nform}) is regular at $z=0$. 
The sections extendible to $z=0$ define the extension 
$\cRz$ of the sheaf $\cR$ to $\cMo \times \C$. 
\end{definition} 

Let $\cR'$ be the $\cO_{\cMo \times \C^*}$-submodule of $\cR$ 
consisting of the sections which locally 
arise from relative $n$-forms $\varphi$ 
of the form (\ref{eq:relative_nform}). 
The Gau\ss-Manin connection on $\cR$ preserves the subsheaf $\cR'$. 
In fact, we have  
\begin{align}
\label{eq:GM_deRham} 
\begin{split} 
\nabla_{a}[\varphi] &= 
\left[\left(
\partial_a f+ \frac{1}{z} (\partial_aW_q) f\right) 
e^{W_q/z} \omega_q \right],  \\
\nabla_{z\partial_z}[\varphi] &= 
\left[\left ( 
z\partial_z f - \frac{1}{z} W_q f -\frac{n}{2}f \right ) 
e^{W_q/z} \omega_q\right], 
\end{split} 
\end{align} 
where $\varphi$ is given in (\ref{eq:relative_nform}) 
and $\partial_a = q_a (\partial/\partial q_a)$. 
Take a point $q$ in the open subset 
$\cMoo\subset \cMo$ appearing 
in Proposition \ref{prop:Jac-Bat}.  
Let $\Gamma_1,\dots,\Gamma_N$ be Lefschetz thimbles of $W_q(y)/z$ 
corresponding to critical points $\crit_1,\dots,\crit_N$.  
Then we have the following asymptotic expansion 
as $z \to 0$ with $\arg(z)$ fixed: 
\begin{equation}
\label{eq:asymptotic_exp_in_z} 
\frac{1}{(-2\pi z)^{n/2}} 
\int_{\Gamma_i} f(q,z,y) e^{W_q(y)/z} \omega_q 
\sim \frac{1}{|\bN_{\rm tor}|} 
\frac{e^{W_q(\crit_i)/z} }{\sqrt{\Hess(W_q)(\crit_i)}} 
(f(q,0,\crit_i) +O(z)) 
\end{equation} 
where $f(q,z,y) \in \cO_{\cMo\times \C}[y_1^\pm,\dots,y_n^\pm]$ 
is regular at $z=0$ and 
$\Hess(W_q)$ is the Hessian of $W_q$ calculated in 
co-ordinates $\log y_1,\dots,\log y_n$. 
Let $\phi_i(y)$ be a regular function on $Y_q$ 
which represents the delta-function supported on 
$\crit_i$ in the Jacobi ring $J(W_q)$. 
Put $\varphi_i = \phi_i(y) e^{W_q/z} \omega_q$.  
By the asymptotics of $\pair{[\varphi_i]}{\Gamma_j}$,  
we know that $[\varphi_1],\dots,[\varphi_N]$ form a basis of $\cR$ for 
sufficiently small $|z|>0$. 
Since $\cR'$ is preserved by the Gau\ss-Manin connection, 
we have $\cR=\cR'$ on the whole $\cMo\times \C^*$. 
In other words, $\cR$ is generated by 
relative $n$-forms of the form (\ref{eq:relative_nform}). 

Let $\Gamma_1^\vee,\dots,\Gamma_N^\vee$ be 
the Lefschetz thimbles of $W_q/(-z)$. 
These are dual to $\Gamma_1,\dots,\Gamma_N$ 
with respect to the intersection pairing (\ref{eq:pairing_Rvee}).  
Then the pairing on $\cR$ can be written as   
\begin{equation}
\label{eq:B-model_pairing}
([\varphi(-z)],[\varphi'(z)])_{\cR} 
= \frac{1}{(2\pi \iu z )^n}\sum_{i=1}^N  
\int_{\Gamma_i^\vee} \varphi(-z)  \cdot \int_{\Gamma_i} 
\varphi'(z).  
\end{equation} 
When $[\varphi]$ and $[\varphi']$ are extendible to $z=0$, 
we have from (\ref{eq:B-model_pairing}) 
and (\ref{eq:asymptotic_exp_in_z}) 
\[
([\varphi],[\varphi'])_{\cR} 
\sim \frac{1}{|\bN_{\rm tor}|^2}
\sum_{i=1}^N \frac{f(q,0,\crit_i) f'(q,0,\crit_i)}{\Hess W_q(\crit_i)} 
+ O(z) 
\]
where we put $\varphi = f(q,z,y) e^{W_q(y)/z} \omega_q$ and 
$\varphi'=f'(q,z,y) e^{W_q(y)/z} \omega_q$. 
This shows that $([\varphi],[\varphi'])_{\cR}$ is 
regular at $z=0$ and the value at $z=0$ equals the residue 
pairing on $J(W_q)$.  
By continuity, we have at all $q\in \cMo$: 
\[
([\varphi],[\varphi'])_{\cR}|_{z=0} = \frac{1}{|\bN_{\rm tor}|^2}
\Res_{\Yo/\cMo} \left [ 
\frac{f(q,0,y) f'(q,0,y) \frac{dy_1\cdots dy_n}{y_1\cdots y_n}}
{y_1\parfrac{W_q}{y_1},\dots, y_n\parfrac{W_q}{y_n}} 
\right].  
\]
Let $\phi'_1,\dots,\phi'_N$ be an arbitrary basis 
of the Jacobi ring and 
put $s_i:=[\phi'_i(y) e^{W_q(y)/z}\omega_q]$. 
Then the Gram matrix $(s_i,s_j)_{\cR}$ is non-degenerate 
in a neighborhood of $z=0$ 
since the residue pairing is non-degenerate. 
This implies that $s_1,\dots,s_N$ 
form a local basis of $\cRz$ around $z=0$. 
Summarizing, 
\begin{proposition}[{\cite[Lemma 2.19]{CIT:I}}]
\label{prop:locfree_cRz} 
The $\cO_{\cMo\times \C^*}$-module $\cR$ 
is generated by relative $n$-forms 
of the form (\ref{eq:relative_nform}). 
The extension $\cRz$ of $\cR$ to $\cMo\times \C$ 
is locally free and the pairing on 
$\cR$ extends to a non-degenerate pairing 
$((-)^*\cRz) \otimes \cRz \to \cO_{\cMo\times \C}$.   
\end{proposition} 

In the algebraic construction by Sabbah, 
the corresponding results were proved 
in \cite[Corollary 10.2]{sabbah-hypergeometric} 
(see also \cite[Proposition 2.13]{douai-sabbah-I}). 

The \emph{Euler vector field} $E$  
on $\cMo$ is defined by 
\begin{equation}
\label{eq:B-model_Euler} 
E := \pr_*\left(\sum_{i=1}^m w_i\parfrac{}{w_i}\right)  
= \sum_{a=1}^r \rho_a q_a \parfrac{}{q_a}, \quad 
\rho_a := \sum_{i=1}^m \sfm_{ia}. 
\end{equation} 
The grading operator $\Grading$ 
acting on sections of $\cRz$ is defined by 
\begin{align}
\label{eq:B-model_grading}
\Grading[\varphi] 
&= 2\left[\left(z\parfrac{f}{z} + \sum_{i=1}^m w_i \parfrac{f}{w_i}\right) 
e^{W/z} \omega \right]  
\end{align} 
for a section $[\varphi]$ of the form (\ref{eq:relative_nform}). 
This grading operator can be written in terms of 
the Gau\ss-Manin connection and the Euler vector field 
(\emph{cf.} the grading operator (\ref{eq:Euler_reg}) 
for the A-model):  
\begin{lemma} 
\label{lem:B-model_grading} 
$\Grading = 2(\nabla_E + \nabla_{z\partial_z} + \frac{n}{2} )$.
\end{lemma} 
\begin{proof}
Using the co-ordinate system 
$(q_a, y_i)$ on $Y$ in Section \ref{subsubsec:LGmodel_def}, 
we can write $\sum_{i=1}^m w_i \parfrac{}{w_i} 
= E + \sum_{i=1}^n c_i y_i\parfrac{}{y_i}$ for some $c_i\in \Q$. 
Here we lift $E$ to a vector field on $Y$ by 
using the co-ordinates $(q_a,y_i)$.   
By $(\sum_{i=1}^m w_i \parfrac{}{w_i})W = W$, 
we have 
\begin{align*} 
\frac{1}{2} \Grading[\varphi] &= 
\left[\left(\left(z\partial_z + 
\textstyle\sum_{i=1}^m w_i \partial_{w_i}\right) 
(f e^{W/z})\right) \omega\right] \\ 
& = \left(
\nabla_{z\partial_z} + \frac{n}{2} + \nabla_E \right)[\varphi] 
+ \left[\left(\left(\textstyle\sum_{i=1}^n 
c_i y_i \partial_{y_i}\right) 
(f e^{W/z})\right) \omega \right]. 
\end{align*} 
The second term is zero in cohomology since it is exact. 
\end{proof} 

\begin{definition}[\emph{cf.} Definition \ref{def:QDM}] 
\label{def:BDM}
Let $\pi\colon \cMo\times \C \to \cMo$ be the projection 
and $(-)\colon \cMo\times \C \to \cMo \times \C$ 
be the map sending $(q,z)$ to $(q,-z)$. 
The \emph{B-model $D$-module of the LG model} 
is the tuple $(\cRz,\nabla,(\cdot,\cdot)_{\cRz})$ 
consisting of 
the locally free sheaf $\cRz$ over $\cMo\times \C$, 
the meromorphic flat connection (\ref{eq:GM_deRham})    
\[
\nabla \colon \cRz \to \cRz(\cMo\times \{0\}) 
\otimes_{\cO_{\cMo\times \C}}  
(\pi^*\Omega^1_{\cMo} \oplus \cO_{\cMo\times \C} \frac{dz}{z})  
\]
and the $\nabla$-flat pairing (\ref{eq:B-model_pairing}) 
\[(\cdot,\cdot)_{\cRz} 
\colon (-)^*\cRz \otimes_{\cO_{\cMo\times \C}} 
\cRz \to \cO_{\cMo\times \C} 
\] 
satisfying $((-)^*s_1,s_2)_{\cRz} = (-)^*((-)^*s_2,s_1)_{\cRz}$.  
This is also equipped with the grading operator  
$\Grading\colon \cRz\to \cRz$ in (\ref{eq:B-model_grading}). 
\end{definition} 

Note that the B-model $D$-module is underlain by 
the integral local system of Lefschetz thimbles. 

\begin{proposition}
\label{prop:BDM_diffgen} 
The B-model $D$-module $\cRz$ is generated by 
$[e^{W_q/z} \omega_q]$ and its derivatives 
$z\nabla_{a_1}z \nabla_{a_2} \cdots z 
\nabla_{a_k} [e^{W_q/z} \omega_q]$ 
as an $\cO_{\cMo\times \C}$-module, 
where $\nabla_a = \nabla_{q_a (\partial/\partial q_a)}$. 
\end{proposition} 
\begin{proof} 
By the discussion preceding Proposition 
\ref{prop:locfree_cRz}, the restriction 
$\cRz|_{\cMo\times\{0\}}$ is identified with the bundle 
$J(W)$ of Jacobi rings over $\cMo$ by 
the map $[f(q,z,y) e^{W_q/z}\omega_q] \mapsto [f(q,0,y)]$. 
Under this identification, the action of $z\nabla_a$ 
corresponds to the multiplication by 
$q_a (\partial W_q/\partial q_a)$ by (\ref{eq:GM_deRham}). 
Because $J(W) \cong B(\cX)$ by Proposition 
\ref{prop:Jac-Bat} and $B(\cX)$ is generated by 
$\sfp_a$'s, $J(W)$ is generated by 
$q_a (\partial W_q/\partial q_a)$ 
as a $\C[q^\pm]$-algebra. 
Therefore, $\cRz$ is generated by 
$z\nabla_{a_1}\cdots z\nabla_{a_k} 
[e^{W_q/z} \omega_q]$ in the neighborhood of $z=0$. 
Let $\tcRz$ be the $\cO_{\cMo\times \C}$-submodule 
of $\cRz$ generated by these derivatives. 
From $\Grading [e^{W_q/z} \omega_q] = 0$ 
and Lemma \ref{lem:B-model_grading}, one finds that 
\[
\nabla_{z^2\partial_z} [e^{W_q/z} \omega_q] 
= \left( \sum_{a=1}^r \rho_a z \nabla_a - \frac{n}{2} z \right) 
[e^{W_q/z} \omega_q]. 
\] 
Hence $\tcRz$ is preserved by $\nabla_{z^2 \partial_z}$. 
Therefore, $\tcRz = \cRz$. 
\end{proof} 

\section{Mirror symmetry for toric orbifolds and 
integral structures} 
\label{sec:intstr_toricmirror} 
Under mirror symmetry, 
the A-model $D$-module (quantum $D$-module) 
should be isomorphic to the B-model $D$-module. 
We give a precise mirror symmetry conjecture for a weak Fano toric 
orbifold and check that the mirror symmetry 
matches up the $\hGamma$-integral structure 
in the A-side and 
the natural integral structure 
in the B-side. 

\subsection{$I$-function and mirror theorem} 
A Givental style mirror theorem 
for a toric orbifold can be stated as 
the equality of the $I$-function and the $J$-function. 
This has been proved for weak Fano toric manifolds 
\cite{givental-mirrorthm-toric} and 
weighted projective spaces \cite{CCLT:wp}. 
A general case for toric orbifolds 
will be proved in \cite{CCIT:toric}. 

\begin{definition}[\cite{CCIT:toric}] 
\label{def:I-funct} 
The \emph{$I$-function} of a toric orbifold $\cX$ is 
an $H_{\rm orb}^*(\cX)$-valued power series 
on $\cM$ defined by 
\[
I(q,z) = e^{\sum_{a=1}^{r} \ov{p}_a \log q_a/z} 
\sum_{d \in \K_{\rm eff}} q^d  
\frac{\prod_{i: \pair{D_i}{d}<0} \prod_{\pair{D_i}{d}\le \nu <0} 
(\ov{D}_i + (\pair{D_i}{d}-\nu) z) }
{\prod_{i: \pair{D_i}{d}>0} \prod_{0\le \nu< \pair{D_i}{d}}
(\ov{D}_i + (\pair{D_i}{d}-\nu) z) }
\unit_{v(d)} 
\]
where $q^d = q_1^{\pair{p_1}{d}}\dots q_r^{\pair{p_r}{d}}$ 
and the index $\nu$ moves in $\Z$. 
Recall that $\ov{p}_a$ and $\ov{D}_j$ 
are the images of 
$p_a$ and $D_j$ under the projection 
$\bL^\vee\otimes \Q \to H^2(\cX,\Q)$.   
Note that $\ov{p}_a = 0$ for $a>r'$, 
$\ov{D}_j =0$ for $j>m'$ and 
$\pair{p_a}{d}\ge 0$ for $d\in \K_{\rm eff}$. 
\end{definition} 

Choose $e_0\in \N$ such that $e_0 \K \subset \bL$. 
Then $e^{-\sum_{a=1}^r \ov{p}_a\log q_a/z} I(q,z)$ 
belongs to $H^*_{\rm orb}(\cX) \otimes \C[z,z^{-1}]
[\![q_1^{1/e_0},\dots,q_r^{1/e_0}]\!]$.  
The $I$-function can be also written in the form: 
\begin{equation}
\label{eq:I-funct_another}
I(q,z) = e^{\sum_{a=1}^r \ov{p}_a \log q_a/z} 
\sum_{d\in \K} q^d \prod_{i=1}^m 
\frac{\prod_{\nu = \ceil{\pair{D_i}{d}} }^{\infty} 
(\ov{D}_i + (\pair{D_i}{d}-\nu) z)}  
{\prod_{\nu =0}^\infty 
(\ov{D}_i + (\pair{D_i}{d}-\nu) z)}\unit_{v(d)}. 
\end{equation}  
Note that all but finite factors cancel 
in the infinite products. 
The summand with $d\in \K\setminus \K_{\rm eff}$ 
vanishes in $H_{\rm orb}^*(\cX)$ 
because we have 
$(\prod_{i:\pair{D_i}{d}\in \Z_{<0}} \ov{D}_i) 
\unit_{v(d)}$ in the numerator 
and this is zero in $H^*(\cX_{v(d)})$ by 
the presentation (\ref{eq:coh_presentation}).  

The $I$-function defines an analytic function 
when $\hrho \in \cl(\tC_\cX)$. 
See Section \ref{subsubsec:weakFano} 
for the condition $\hrho \in \cl(\tC_\cX)$.  

\begin{lemma} 
\label{lem:convergence} 
The $I$-function is a convergent power series in $q_1,\dots,q_r$ 
if and only if 
$\hrho$ is in the closure $\cl(\tC_\cX)$ 
of the extended K\"{a}hler cone.  
In this case, the $I$-function has the asymptotics
\[
I(q,z) = 1 + \frac{\tau(q)}{z} + o(z^{-1})
\]
where $\tau(q)$ is a multi-valued function 
taking values in $H^{\le 2}_{\rm orb}(\cX)$. 
The map $q \mapsto \tau(q)$ 
is called the \emph{mirror map}. 
\end{lemma}

When $\hrho \in \cl(\tC_\cX)$, 
the mirror map $\tau$ takes the form 
\begin{equation}
\label{eq:mirrormap_exp} 
\tau(q) = \sum_{a=1}^{r'} (\log q_a) \ov{p}_a +
 \sum_{j=m'+1}^m q^{D^\vee_j} \frD_j 
+ \text{ h.o.t.}, 
\end{equation} 
where h.o.t.\ (higher order term) 
is a power series in $q_1^{1/e_0},\dots q_r^{1/e_0}$.  
Thus $\tau$ is a local embedding (resp. isomorphism) 
near $q=0$ 
if $\ov{p}_1,\dots,\ov{p}_{r'},\frD_{m'+1},\dots,\frD_m$ 
are linearly independent (resp. basis of $H^{\le 2}_{\rm orb}(\cX)$).   
See (\ref{eq:frD}) for $\frD_j$.  
The following ``mirror theorem"  
will be proved in \cite{CCIT:toric}. 
\begin{conjecture}
\label{conj:mirrorthm} 
Assume that $\hrho\in \cl(\tC_\cX)$. Then the $I$-function 
and the $J$-function coincide via the co-ordinate change 
$\tau = \tau(q)$: 
\[
I(q,z) = J(\tau(q),z),  
\]
where $\tau(q)$ is the mirror map 
in Lemma \ref{lem:convergence}. 
\end{conjecture} 

We remark that the equality $I=J$ above is consistent 
with monodromy transformations on $\cM$. 
Take a loop $[0,1]\ni \theta 
\mapsto e^{-2\pi\iu \xi \theta} q 
= (e^{-2\pi\iu\xi_1 \theta} q_1, \dots, e^{-2\pi\iu\xi_r \theta} q_r)\in \cM$ 
for $\xi = \sum_{a=1}^r \xi_a p_a\in \bL^\vee$. 
The monodromy of $I(q,z)$ along this loop is given by  
\[
I(e^{-2\pi\iu \xi}q,z) = G^\cH(\xi) I(q,z)
\] 
where $G^\cH(\xi)=G^{\cH}([\xi])$ 
is the Galois action (\ref{eq:GaloisH}) 
of the class $[\xi]$ in 
$\bL^\vee/\sum_{j>m'} \Z D_j \cong H^2(\cX,\Z)$.   
Therefore, we have 
\begin{equation}
\label{eq:tauGalois} 
\tau(e^{-2\pi\iu\xi}q) = G(\xi) \tau(q)  
\end{equation} 
where $G(\xi)=G([\xi])$ is given in (\ref{eq:Galois}). 
These two equations are compatible with 
the behavior (\ref{eq:Galois_J}) of $J(\tau,z)$. 
This moreover shows that $\tau$ 
induces a single-valued map 
\begin{equation}
\label{eq:mirrormap_quot} 
\tau \colon 
\{(q_1,\dots,q_r)\in \cM\;;\; 0<|q_a|<\epsilon \}  
\longrightarrow H^{\le 2}_{\rm orb}(\cX)/H^2(\cX,\Z) 
\end{equation} 
for a sufficiently small $\epsilon>0$. 

\subsection{GKZ system and an 
isomorphism of $D$-modules} 
\label{subsec:GKZ} 
The mirror theorem $I=J$ implies that 
the B-model $D$-module is isomorphic to 
the A-model $D$-module (quantum $D$-module) 
pulled back by the mirror map $\tau$. 
The $I$-function generates 
a confluent version of the 
Gelfand-Kapranov-Zelevinsky (GKZ) $D$-module 
\cite{GKZ:hypergeom} 
studied by Adolphson \cite{adolphson}.  
This turns out to be isomorphic 
to the B-model $D$-module. 

Set $\partial_a := q_a (\partial/\partial q_a)$. 
We write $q^\pm, z\partial$ as shorthand 
for $q_1^\pm,\dots,q_r^\pm$ and 
$z\partial_1,\dots, z\partial_r$. 
Introduce a differential operator $\cP_d\in 
\C[z,q^\pm]\langle z\partial \rangle$ 
for $d\in \bL$ as  
\[
\cP_d:=
q^d \prod_{i:\pair{D_i}{d}<0} 
\prod_{\nu =0}^{-\pair{D_i}{d}-1} (\cD_i - \nu z) 
- \prod_{i: \pair{D_i}{d}>0} 
\prod_{\nu =0}^{\pair{D_i}{d}-1} (\cD_i- \nu z).  
\]
Here we put $\cD_i := \sum_{a=1}^r \sfm_{ia} z \partial_a$.    
Note that $\cP_d$ is well-defined
since $\pair{D_i}{d}\in \Z$ when $d\in \bL$. 
Define the GKZ $D$-module $M_{\rm GKZ}$ by 
\[
M_{\rm GKZ} := \C[z,q^\pm]
\langle z\partial \rangle 
\Big/ \sum_{d\in\bL} 
\C[z,q^\pm]\langle z\partial \rangle \cP_d. 
\]
A grading operator $\Grading$ 
on $M_{\rm GKZ}$ is defined by 
\begin{equation}
\label{eq:GKZ_grading} 
\Grading([f(z,q) (z\partial)^k]) = 
\left [ \left ( 2 |k| f + 2 z \parfrac{f}{z} + 
2 E f \right )  (z\partial)^k \right],  
\end{equation} 
where $k\in (\Z_{\ge 0})^r$ is a multi-index,  
$|k| = \sum_{a=1}^r k_a$ and 
$E = \sum_{a=1}^r \rho_a \partial_a$ is 
the Euler vector field (\ref{eq:B-model_Euler}) 
of the B-model $D$-module. 
This is well-defined because of the 
homogeneity of the relation $\cP_d$. 
Using the grading operator $\Grading$, 
we can introduce a flat connection 
$\nabla \colon M_{\rm GKZ} \to \frac{1}{z} 
M_{\rm GKZ} \otimes (\C\frac{dz}{z} \oplus 
\bigoplus_{a=1}^r \C \frac{dq_a}{q_a})$ by 
(\emph{cf.} (\ref{eq:Euler_reg}), Lemma \ref{lem:B-model_grading})  
\begin{align*}
 \nabla_a [P(z,q,z\partial)]  &:= 
 \frac{1}{z} [z\partial_a P(z,q,z\partial)],  
 \quad 1\le a\le r;  \\ 
\nabla_{z\partial_z} 
& :=  \frac{1}{2} \Grading  - \nabla_E - \frac{n}{2},   
\end{align*} 

\begin{proposition}
\label{prop:coherent_GKZ} 
The $\cO_{\cMo}[z]$-module $\tM_{\rm GKZ} := 
M_{\rm GKZ}\otimes_{\C[z,q^\pm]} \cO_{\cMo}[z]$ 
is finitely generated as an $\cO_{\cMo}[z]$-module. 
The fiber of $\tM_{\rm GKZ}$ at every point $(q,z) 
\in \cMo \times \C$ has dimension less than or equal to 
$|\bN_{\rm tor}| \times n! \Vol(\hS)$. 
\end{proposition} 
\begin{proof} 
For a differential operator 
$P = \sum_{k} P_k(z,q) (z\partial)^k 
\in \cO_{\cMo}[z]\langle z\partial \rangle$ 
of rank $s$, 
its principal symbol $\sigma(P)$ is defined to be  
$\sigma(P) :=\sum_{|k| = s} P_k(z,q) \sfp^k$ 
(the highest order term in $z\partial$), 
where $k\in (\Z_{\ge 0})^r$ is a multi-index
and $|k| = \sum_{a=1}^r k_a$. 
For example, 
\[
\sigma(\cP_d) =
\begin{cases}  
- \prod_{i: \pair{D_i}{d} >0} 
\sfw_i^{\pair{D_i}{d}} & \text{if }\pair{\hrho}{d}>0 ; \\ 
q^d \prod_{i: \pair{D_i}{d}<0} 
 \sfw_i^{-\pair{D_i}{d}} 
- \prod_{i:\pair{D_i}{d}>0} \sfw_i^{\pair{D_i}{d}} 
& \text{if } \pair{\hrho}{d} =0;  \\ 
q^d \prod_{i: \pair{D_i}{d}<0} \sfw_i^{-\pair{D_i}{d}} 
& \text{if } \pair{\hrho}{d} < 0.  
\end{cases} 
\]
Recall that 
$\sfw_i = \sum_{a=1}^r \sfm_{ia} \sfp_a$ 
and $\hrho = \sum_{i=1}^m D_i \in \bL^\vee$. 
By a standard argument, we know that  
$\tM_{\rm GKZ}$ is finitely generated as 
an $\cO_{\cMo}[z]$-module 
once we know that 
\[
B_{\rm c}(\cX):= 
\cO_{\cMo}[\sfp_1,\dots,\sfp_r]/
\langle \sigma(\cP_d)\;;\; d\in \bL \rangle   
\]
is a finitely generated $\cO_{\cMo}$-module. 
Adolphson \cite[Section 3]{adolphson} showed 
that the characteristic variety 
of the GKZ $D$-module is supported on 
the zero section when the corresponding 
Laurent polynomials $W_q$ are non-degenerate.  
Although the $D$-module in \cite{adolphson} 
is a little different from ours and 
it is assumed that $\bN$ is torsion free there, 
the same argument as in \cite[Section 3]{adolphson} shows\footnote
{Note that $\sigma(P_d)$ and $\sfw_i$ correspond 
to $\sigma(\Box_l)$ and $y_i$ in \cite{adolphson}.} that 
if $\sigma(\cP_d)(q,\sfp) =0$ for all $d\in \bL$,  
\begin{itemize} 
\item 
Either $(\sfp_1,\dots,\sfp_r) = 0$ or   
there exists a proper face $\Delta$ of $\hS$ such 
that $\sfw_i \neq 0$ if and only if $b_i \in \Delta$  
(\cite[Lemmas 3.1, 3.2]{adolphson});  

\item  In the latter case,  
$W_{q,\Delta}(y)$ has a critical point in $(\C^*)^n$  
(\cite[Lemma 3.3]{adolphson}).  
\end{itemize} 
Thus, $\sfp_1 = \dots = \sfp_r =0$ 
if $q\in \cMo$ and $\sigma(\cP_d)(q,\sfp)=0$ for all $d\in \bL$. 
By Hilbert's Nullstellensats, $\sfp_a^k$ vanishes 
in $B_{\rm c}(\cX)$ for a sufficiently big $k>0$, 
so $B_{\rm c}(\cX)$ is finitely generated 
as an $\cO_{\cMo}$-module. 

Since a coherent sheaf admitting a flat connection 
is locally free, we know that $\tM_{\rm GKZ}$ is 
locally free away from $z=0$.   
On the other hand, the restriction to $z=0$ of $\tM_{\rm GKZ}$ 
is isomorphic to the Batyrev ring: 
\[
\tM_{\rm GKZ}/z \tM_{\rm GKZ} \cong 
B(\cX) \otimes_{\C[q^\pm]} \cO_{\cMo}.  
\]
This is isomorphic to the Jacobi ring by 
Proposition \ref{prop:Jac-Bat} (i) 
and of rank $|\bN_{\rm tor}|\times n!\Vol(\hS)$ 
by Proposition \ref{prop:kouchnirenko} (iii). 
The conclusion follows from 
Nakayama's lemma. 
\end{proof} 

\begin{remark}
\label{rem:adolphson} 
The rank of the ``confluent" GKZ $D$-module was 
calculated in \cite{adolphson} 
under weaker assumptions  
(it is not assumed that $\hS$ 
contains the origin in its interior). 
Our $D$-module $M_{\rm GKZ}$ is a dimensional reduction 
of the original GKZ system 
in \cite{GKZ:hypergeom,adolphson} and 
is also referred to as the \emph{Horn system}. 
It is also homogenized by $z$. 
The argument above is an adaptation 
(and a shortcut) of \cite{adolphson} to 
our $D$-module $M_{\rm GKZ}$. 
We will see in the proof of Proposition \ref{prop:Dmoduleiso} 
that $\tM_{\rm GKZ}$ is exactly of rank 
$|\bN_{\rm tor}|\times n! \Vol(\hS)$. 
\end{remark} 

\begin{lemma} 
\label{lem:GKZ-ann-I} 
Assume that $\hrho \in \cl(\tC_\cX)$. 
Then the $I$-function and the oscillatory integrals 
(associated to the LG model in Section 
\ref{subsec:LGmodel}) satisfy the 
GKZ-type differential equations: 
\[ 
\cP_d I(q,z) = \cP_d \left(
\int_\Gamma e^{W_q/z} \omega_q \right) = 0, 
\quad d\in \bL,  
\] 
where $\Gamma$ is an arbitrary Lefschetz thimble. 
\end{lemma} 
\begin{proof} 
We use the expression (\ref{eq:I-funct_another}) 
of the $I$-function. Put 
\[
\Box_d := \prod_{i=1}^m 
\frac{\prod_{\nu = \ceil{\pair{D_i}{d}} }^{\infty} 
(\ov{D}_i + (\pair{D_i}{d}-\nu) z)}  
{\prod_{\nu =0}^\infty 
(\ov{D}_i + (\pair{D_i}{d}-\nu) z)}, \quad d\in \bL\otimes \Q.  
\]
Using $\cD_i (e^{\sum_{a=1}^r \ov{p}_a \log q_a/z}q^{\delta})   
= e^{\sum_{a=1}^r \ov{p}_a \log q_a/z} 
q^{\delta} (\ov{D}_i+\pair{D_i}{\delta}z)$, 
one finds that $\cP_d I(q,z) =0$ for $d\in \bL$ 
is equivalent to the difference equation: 
\[
\Box_{\delta -d} 
\prod_{i:\pair{D_i}{d}<0} \prod_{\nu=0}^{-\pair{D_i}{d}-1} 
(\ov{D}_i+(\pair{D_i}{\delta}-\nu)z)  
= \Box_{\delta} \prod_{i:\pair{D_i}{d}>0}\prod_{\nu=0}^{\pair{D_i}{d}-1} 
(\ov{D}_i+(\pair{D_i}{\delta}-\nu)z)    
\]
for all $\delta \in \K$. This is easy to check. 

We omit the proof for oscillatory integrals 
since it is completely parallel to the case 
of toric manifolds (see \emph{e.g.} 
\cite[Proposition 5.1]{iritani-coLef}).  
\end{proof}

\begin{lemma} 
\label{lem:diffI_twsector} 
For $\delta\in \K$ such that 
$\pair{D_i}{\delta}>0$ for all $i$, we have 
\[
q^{-\delta} \left(\prod_{i=1}^m 
\prod_{\nu=0}^{\ceil{\pair{D_i}{\delta}}-1}  
(\cD_i - \nu z) \right) I(q,z) 
= e^{\sum_{a=1}^r \ov{p}_a \log q_a/z} 
(\unit_{v(\delta)} + O(q^{1/e_0}))   
\]
for $e_0\in \N$ satisfying $e_0 \K \subset \bL$. 
\end{lemma} 
\begin{proof}
Using the expression (\ref{eq:I-funct_another}), 
we find that the left-hand side is 
\[
e^{\sum_{a=1}^r \ov{p}_a \log q_a/z} 
\sum_{d\in \K} q^{d-\delta} \prod_{i=1}^m 
\frac{\prod_{\nu= \ceil{\pair{D_i}{d}}}^\infty 
(\ov{D}_i+(\pair{D_i}{d}-\nu)z) } 
{\prod_{\nu = \ceil{\pair{D_i}{\delta}}}^\infty 
(\ov{D}_i+(\pair{D_i}{d}-\nu)z) } 
\unit_{v(d)}. 
\]
We claim that the summand vanishes when 
$\pair{p_a}{d-\delta}<0$ for some $a$.  
Note that there remains a factor 
$(\prod_{i: \pair{D_i}{d}\in \Z, 
\pair{D_i}{d}<\pair{D_i}{\delta}} 
\ov{D}_i) \unit_{v(d)}$ 
in the numerator.  
Thus by (\ref{eq:coh_presentation}), 
it suffices to show that 
$I:=\{i\;;\; \pair{D_i}{d}\in \Z, \ 
\pair{D_i}{d-\delta}\ge 0\}\notin \cA$. 
Suppose $I\in \cA$. 
Because $p_a\in \cl(\tC_\cX)$, there exists $c_i\ge 0$ 
for $i\in I$ such that $p_a = \sum_{i\in I} c_i D_i$ 
by the definition of $\tC_\cX$.  
Then $\pair{p_a}{d-\delta} = \sum_{i\in I} c_i \pair{D_i}{d-\delta}
\ge 0$. This is a contradiction. 
\end{proof} 

By the condition (C) in Section \ref{subsubsec:def_toricorbifolds}, 
for each $v\in \Boxop$, there exists $\delta\in \K$ 
such that $v(\delta) = v$ and $\pair{D_i}{\delta}>0$ for all $i$.  
Thus by Lemma \ref{lem:diffI_twsector} 
and the presentation (\ref{eq:toricorbcoh}), 
(\ref{eq:coh_presentation}) of $H_{\rm orb}^*(\cX)$, 
we can find differential operators 
$P_i(z,q,z\partial) \in 
\C[z,q^{\pm1/e_0}]\langle z\partial \rangle$, $1\le i\le N$ 
such that 
\begin{equation}
\label{eq:diffI_asymp} 
P_i(z,q,z\partial) I(q,z) = 
e^{\sum_{a=1}^r \ov{p}_a \log q_a/z} (\phi_i + O(q^{1/e_0})), 
\end{equation} 
where $\phi_i$, $1\le i\le N$ 
is a basis of $H_{\rm orb}^*(\cX)$. 
Under Conjecture \ref{conj:mirrorthm}, we have 
\begin{equation}
\label{eq:diffI_Birkhoff}
P_i(z,q,z\partial) I(q,z) 
= P_i(z,q,z\partial) J(\tau(q),z) 
= L(\tau(q),z)^{-1} 
P_i(z,q,z \tau^*\nabla) \unit. 
\end{equation} 
Here $L(\tau,z)$ is the fundamental solution 
in (\ref{eq:fundamentalsol_L}) and 
$\nabla$ is the Dubrovin connection: 
$\tau^*\nabla$ is shorthand for 
$\tau^*\nabla_1,\dots,\tau^*\nabla_r$ 
with $\tau^*\nabla_a :=  
\nabla_{\tau_*(q_a (\partial /\partial q_a))}$.  
Since $L(\tau(q),z)^{-1} = \unit + O(z^{-1})$ 
(regular at $z=\infty$) 
and $P_i(z,q,z\tau^*\nabla)\unit$ is 
regular at $z=0$, 
the equation (\ref{eq:diffI_Birkhoff}) 
can be viewed as the Birkhoff factorization 
(see \emph{e.g.} \cite{pressley-segal}) 
of the element 
\[
S^1 \ni z \longmapsto 
\begin{bmatrix} 
\vert &   & \vert \\
P_1 I & \dots & P_N I \\
\vert &   & \vert 
\end{bmatrix} 
\ \sim \ e^{\sum_{a=1}^r \ov{p}_a\log q_a/z} 
(\unit + O(q^{1/e_0})) 
\]
in the loop group $LGL(N,\C)$.   
Here the asymptotics (\ref{eq:diffI_asymp}) 
show that the matrix $[P_1 I,\dots, P_N I]$ is invertible and 
admits the (unique) Birkhoff factorization\footnote
{The convergence of quantum cohomology is not a priori known. 
However the Birkhoff factorization here 
can be done uniquely over the ring of 
formal power series in $q_1^{1/e_0},\dots, q_r^{1/e_0}$ 
after removing the factor $e^{\sum_{a=1}^r\ov{p}_a \log q_a/z}$.   
See \cite[Theorem 3.9]{iritani-efc}.} 
when $|q_a|$ is sufficiently small. 
In particular, it follows that the fundamental 
solution $L(\tau(q),z)$ is analytic 
for small values of $|q_a|$ 
and that the quantum cohomology/$D$-module is convergent 
over the image of $\tau$. 
Note that by (\ref{eq:diffI_asymp}), we have 
\begin{equation}
\label{eq:diffunit_asymp} 
P_i(z,q,z \tau^*\nabla) \unit = \phi_i + O(q^{1/e_0}) 
\end{equation} 
and that these vectors form a basis 
of $H^*_{\rm orb}(\cX)$ for small $|q_a|$. 

Now we formulate toric mirror symmetry 
as an isomorphism of $D$-modules. 
\begin{proposition} 
\label{prop:Dmoduleiso}
Assume that our initial data satisfies 
$\hrho \in \cl(\tC_\cX)$ 
and that Conjecture \ref{conj:mirrorthm} holds for $\cX$. 
The B-model $D$-module (in Definition \ref{def:BDM}) 
is isomorphic to the pull back  
of the A-model $D$-module (in Definition \ref{def:QDM}) 
under the mirror map $\tau$ in (\ref{eq:mirrormap_quot}): 
\[
\Mir \colon 
(\cRz,\nabla,(\cdot,\cdot)_{\cRz})\Big|_{V_\epsilon\times \C} 
\cong (\tau\times \id)^* ((F, \nabla, (\cdot,\cdot)_F)/H^2(\cX,\Z)) 
\] 
where $V_\epsilon = \{(q_1,\dots, q_r)\in \cM\;;\; 
0<|q_a|<\epsilon \}$ and $\epsilon>0$ 
is a sufficiently small real number. 
The right-hand side is the quotient 
by the Galois action. 
The isomorphism $\Mir$ sends $[e^{W_q/z}\omega_q]$ 
to the unit section $\unit$ of $F$. 
\end{proposition} 
\begin{proof}
First we identify the GKZ $D$-module with 
the A-model $D$-module. 
Consider a $D$-module homomorphism: 
\begin{align}
\label{eq:GKZ_Amodel} 
\begin{split}  
M_{\rm GKZ} \otimes_{\C[z,q^\pm]} 
\cO_{V_\epsilon\times \C}&\longrightarrow 
\cO((\tau\times \id)^* (F/H^2(\cX,\Z)) \\ 
[P(z,q,z\partial)]  &\longmapsto 
P(z,q, z\tau^*\nabla) \unit 
\end{split}
\end{align} 
We claim that this map is an isomorphism. 
By Lemma \ref{lem:GKZ-ann-I} and (\ref{eq:diffI_Birkhoff}), 
this map is well-defined. 
The equation (\ref{eq:diffunit_asymp}) shows 
that this is surjective for some 
small $\epsilon>0$. 
By Lemma \ref{lem:qsmall_kouchnirenko}, 
we may assume $V_\epsilon \subset \cMo$. 
Then we can deduce the claim by comparing 
the ranks (Proposition \ref{prop:coherent_GKZ} and  
Lemma \ref{lem:rankmatch}). 
Next consider a $D$-module homomorphism: 
\begin{align}
\label{eq:GKZ_Bmodel}
\begin{split} 
M_{\rm GKZ} \otimes_{\C[z,q^\pm]} 
\cO_{V_\epsilon\times \C} 
& \longrightarrow \cRz|_{V_\epsilon \times \C}\\ 
[P(z,q,z\partial)] & \longmapsto 
P(z,q,z\nabla) [e^{W_q/z} \omega_q] 
\end{split} 
\end{align} 
where $\nabla$ is the flat connection 
of the B-model $D$-module. 
This is well-defined by Lemma \ref{lem:GKZ-ann-I} 
and surjective by Proposition \ref{prop:BDM_diffgen}. 
Thus it is an isomorphism again 
by comparison of the ranks 
(Propositions \ref{prop:locsys_Lefschetz} 
and \ref{prop:coherent_GKZ}). 
By composing the two isomorphisms 
(\ref{eq:GKZ_Amodel}), (\ref{eq:GKZ_Bmodel}), 
we get the desired isomorphism 
$\Mir\colon \cRz|_{V_\epsilon\times \C} 
\cong \cO((\tau\times \id)^*(F/H^2(\cX,\Z)))$ 
sending $[e^{W_q/z}\omega_q]$ to $\unit$. 

It is clear that $\nabla_a=\nabla_{q^a(\partial /\partial q_a)}$ 
corresponds to $\tau^*\nabla_a$ under the map $\Mir$. 
It is easy to check that the isomorphisms 
(\ref{eq:GKZ_Amodel}) and (\ref{eq:GKZ_Bmodel}) 
preserve the grading operators (see (\ref{eq:Euler_reg}), 
(\ref{eq:GKZ_grading}) and (\ref{eq:B-model_grading}); 
we use the homogeneity of the series 
$e^{-\sum_{a=1}^r \ov{p}_a\log q_a/z}I(q,z)$). 
Hence $\Mir$ preserves $\Grading$ 
and so sends $\nabla_{z\partial_z}$ to 
$\tau^*\nabla_{z\partial_z}$  
(we use the fact that 
$\tau$ preserves the Euler vector field). 

The proof of $(\cdot,\cdot)_{\cRz} 
= (\tau\times \id)^*(\cdot,\cdot)_F$ is 
given in Appendix \ref{subsec:pairing}. 
\end{proof} 

\begin{corollary}
Under the same assumptions 
as Proposition \ref{prop:Dmoduleiso}, 
the quantum cohomology of a toric orbifold 
$\cX$ is generically semisimple, \emph{i.e.} 
$(H_{\rm orb}^*(\cX),\circ_\tau)$ is isomorphic to 
the direct sum of $\C$ as a ring for a generic 
$\tau\in U$. 
\end{corollary} 
\begin{proof}
The quantum cohomology of $\cX$ 
is identified with the Jacobi ring $J(W_q)$ of the mirror. 
The conclusion follows from Proposition \ref{prop:Jac-Bat} (ii). 
\end{proof} 

\begin{remark} 
When $\cX$ is not weak Fano, 
the mirror theorem Conjecture \ref{conj:mirrorthm} 
should be replaced with the Coates-Givental \cite{coates-givental} 
style statement that the $I$-function 
is on the Givental's Lagrangian cone 
(\ref{eq:Giventalcone}).  
The $D$-module isomorphism cannot hold since 
the ranks are different ($|\bN_{\rm tor}|\times n!\Vol(\hS) 
> \dim H_{\rm orb}(\cX)$), but the quantum $D$-module 
should be isomorphic to a certain completion 
of the GKZ $D$-module at the large radius limit
$q=0$ and the semisimplicity of quantum cohomology should still hold. 
The details will appear in \cite{CCIT:toric}. 
(See \cite{iritani-genmir,iritani-coLef} 
for toric manifolds.) 
\end{remark}

\subsection{The integral structures match} 
\begin{theorem} 
\label{thm:pulledbackintstr} 
Let $\cX$ be a weak Fano projective 
toric orbifold defined by 
initial data satisfying $\hrho\in \cl(\tC_\cX)$. 
Assume that Conjecture \ref{conj:mirrorthm} and 
Assumption \ref{assump:Ktheory}, (c) 
hold for $\cX$. 
Then the mirror isomorphism $\Mir$ in Proposition 
\ref{prop:Dmoduleiso} sends 
the natural integral structure 
(lattice of Lefschetz thimbles) 
of the B-model $D$-module to the $\hGamma$-integral structure 
(Definition \ref{def:A-model_int})  
of the A-model $D$-module. 
\end{theorem} 

First we draw a corollary on Dubrovin's conjecture 
\cite[4.2.2]{dubrovin-analyticth} from this theorem. 
Since the $\hGamma$-integral structure  
is defined to be the image of the $K$-group, 
we can identify 
the integral lattice $R_{\Z,(q,z)}^\vee$ 
generated by Lefschetz thimbles 
with (the dual\footnote{We identify the dual of 
the $K$-group with the $K$-group itself by the Mukai pairing.} 
of) the $K$-group $K(\cX)$. 
This also identifies the pairings on the both sides.  
Let $V_1,\dots,V_N \in K(\cX)$ correspond to 
a basis $\Gamma_1,\dots,\Gamma_N$ of Lefschetz thimbles
whose images under $W_q$ are straight half-lines. 
Then we have 
\[
\chi(V_i^\vee \otimes V_j) = 
\sharp(\Gamma_i \cap e^{\pi\iu} \Gamma_j), 
\]
where $e^{\pi\iu}\Gamma_j$ is the parallel 
translate of $\Gamma_j\in H_n(Y_q,\{y\;;\;\Re(W_q(y)/z)\ll 0\})$ 
along the path 
$[0,1]\ni \theta \mapsto e^{\pi\iu\theta} z$ 
(\emph{cf.} (\ref{eq:pairing_Sol})).  
On the other hand,  
the quantum differential equation in $z$  
\begin{equation}
\label{eq:qde_z} 
\nabla_{z\partial_z} \psi(z) = 
\left( 
z\parfrac{}{z} - \frac{1}{z} E\circ  
+ \mu \right) \psi(z) = 0
\end{equation} 
is irregular singular at $z=0$ and 
defines a \emph{Stokes matrix} 
(see \cite{dubrovin-analyticth,dubrovin-painleve}). 
Under mirror symmetry, the Stokes matrix 
is given by the intersection numbers 
$\sharp(\Gamma_i \cap e^{\pi\iu} \Gamma_j)$ 
by Picard-Lefschetz theory 
(since a solution $\psi$ is given by 
oscillatory integrals over $\Gamma_i$'s; 
see \emph{e.g.} \cite{cecotti-vafa-classification, ueda-cubic}).  
Hence, 

\begin{corollary}[$K$-group version of Dubrovin's conjecture]  
\label{cor:K_Dubrovin} 
Under the same assumptions 
as Theorem \ref{thm:pulledbackintstr},  
there exist $V_1,\dots, V_N \in K(\cX)$ 
such that the matrix $S=(S_{ij})$, 
$S_{ij} := \chi(V_i^\vee \otimes V_j)$ 
is a Stokes matrix of the quantum 
differential equation of $\cX$. 
(In particular, $S$ is upper-triangular and 
$S_{ii}=1$.) 
\end{corollary}
 
\begin{remark} 
Dubrovin's conjecture \cite{dubrovin-analyticth} 
furthermore asserts that $V_1,\dots,V_N$ here should come 
from an \emph{exceptional collection} in the derived category. 
This should follow from homological mirror symmetry. 
For toric varieties, different versions 
of homological mirror symmetry have been obtained 
(or announced) by Abouzaid \cite{abouzaid}, 
Fang-Liu-Treumann-Zaslow \cite{FLTZ} and Bondal-Ruan 
\cite{bondal-ruan}. The author is not sure 
if their results imply Dubrovin's conjecture 
since, except for the approach by Bondal-Ruan, 
they do not deal with Lefschetz thimbles directly. 
For a weighted projective space $\cX$,  
$\Gamma_1,\dots,\Gamma_N$ are the monodromy 
transforms (in $q$) of 
the real Lefschetz thimble $\Gamma_\R$ (see 
Theorem \ref{thm:cc_match} below),  
so these actually correspond to 
an exceptional collection 
$\cO(-a),\cO(-a+1),\dots, \cO(b)$ for some $a,b$. 
(Dubrovin's conjecture for $\cX=\Proj^n$ 
was proved by Guzzetti \cite{guzzetti}.) 
For general $\cX$, it might be difficult to calculate 
$V_i$ corresponding to $\Gamma_i$ whose image under 
$W_q$ is a straight half-line. 
\end{remark} 

Theorem \ref{thm:pulledbackintstr} 
follows from the matching of the central charges 
from quantum cohomology and LG model. 
Consider the fibration formed by real points on 
(\ref{eq:fibration_LG}):  
\[
\begin{CD}
\unit @>>> \Hom(\bN,\R_{>0}) @>>> Y_\R := (\R_{>0})^m 
@>{\pr|_{Y_\R}}>> \cM_\R :=\Hom(\bL,\R_{>0}) @>>> \unit. 
\end{CD}
\]
Here we regard $\R_{>0}$ as an abelian group 
with respect to the multiplication.  
This exact sequence splits and the section 
given by the matrix $(\ell_{ia})$ 
in Section \ref{subsubsec:LGmodel_def} 
is single-valued over the real locus $\cM_\R$. 
For $q\in \cM_\R$, the real Lefschetz thimble 
$\Gamma_\R\subset Y_q$ is defined to be 
\[
\Gamma_\R:=Y_q\cap Y_\R 
= \{(y_1,\dots,y_n)\in Y_q \;;\; y_i >0 \}
\cong \Hom(\bN,\R_{>0}). 
\] 
The oscillatory integral 
$\int_{\Gamma_\R} e^{-W_q/z} \omega_q$ is well-defined 
for $q\in \cM_\R$ and $z>0$. 
We also define $\Gamma_{\rm c} \subset Y_q$ 
to be the parallel translate of the 
monodromy-invariant compact cycle 
\[
\Gamma_{\rm c} := \Hom(\bN,S^1) \subset Y_{q=1}. 
\]
Note that $\Gamma_{\rm c}$ is a disjoint union 
of $|\bN_{\rm tor}|$ number of tori $(S^1)^n$. 

\begin{theorem}
\label{thm:cc_match}
Assume that $\hrho \in \cl(\tC_{\cX})$ and that 
Conjecture \ref{conj:mirrorthm} holds. 
The quantum cohomology central charges  
(\ref{eq:qc_centralcharge}) 
of the structure sheaf $\cO_\cX$ and 
the skyscraper sheaf $\cO_\pt$ 
are given by the oscillatory integrals 
over the real Lefschetz thimble $\Gamma_\R$ 
and the compact cycle $\Gamma_{\rm c}$ respectively:  
\begin{align}
\label{eq:cc_str} 
Z(\cO_\cX)(\tau(q),z) &= \frac{1}{(2\pi \iu)^n}
\int_{\Gamma_\R\subset Y_q} e^{-W_q/z} \omega_q,  
\quad q\in \cM_\R, z>0; 
\\ 
\label{eq:cc_sky} 
Z(\cO_\pt)(\tau(q),z) 
& = \frac{1}{(2\pi\iu)^n}
\int_{\Gamma_{\rm c}\subset Y_q} e^{-W_q/z} \omega_q, 
\quad (q,z)\in \cM \times \C^*,  
\end{align} 
where $\tau(q)$ is the mirror map. 
In the equation (\ref{eq:cc_str}), the branches of 
$\log z$, $\tau(q)$ in the definition of the left-hand side is 
chosen so that $\log z \in \R, 
\tau(q) \in H^{\le 2}_{\rm orb}(\cX,\R)$. 
\end{theorem} 

The right-hand sides of (\ref{eq:cc_str}), 
(\ref{eq:cc_sky}) are considered as 
the \emph{LG central charges}  
(called \emph{BPS mass} in \cite{hori-vafa}) 
of $\Gamma_\R$ and $\Gamma_{\rm c}$.  
This corresponds to a compact toric  
version of Hosono's conjecture 
\cite[Conjecture 2.2]{hosono}, 
which was was stated 
for Calabi-Yau complete intersections 
in terms of hypergeometric series 
(in place of $Z(V)$) 
and periods (in place of oscillatory integrals). 

\begin{remark}
(i) The equality (\ref{eq:cc_str}) 
of central charges solves a connection 
problem for the quantum differential equation  
(\ref{eq:qde_z}) in $z$ which is regular singular at $z=\infty$ 
and irregular singular at $z=0$. 
The oscillatory integral admits an asymptotic expansion 
at $z=0$ and $Z(\cO_\cX)$ is (by definition)  
expanded in a power series in $z^{-1}$. 

(ii) This theorem suggests that, under homological mirror symmetry, 
the thimble $\Gamma_\R$ (or $\Gamma_{\rm c}$)   
(an object of Fukaya-Seidel category of the LG model), 
should correspond to the structure sheaf 
$\cO_\cX$ (or $\cO_\pt$) 
(an object of the derived category of coherent sheaves on $\cX$).  
This correspondence is consistent with the 
Strominger-Yau-Zaslow (SYZ) picture \cite{SYZ}. 
The cycle $\Gamma_\R$ (resp. $\Gamma_{\rm c}$) 
gives a Lagrangian section (resp. fiber) of 
the SYZ fibration, so should correspond 
to the structure (resp. skyscraper) sheaf. 
\end{remark}

\subsubsection{Proof of Theorem \ref{thm:pulledbackintstr} 
under Theorem \ref{thm:cc_match}} 
\label{subsubsec:pr_main}
Fix a point $q\in \cM_\R$ and $z>0$.  
The mirror isomorphism $\Mir$ in Proposition 
\ref{prop:Dmoduleiso} defines a map 
\[
R^\vee_{(q,-z)} = 
H_n(Y_q,\{y\in Y_q\;;\;\Re(W_q(y)/(-z))\ll 0\}) 
\to \Sol(\cX), \quad 
\Gamma \mapsto s_\Gamma(\tau,z) 
\]
such that 
\[
\left(\Mir[\varphi],  s_\Gamma(\tau(q),z)\right)_{\rm orb} = 
\pair{[\varphi]}{\Gamma}, \quad 
\forall [\varphi] \in \cRz_{(q,-z)}, 
\]
where the right-hand side is 
the pairing in (\ref{eq:pairing_R_Rvee}) 
and $\log z$ and $\tau(q)$ in the left-hand side 
are taken to be real as above. 
Let $\tSol(\cX)_\Z$ be the image of this map. 
We need to show that $\tSol(\cX)_\Z$ coincides 
with the $\hGamma$-integral structure $\Sol(\cX)_\Z$. 
From the definition (\ref{eq:qc_centralcharge}) 
of $Z(\cO_\cX)$, one can rewrite (\ref{eq:cc_str}) as  
\[
(\unit, \cZ_K(\cO_\cX)(\tau(q),z) )_{\rm orb} = 
\pair{[e^{-W_q/z} \omega_q]}{\Gamma_\R}. 
\]   
Because $\Mir$ sends 
$[e^{-W_q/z}\omega_q]\in \cRz_{(q,-z)}$ 
to $\unit \in F_{(\tau(q),-z)}$ and the B-model 
$D$-module is generated by $[e^{-W_q/z}\omega_q]$ 
and its derivatives (Proposition \ref{prop:BDM_diffgen}), 
we have $\cZ_K(\cO_\cX) = s_{\Gamma_\R} \in \tSol(\cX)_\Z$. 
Because $K(\cX)$ is generated by line bundles \cite{borisov-horja-K} 
and $\tSol(\cX)_\Z$ is preserved by the Galois action, 
we have $\Sol(\cX)_\Z = \cZ_K(\Z[\Pic(\cX)]\cO_\cX) 
\subset \tSol(\cX)_\Z$. 
Because the pairing of the A-model and B-model 
coincide,  $\tSol(\cX)_\Z$ is a unimodular 
lattice in $\Sol(\cX)$. 
Under Assumption \ref{assump:Ktheory} (c), 
$\Sol(\cX)_\Z$ is also unimodular. 
Therefore $\Sol(\cX)_\Z = \tSol(\cX)_\Z$.

\subsection{Equivariant perturbation}
Here we prove Theorem \ref{thm:cc_match}. 
We will make use of Givental's \emph{equivariant mirror}  
which gives a perturbation of oscillatory integrals. 
This is considered as a mirror of 
equivariant quantum cohomology of toric orbifolds.  
We prove an equivariant version of (\ref{eq:cc_str}) 
and conclude (\ref{eq:cc_str}) by taking 
the non-equivariant limit. 
In this article, we do not formulate 
equivariant mirror symmetry.

\subsubsection{Equivariant oscillatory integrals} 
Let $T:=(\C^*)^m$ act on our toric orbifold $\cX=\C^m/\!/\T$ 
via the diagonal action of $(\C^*)^m$ on $\C^m$. 
Let $-\lambda_1,\dots,-\lambda_m$ be the equivariant variables 
corresponding to generators of $H^*_{T}(\pt)$. 
Here $\lambda_i$ denotes either 
a cohomology class or a complex number 
depending on the context. 
Givental's equivariant mirror \cite{givental-mirrorthm-toric} 
is given by the following perturbed potential $W^{\lambda}$: 
\[
W^{\lambda} := \sum_{i=1}^m (w_i + \lambda_i \log w_i) 
= W+ \sum_{i=1}^m \lambda_i \log w_i.   
\] 
Hereafter $\lambda_i$ denotes a complex number. 
This is a multi-valued function on each fiber $Y_q$. 
Morse theory for $\Re(W^\lambda(y)/z)$ will compute 
relative homology with coefficients in some local system. 
For a cycle $\Gamma\subset Y_q$ in such a relative homology, 
we can define the \emph{equivariant oscillatory integral}: 
\[
\int_{\Gamma} e^{W^{\lambda}/z} \omega_q 
= \int_{\Gamma} e^{W/z} \prod_{i=1}^m w_i^{\lambda_i/z} \omega_q.  
\]
For our purpose, it is more convenient to 
use the exponent $\lambda_i/(2\pi\iu)$ instead of  $\lambda_i/z$.  
Define 
\begin{equation}
\label{eq:equiv_osc_int}
\cI^\lambda_\Gamma(q,z) := 
\frac{1}{(2\pi\iu)^{n}} 
\int_{\Gamma} e^{\frac{w_1+\dots+w_m}{z}} 
\prod_{i=1}^m w_i^{\frac{\lambda_i}{2\pi\iu}} \omega_q. 
\end{equation} 
Again, the equivariant oscillatory integral 
$\cI^\lambda_{\Gamma_\R}(q,-z)$ for the real Lefschetz 
thimble $\Gamma_\R$ is well-defined when $q\in \cM_\R$ 
and $z>0$. 

\subsubsection{Equivariant $H$-function}
Recall that the quantum cohomology central charge can be 
written in terms of the $H$-function (\ref{eq:H-funct}) 
(see (\ref{eq:cc_byH})). 
Under the mirror theorem, we can write the $H$-function 
as a hypergeometric series with coefficients given by 
products of Gamma functions. 
This type of hypergeometric series 
has been used by Horja \cite{horja1},  
Hosono \cite{hosono} and Borisov-Horja 
\cite{borisov-horja-FM}\footnote
{We named it after Horja and Hosono.}. 

By abuse of notation, we write $H(q,z):= H(\tau(q),z)$. 
Using Gamma functions, we can write 
the $I$-function (\ref{eq:I-funct_another}) as   
\[
I(q,z) = e^{\sum_{a=1}^{r} \ov{p}_a \log q_a/z} 
\sum_{d\in \K_{\rm eff}} 
\frac{q^d}{z^{\pair{\hrho}{d}}} 
\prod_{i=1}^m 
\frac{\Gamma( 1- \{-\pair{D_i}{d}\} + \ov{D}_i/z)}
     {\Gamma( 1+ \pair{D_i}{d}+ \ov{D}_i/z) }
     \frac{\unit_{v(d)}}{z^{\iota_{v(d)}}}. 
\]
Using this expression and Conjecture \ref{conj:mirrorthm}, 
we calculate the $H$-function (\ref{eq:H-funct}) as 
\begin{align}
\nonumber  
H(q,z) &= (-1)^n z^{n/2} \inv^* (2\pi\iu)^{-\deg/2} 
\hGamma(T\cX)^{-1} z^{-\rho}z^\mu I(q,z) \\ 
\label{eq:H-series} 
& = (-1)^n \sum_{d\in \K_{\rm eff}}
x^{\frac{\ov{p}}{2\pi\iu}+d} 
\frac{\unit_{\inv(v(d))}} 
{\prod_{i=1}^m \Gamma ( 
1+ \pair{D_i}{d}+ \frac{\ov{D}_i}{2\pi \iu})},   
\end{align} 
where we used the fact that 
the $v(d)$-component of $\hGamma(T\cX)$  
(for $d\in \K_{\rm eff}$) 
is given by $\prod_{i=1}^m \Gamma(1 - \{-\pair{D_i}{d}\}+\ov{D}_i)$ 
and set 
\[
x^{\frac{\ov{p}}{2\pi\iu} +d} 
:= e^{\sum_{a=1}^{r} 
(\frac{\ov{p}_a}{2\pi\iu} + \pair{p_a}{d}) \log x_a}, 
\quad
\log x_a := \log q_a -\rho_a \log z \quad 
\text{(\emph{i.e.} $x_a = \frac{q_a}{z^{\rho_a}}$)}. 
\]  
We introduce $T$-equivariant $I$- and $H$-functions. 
As in Section \ref{subsubsec:KC_nefbasis}, 
$\xi\in \bL^\vee$ defines the 
orbifold line bundle $L_\xi$ on $\cX$:  
\[
L_\xi = \cU_\eta \times \C \Big/
(z_1,\dots,z_m,c) 
\sim (t^{D_1}z_1,\dots,t^{D_m}z_m,t^\xi c), \ t\in \T.   
\] 
The line bundle $L_\xi$ admits a canonical $T$-action: 
$T=(\C^*)^m$ acts diagonally on the first factor 
and the trivially on the second factor.  
By taking the $T$-equivariant first Chern class, 
we can associate to every element $\xi \in \bL^\vee$ 
an equivariant class $c_1^{T}(L_\xi)\in H^2_T(\cX)$. 
We denote by $\ov{p}^\lambda_1,\dots,
\ov{p}^\lambda_r\in H_T^2(\cX)$ 
the $T$-equivariant cohomology classes 
corresponding to $p_1,\dots,p_r\in \bL^\vee$. 
Note that $\ov{p}^\lambda_{r'+1},\dots,\ov{p}^\lambda_r$ 
may be non-zero. 
We denote by $\ov{D}^\lambda_i\in H_T^2(\cX)$ the 
\emph{$T$-equivariant Poincar\'{e} dual}  
of the toric divisor $\{z_i=0\}$. 
Note that $\ov{D}^\lambda_j=0$ for $j>m'$ 
even in equivariant cohomology 
(since $\{z_j=0\}$ is empty). 
When $e^{-\lambda_i}$ denotes 
the 1-dimensional $T$-representation given by the 
$i$-th projection $T \to \C^*$, 
the divisor $\{z_i=0\}$ becomes the zero-locus 
of a $T$-equivariant section of 
$L_{D_i} \otimes e^{-\lambda_i}$. 
Thus we have (\emph{cf.} (\ref{eq:Dprel})) 
\begin{equation}
\label{eq:Dprel_equiv}
\ov{D}^\lambda_i = \sum_{a=1}^r \sfm_{ia} 
\ov{p}^\lambda_a - \lambda_i  \quad 
\text{in $H_T^2(\cX)$.}  
\end{equation} 
The equivariant $I$-function is defined by the same formula 
in Definition \ref{def:I-funct} with all the appearance of 
$\ov{p}_a$, $\ov{D}_j$ replaced by 
$\ov{p}^\lambda_a$, $\ov{D}^\lambda_j$. 
The equivariant $H$-function 
$H^\lambda(q,z)$ is defined\footnote 
{The factor $z^{-\frac{\lambda_1+\dots+ \lambda_m}{2\pi\iu}}$ 
comes from the $T$-equivariant first Chern class  
$c_1^T(T\cX) = \sum_{a=1}^{r'} \rho_a \ov{p}^\lambda_a 
- (\lambda_1+\dots+ \lambda_m)$.} to be: 
\begin{equation}
\label{eq:equiv_H}
H^{\lambda}(q,z) := (-1)^n 
z^{-\frac{\lambda_1+\dots+\lambda_m}{2\pi\iu}} 
\sum_{d\in \K_{\rm eff}} x^{\frac{\ov{p}^\lambda}{2\pi\iu}+d} 
\frac{\unit_{\inv(v(d))}}
{\prod_{i=1}^m \Gamma (1+ \pair{D_i}{d}+ 
\frac{\ov{D}^\lambda_i}{2\pi\iu})}. 
\end{equation} 
We regard the equivariant $I$- and $H$-functions 
as functions taking values in 
$H_{{\rm orb},T}^*(\cX)$ and $H^*_{T}(I\cX)$ respectively.  
(Here $H_{{\rm orb},T}^*(\cX) := 
\bigoplus_{v\in\sfT} H^{*-2\iota_v}_T(\cX_v)$.) 
\begin{remark}
\label{rem:equiv}
The equivariant $I$- and $H$-functions 
should be understood as follows.  
For a toric orbifold $\cX$, $H_T^*(I\cX)$ is a free 
$H^*_T(\pt)=\C[\lambda_1,\dots,\lambda_m]$-module 
of rank $\dim H^*(I\cX)$. 
Thus we can regard $H_T^*(I\cX)$ as a finite-dimensional 
vector bundle over $\Spec H^*_T(\pt)$.  
The $I$-function (resp. $H$-function) makes sense as 
a multi-valued meromorphic (resp. holomorphic) 
section of the $H^*_{\rm orb}(\cX)$-bundle 
over the space $\{(q,z,\lambda) \in 
\cM \times \C^* \times \Spec H_T^*(\pt)\;;\; 0<|q_a|<\epsilon\}$. 
\end{remark}

\subsubsection{Oscillatory integral and $H$-function}
We prove a $T$-equivariant generalization 
of (\ref{eq:cc_str}). 
Since $Z(\cO_\cX)$ can be written 
in terms of the $H$-function (\ref{eq:cc_byH}), 
the following theorem proves (\ref{eq:cc_str})   
by the non-equivariant limit $\lambda_i\to 0$. 

\begin{theorem}
\label{thm:connection_cI_H} 
Assume that $\hrho\in \cl(\tC_\cX)$.  
The equivariant oscillatory integral 
(\ref{eq:equiv_osc_int}) 
and the equivariant $H$-function 
(\ref{eq:equiv_H}) are related by
\begin{align}
\label{eq:equivosc_H_str}
\cI_{\Gamma_\R}^\lambda (q,-z) &= 
\int_{I\cX} 
H^\lambda (q,e^{\pi\iu} z) \cup \tTd^\lambda(T\cX), \quad 
q\in \cM_\R, \ z>0,  
\end{align} 
where $\tTd^\lambda(T\cX)$ is the $T$-equivariant Todd class 
defined similarly to Section \ref{subsec:hGamma_intstr}. 
The branches of the logarithm in the right-hand side 
are chose so that $\log z\in \R, \log q_a \in \R$.  
\end{theorem} 

\begin{remark} 
Even if $\hrho\notin \cl(\tC_\cX)$, 
the left-hand side of (\ref{eq:equivosc_H_str}) 
makes sense as an analytic function in $q$ and $z$. 
In this case, the right-hand side could be understood as 
the asymptotic expansion in $q_1,\dots,q_r$ 
of the left-hand side in the limit $q_a \searrow +0$.  
\end{remark} 

By the localization theorem \cite{atiyah-bott} 
in equivariant cohomology, 
the inclusion $i\colon I\cX^T \rightarrow I\cX$ 
of the $T$-fixed point set $I\cX^T$ 
induces an isomorphism 
$i^* \colon H^*_T(I\cX)\otimes_{H^*_T(\pt)} \C(\lambda)
\to H^*(I\cX^T)\otimes_{H^*_T(\pt)} \C(\lambda)$, 
where $\C(\lambda)$ is the fraction field of 
$H^*_T(\pt)=\C[\lambda_1,\dots,\lambda_m]$. 
The number of fixed points in $I\cX$ 
is equal to $N: = \dim H^*_{\rm orb}(\cX)$  
(see the proof of Lemma \ref{lem:rankmatch}). 
A $T$-fixed point in $I\cX$ is labeled 
by a pair $(\sigma,v)$ of a fixed point 
$\sigma \in \cX^T$ and $v\in \Boxop$ 
such that $\sigma \in \cX_v$. 
Moreover, a fixed point $\sigma\in \cX^T$ 
is in one-to-one correspondence with 
a maximal cone of the fan $\Sigma$ 
spanned by $\{b_i\;;\; \sigma\in \{z_i=0\}\}$. 
By restricting $H^\lambda(q,z)$ to a fixed point $(\sigma,v)$, 
we get a function $H^\lambda_{\sigma,v}(q,z)$ 
in $q$, $z$ and $\lambda$. 
We call it a \emph{component} of the $H$-function.

\begin{lemma}
The equivariant $H$-function $H^\lambda(q,z)$ and 
the oscillatory integral $\cI^\lambda_{\Gamma_\R}(q,z)$ 
are solutions to the following GKZ-type 
differential equations:  
\begin{gather} 
\label{eq:equiv_GKZ} 
\cP^\lambda_d f(q,z) = 0, \quad  d\in \bL, \\  
\label{eq:equiv_GKZ_z}
\left (z\parfrac{}{z} + \sum_{a=1}^r \rho_a \partial_a
\right) f(q,z) = 
\frac{\lambda_1+\cdots+\lambda_m}{2\pi\iu} f(q,z), 
\end{gather} 
where $\partial_a := q_a (\partial/\partial q_a)$, 
\[
\cP^\lambda_d := 
q^d \prod_{\pair{D_i}{d}<0} \prod_{\nu = 0}^{-\pair{D_i}{d}-1}
(\cD^\lambda_i - \nu z) - 
\prod_{\pair{D_i}{d}>0} \prod_{\nu =0}^{\pair{D_i}{d}-1} 
(\cD^\lambda_i - \nu z ),   
\] 
and $\cD^\lambda_i := \sum_{a=1}^r \sfm_{ia} 
z \partial_a - z \lambda_i/(2\pi\iu)$. 
The $N$ components $H^\lambda_{\sigma,v}(q,z)$ 
of the $H$-function form a basis of solutions to 
these differential equations for generic 
$\lambda_i$'s and small $q_a$'s.  
\end{lemma} 
\begin{proof}
The proof here is an equivariant generalization 
of the argument in Section \ref{subsec:GKZ}. 
The proof of (\ref{eq:equiv_GKZ}) 
for $f(q,z) = H^\lambda(q,z)$ or 
$\cI_{\Gamma_\R}^\lambda(q,z)$ 
is similar to Lemma \ref{lem:GKZ-ann-I}. 
(For $H^\lambda(q,z)$, we rewrite (\ref{eq:equiv_H}) 
as a summation over $d\in \K$; 
the terms with $d\in \K\setminus \K_{\rm eff}$ 
automatically vanish by relations in $H^*_T(I\cX)$.) 
The equation (\ref{eq:equiv_GKZ_z}) 
means the homogeneity of $f(q,z)$. 
The details are left to the reader.  

In order to show that the components of 
the $H$-function form a basis of solutions, 
we consider the \emph{equivariant GKZ $D$-module}: 
\[
M_{\rm GKZ}^\lambda := \C[z,q^\pm]\langle z\partial \rangle 
\Big/
\sum_{d\in \bL} \C[z,q^\pm]\langle z\partial \rangle \cP^\lambda_d 
\]
for fixed complex numbers $\lambda_1,\dots,\lambda_m$. 
This also admits a flat connection as in Section \ref{subsec:GKZ}. 
Since the differential operator $\cP^\lambda_d$ 
has the same principal symbol as $\cP_d$ 
and $M_{\rm GKZ}^\lambda/z M_{\rm GKZ}^\lambda$ 
is independent of $\lambda$, 
the same argument as the proof 
of Proposition \ref{prop:coherent_GKZ} shows that 
$M_{\rm GKZ}^\lambda \otimes_{\C[z,q^\pm]} \cO_{\cMo\times \C^*}$ 
is locally free of rank $\le N$. 
Therefore, we have at most $N$ linearly independent 
solutions to the GKZ-system (\ref{eq:equiv_GKZ}), 
(\ref{eq:equiv_GKZ_z}). 
On the other hand, similarly to Lemma \ref{lem:diffI_twsector}, 
we can show that 
\begin{align*} 
\left(q^{-\delta} \prod_{i=1}^m 
\prod_{\nu=0}^{\ceil{\pair{D_i}{\delta}}-1} 
(\cD^\lambda_i - \nu z)\right) & H^\lambda(q,z)  
= (-1)^n z^{\frac{\lambda_1+\dots+\lambda_m}{2\pi\iu} 
+ \iota_{v(\delta)}} x^{\frac{\ov{p}^\lambda}{2\pi\iu}} \times \\  
\times & \left( \frac{\unit_{\inv(v(\delta))}}{
\prod_{i=1}^m \Gamma (1- \{-\pair{D_i}{\delta}\} 
+ \frac{\ov{D}^\lambda_i}{2\pi\iu} )} 
+ O(q^{1/e_0}) \right ) 
\end{align*} 
for $\delta\in \K$ such that $\pair{D_i}{\delta}>0$ 
for all $i$. 
Because $H^*_T(I\cX)$ is generated by $\unit_v$, $v\in \Boxop$ 
over $H^*_T(\pt)[\ov{p}^\lambda_1,\dots,\ov{p}^\lambda_{r'}]$ 
(\emph{cf.} (\ref{eq:toricorbcoh}), (\ref{eq:coh_presentation})), 
suitable derivatives of $H^\lambda(q,z)$ 
form a meromorphic basis\footnote
{Here we regard $H^*_T(I\cX)$ as a vector bundle over 
$\Spec H_T^*(\pt)$ as in Remark \ref{rem:equiv}.} 
of $H^*_T(I\cX)$ (\emph{cf.} (\ref{eq:diffI_asymp})). 
This shows that $N$ components 
$H_{\sigma,v}^\lambda(q,z)$ 
of $H^\lambda(q,z)$ are linearly independent 
for generic values of $\lambda_1,\dots,\lambda_m$.  
\end{proof}

From this lemma, we know that there exist coefficient functions  
$c_{\sigma,v}(\lambda)$ such that 
\begin{equation}
\label{eq:oscint_H_coeff}
\cI^\lambda_{\Gamma_\R} (q,-z) = 
\sum_{(\sigma,v)\in I\cX^T} c_{\sigma,v}(\lambda) 
H_{\sigma,v}^\lambda(q,e^{\pi\iu} z). 
\end{equation} 
We will determine an analytic function 
$c_{\sigma,v}(\lambda)$ in $\lambda$ 
by putting $z=1$ and studying the asymptotic behavior of 
the both hand sides in the limit $q_a\searrow +0$. 

We start with the oscillatory integral. 
Take a fixed point $\sigma\in\cX^T$. 
Define $I^\sigma\in \cA$ by 
$I^\sigma=\{i\;;\; \sigma \notin \{z_i=0\}\}$. 
We can take $\{w_j \;;\; j\notin I^\sigma \}$ 
as a co-ordinate system on 
$Y_q\cap Y_\R=\Gamma_\R$.  
We can express $w_i$ for $i\in I^\sigma$ in terms of 
$\{w_j\;;\; j\notin I^\sigma\}$ and $q_a$, $a=1,\dots,r$  
by solving (\ref{eq:fibration_LG_formula}). 
Put 
\[
w_i = \prod_{a=1}^r q_a^{\ell_{ia}^\sigma} 
\prod_{j\notin I^\sigma} w_j^{b^{\sigma}_{ij}}, \quad i\in I^\sigma. 
\]
Here $(\ell_{ia}^\sigma)_{i\in I^\sigma,1\le a\le r}$ 
is the matrix inverse to $(\sfm_{ia})_{i\in I^\sigma, 1\le a\le r}$. 
Because $p_a\in \cl(\tC_\cX) \subset 
\sum_{i\in I^\sigma} \R_{\ge 0} D_i$, 
it follows that $\ell_{ia}^\sigma \ge 0$. 
We can see that $b^\sigma_{ij}\in \Q$ is determined by 
$b_i = \sum_{j\notin I^\sigma} b^\sigma_{ij} b_j$ in $\bN\otimes \R$. 
Let $V(\sigma)$ be $n!|\bN_{\rm tor}|$ times 
the volume of the convex hull of  
$\{b_j\;;\; j\notin I^\sigma\}\cup \{0\}$ in $\bN\otimes \R$.  
The holomorphic volume form $\omega_q$ 
can be written in terms of $\{w_j\;;\; j\notin I^\sigma\}$ 
as  
\[
\omega_q = \frac{1}{V(\sigma)} 
\prod_{j\notin I^\sigma} \frac{dw_j}{w_j}.  
\]
We set 
\[
\K_{\rm eff,\sigma} := 
\{d\in \bL\otimes \Q \;;\; \pair{D_i}{d}\in \Z_{\ge 0},  
\forall i\in I^\sigma\} 
= \bigoplus_{i\in I^\sigma} \Z_{\ge 0} \ell_i^\sigma.    
\]
Here, $\ell_i^\sigma\in \bL\otimes \Q$ is defined by 
$\pair{p_a}{\ell_i^\sigma} = \ell_{ia}^\sigma$. 
Then we have 
$\K_{\rm eff} = \bigcup_{\sigma\in \cX^T} \K_{\rm eff,\sigma}$. 
We denote by $\ov{p}^\lambda_a(\sigma)$ and 
$\ov{D}^\lambda_j(\sigma)$ the restrictions of 
$\ov{p}^\lambda_a, \ov{D}^\lambda_j \in H^*_T(\cX)$ 
to the fixed point $\sigma$. 
By using $\ov{D}^\lambda_i(\sigma) =0$ for $i\in I^\sigma$ and 
(\ref{eq:Dprel_equiv}), we calculate 
\begin{equation}
\label{eq:ovDjsigma}
\ov{p}^\lambda_a(\sigma) = 
\sum_{i\in I^\sigma} \lambda_i \ell_{ia}^\sigma, \quad 
\ov{D}^\lambda_j(\sigma) = 
-\lambda_j - \sum_{i\in I^\sigma} \lambda_i b_{ij}^\sigma, \quad 
j\notin I^\sigma.  
\end{equation} 
For a function $f(q_1,\dots,q_r)$ in $(q_1,\dots,q_r)\in (\R_{>0})^r$, 
we write $f(q_1,\dots,q_r) =O(M)$ for $M\in \R$ 
when $f(tq_1,\dots,tq_r) =O(t^M)$ as $t\searrow +0$. 
\begin{lemma}
\label{lem:expansion_oscint} 
Let $\sigma$ be a fixed point in $\cX$. 
For any $M>0$, there exists $M'>0$ such that the following holds. 
For $\lambda_1,\dots,\lambda_m$ such that 
$\Re(-\ov{D}^\lambda_j(\sigma)/(2\pi\iu)) >M'$ 
for all $j\notin I^\sigma$, $\cI^\lambda_{\Gamma_\R}(q,-1)$ with 
$(q_1,\dots,q_r)\in (\R_{>0})^r$ 
has the expansion  
\begin{align*} 
\cI^\lambda_{\Gamma_\R}&(q,-1) =
(-1)^n 
\frac{e^{(\lambda_1+\dots+\lambda_m)/2}}
{V(\sigma)} 
(e^{-\pi\iu \hrho}q)^{\frac{\ov{p}^\lambda(\sigma)}{2\pi\iu}}\times \\ 
 & \left(
\sum_{\substack{d\in \K_{\rm eff,\sigma}, \\ |d|<M}} 
\frac{(e^{-\pi\iu \hrho} q)^d} 
{\prod_{j\notin I^\sigma}
(1-e^{-2\pi\iu\pair{D_j}{d}-\ov{D}^\lambda_j(\sigma)}) 
\prod_{i=1}^m 
\Gamma(1+\pair{D_i}{d} + \frac{\ov{D}^\lambda_i(\sigma)}{2\pi\iu}) }  
+ O(M) \right),
\end{align*}
where $|d| := \sum_{a=1}^r \pair{p_a}{d}$ and we set 
\begin{align*}
(e^{-\pi\iu\hrho}q)^{\frac{\ov{p}^\lambda(\sigma)}{2\pi\iu}} &:= 
\prod_{a=1}^r 
(e^{-\pi\iu\rho_a}q_a)^{\frac{\ov{p}^\lambda_a(\sigma)}{2\pi\iu}}, 
\quad 
(e^{-\pi\iu\hrho}q)^d :
= \prod_{a=1}^r (e^{-\pi\iu\rho_a}q_a)^{\pair{p_a}{d}}. 
\end{align*}
\end{lemma} 
\begin{proof}
Using the notation above and (\ref{eq:ovDjsigma}), 
we can write 
\[
\cI_{\Gamma_\R}^\lambda(q,-1) = 
\frac{q^{\frac{\ov{p}^\lambda(\sigma)}{2\pi\iu}}}
{(2\pi\iu)^n V(\sigma)}
\int_{(0,\infty)^n} 
\exp \left( -\sum_{i\in I^\sigma} q^{\ell_i^\sigma} w_\sigma^{b_i} 
 \right )  
e^{-\sum_{j\notin I^\sigma} w_j}
w_\sigma^{-\frac{\ov{D}^\lambda(\sigma)}{2\pi\iu}}  
\frac{dw_\sigma}{w_\sigma},  
\]
where we put 
$w_\sigma^{b_i} := \prod_{j\notin I^\sigma} w_j^{b^\sigma_{ij}}$, 
$w_\sigma^{-\frac{\ov{D}^\lambda(\sigma)}{2\pi\iu}} := \prod_{j\notin I^\sigma} 
w_j^{-\frac{\ov{D}^\lambda_j(\sigma)}{2\pi\iu}}$ 
and $dw_\sigma/w_\sigma := \prod_{j\notin I^\sigma} (dw_j/w_j)$. 
Consider the Taylor expansion: 
\[
\exp\left( 
-\sum_{i\in I^\sigma} q^{\ell_i^\sigma} w_\sigma^{b_i} \right) 
= \sum_{\substack{n_i\ge 0\;;\; i\in I^\sigma, \\ 
|\sum_{i\in I^\sigma} n_i \ell_i^\sigma |<M} } 
\frac{\prod_{i\in I^\sigma} (-1)^{n_i}q^{ n_i\ell_i^\sigma}
w_\sigma^{n_ib_i}}{\prod_{i\in I^\sigma} n_i!}  
+O(M). 
\] 
When $\Re(-\ov{D}^\lambda_j(\sigma)/(2\pi\iu))$ is sufficiently big 
for all $j\notin I^\sigma$, each term in the right-hand side 
is integrable for the measure $e^{-\sum_{j\notin I^\sigma} w_j}
w_\sigma^{-\frac{\ov{D}^\lambda(\sigma)}{2\pi\iu}}  
(dw_\sigma/w_\sigma)$ on $(0,\infty)^n$. 
Therefore, we calculate 
\[
\cI_{\Gamma_\R}^\lambda(q,-1) = \frac{q^{\frac{\ov{p}^\lambda(\sigma)}{2\pi\iu}}}
{(2\pi\iu)^n V(\sigma)} 
\left( 
\sum_{\substack{d\in \K_{\rm eff,\sigma},\\ |d|<M} }
\frac{(-1)^{\sum_{i\in I^\sigma} n_i} q^d}
{\prod_{i\in I^\sigma}n_i!} 
\prod_{j\notin I^\sigma} 
\Gamma\left
(\sum_{i\in I^\sigma} n_i b_{ij}^\sigma -\tfrac{\ov{D}^\lambda_j(\sigma)}{2\pi\iu}
\right) +O(M)\right), 
\]
where $d=\sum_{i\in I^\sigma} n_i \ell_i^\sigma$. 
Using $n_i = \pair{D_i}{d}$, $\sum_{i\in I^\sigma} n_i b_{ij}^\sigma 
= -\pair{D_j}{d}$ and $\Gamma(z)\Gamma(1-z)= \pi /\sin (\pi z)$, 
we arrive at the formula in the lemma. 
\end{proof} 

Next we study the asymptotics of 
$H^\lambda_{\sigma,v}(q,e^{\pi\iu})$ 
in the limit $q\searrow +0$. 

\begin{lemma}
\label{lem:expansion_H}
Let $\sigma$ be a fixed point in $\cX$.  
For a given $M>0$, 
there exists an open set $V \subset (\C)^m$ 
such that both 
the expansion in Lemma \ref{lem:expansion_oscint} 
and the expansion 
\begin{align*} 
H_{\tau,v}^\lambda(q,e^{\pi\iu}) &= 
(-1)^n e^{(\lambda_1+\cdots+\lambda_m)/2} 
(e^{-\pi\iu\hrho} q)^{\frac{\ov{p}^\lambda(\sigma)}{2\pi\iu}} \\
\times &
\begin{cases}
\displaystyle
\sum_{\substack{d\in \K_{\rm eff,\sigma}; \\ \inv(v(d)) = v, |d|<M}}
\frac{(e^{-\pi\iu\hrho}q)^d}{\prod_{i=1}^m 
\Gamma(1+\pair{D_i}{d} + \frac{\ov{D}^\lambda_i(\sigma)}{2\pi\iu})}+O(M) \quad 
& \text{if }\tau = \sigma; \\
O(M) & \text{if }\tau\neq \sigma.
\end{cases} 
\end{align*} 
hold when $(\lambda_1,\dots,\lambda_m)\in V$. 
\end{lemma} 
\begin{proof} 
By the definition (\ref{eq:equiv_H}) of $H^\lambda(q,z)$,  
it suffices to show that both the expansion in 
Lemma \ref{lem:expansion_oscint} and 
the inequality 
\begin{equation}
\label{eq:psigma_ptau}
\sum_{a=1}^r \Re(\frac{\ov{p}^\lambda_a(\sigma)}{2\pi\iu}) +M 
< \sum_{a=1}^r \Re(\frac{\ov{p}^\lambda_a(\tau)}{2\pi\iu}), 
\quad \forall \tau \neq \sigma  
\end{equation} 
hold for some $(\lambda_1,\dots,\lambda_m)\in (\iu\R)^m$. 
(Note that these are open conditions for $\lambda$.)
Hereafter we take $\lambda_j$ to be purely imaginary. 
Recall that $\cX$ can be written as a symplectic quotient 
$\mathfrak{h}^{-1}(\eta)/\T_\R^r$ 
(\ref{eq:X_symplecticquot}) and is endowed with 
the reduced symplectic form depending on $\eta$. 
Without changing the orbifold $\cX$, 
we can choose the vector $\eta\in \bL\otimes \R$ 
to be $p_1+\cdots + p_r \in \tC_\cX$. 
Define a Hamiltonian function 
$\mathfrak{h}_{\eta,\lambda}\colon \cX\to \R$ by 
\[
\mathfrak{h}_{\eta,\lambda} (z_1,\dots,z_m) 
= -\sum_{i=1}^m \frac{\lambda_j}{2\pi\iu} |z_j|^2, \quad 
(z_1,\dots,z_m)\in \mathfrak{h}^{-1}(\eta).  
\]
This generates an (almost periodic) 
Hamiltonian $\R$-action $z_i \mapsto 
e^{-\lambda_i s}z_i, s\in \R$ on $\cX$. 
In general, an almost periodic Hamiltonian 
attains its global maximum at  
a critical point of index $2n = \dim_\R \cX$.  
(This follows from the so-called Mountain-Path Lemma and 
the fact that there are no critical points of odd index. 
See \emph{e.g.} \cite{audin}). 
Because the weights of $T_\sigma \cX$ for this $\R$-action are  
$\{\ov{D}^\lambda_j(\sigma)/(2\pi\iu)
\;;\; j\notin I^\sigma\}$, it follows that 
\begin{equation}
\label{eq:maximum} 
-\ov{D}^\lambda_j(\sigma)/(2\pi\iu) >0, \ \forall j \notin I^\sigma  
\Longrightarrow 
\text{$\mathfrak{h}_{\eta,\lambda}$ 
attains its unique maximum at $\sigma$} 
\end{equation} 
By (\ref{eq:ovDjsigma}), one can choose 
$(\lambda_1,\dots,\lambda_m)\in (\iu \R)^m$ 
so that $-\ov{D}^\lambda_j(\sigma)/(2\pi\iu)$, 
$j\notin I^\sigma$ are arbitrarily large positive numbers
and that the expansion in Lemma \ref{lem:expansion_oscint} holds.  
Then by (\ref{eq:maximum}), we know that 
$\mathfrak{h}_{\eta,\lambda}(\sigma)
>\mathfrak{h}_{\eta,\lambda}(\tau)$ 
for every other fixed point $\tau \neq \sigma$. 
On the other hand, using $\eta = p_1+\cdots+p_r$, 
we can easily show that 
$\mathfrak{h}_{\eta,\lambda}(\sigma) =
 -\sum_{a=1}^r \ov{p}^\lambda_a(\sigma)/(2\pi\iu)$.  
Therefore, by rescaling $\lambda_i$ if necessary, 
we can achieve the inequality (\ref{eq:psigma_ptau}). 
\end{proof}

Comparing the expansions in 
Lemmas \ref{lem:expansion_oscint} 
and \ref{lem:expansion_H}, 
we conclude 
\[
c_{\sigma,v}(\lambda) = 
\frac{1}{V(\sigma) 
\prod_{i\notin I^\sigma} 
(1-e^{-2\pi\iu f_{v}([D_i])-\ov{D}^\lambda_i(\sigma)})}, 
\]
where $c_{\sigma,v}$ is the coefficient 
appearing in (\ref{eq:oscint_H_coeff}) 
and $f_v([D_i])\in [0,1)$ is the rational number associated 
to $[D_i]\in H^2(\cX,\Z)$ (see Section \ref{sec:A-model} and 
(\ref{eq:toric_fv})). 
Hence, we find 
\[
c_{\sigma,v}(\lambda) = \frac{1}{V(\sigma)} 
\frac{\tTd^\lambda(T\cX)|_{(\sigma,v)} }{e_T(T_\sigma\cX_v)}, 
\]
where $\tTd^\lambda(T\cX)|_{(\sigma,v)}$ 
is the restriction of the equivariant Todd class $\tTd^\lambda(T\cX)$ 
to the fixed point $(\sigma,v)$ in $I\cX$ and 
$e_T(T_\sigma\cX_v)$ is the $T$-equivariant Euler class
of $T_\sigma \cX_v$ (regarded as a 
$T$-equivariant vector bundle over a point $\sigma$). 
Here $\lambda_i$ is regarded as an element of 
$H^2_T(\pt)$ and we used the fact that 
\begin{align*} 
e_T(T_\sigma \cX_v) & = 
\prod_{i\notin I^\sigma, f_v([D_i]) = 0 } 
\ov{D}_i^\lambda(\sigma).  
\end{align*} 
Since $V(\sigma)$ is the order of 
the automorphism group $\Aut(\sigma)$ at $\sigma\in \cX$, 
the Equation (\ref{eq:equivosc_H_str}) 
follows from the localization 
theorem in $T$-equivariant cohomology \cite{atiyah-bott}. 

\subsubsection{Proof of (\ref{eq:cc_sky})} 
We use (\ref{eq:cc_sky}) when we prove the matching of 
pairings in Appendix \ref{subsec:pairing}. 
Since the both-hand sides of (\ref{eq:cc_sky}) 
are monodromy-invariant, by (\ref{eq:cc_byH}),  
it suffices to prove that 
\[
(-1)^n i_{\pt}^*(H(q,z)) = 
\frac{1}{(2\pi\iu)^n} 
\int_{\Gamma_{\rm c}} e^{W_q/z} \omega_q 
\] 
where $i_\pt\colon \pt \to \cX \subset I\cX$ 
is an inclusion of a point and 
we used the fact that  
$[\cO_\pt^\vee]=(-1)^n [\cO_\pt]$. 
By the residue calculations, 
the right-hand side is (see (\ref{eq:W_q})): 
\[
\sum_{
\substack{(k_1,\dots,k_m)\in (\Z_{\ge 0})^m \\  
\sum_{i=1}^m k_i b_i = 0}}  
\frac{1}{k_1!\cdots k_m!}
\frac{q^{k_1 \ell_1 + \cdots + k_m \ell_m}}{z^{k_1+\cdots+k_m}}.  
\] 
Because $(k_1,\dots,k_m)$ appearing in the 
summation gives an element $d\in \bL$ 
such that $k_i = \pair{D_i}{d}$, we can see 
that this equals $(-1)^n i_\pt^*H(q,z)$ 
by (\ref{eq:H-series}).

\section{Integral periods and crepant resolution conjecture} 
\label{sec:integralperiods} 
In mirror symmetry for Calabi-Yau manifolds
(see \emph{e.g.} \cite{CDGP,morrison,deligne}), 
flat co-ordinates (or mirror map) $\tau_i$ on the B-model 
in a neighborhood of a maximally unipotent monodromy point 
was given by periods over integral cycles  
$A_1,\dots, A_r$ of a holomorphic $n$-form $\Omega$ 
\[
\tau_i = \int_{A_i} \Omega,   
\]
where $\Omega$ is normalized by the condition: 
\[
\int_{A_0} \Omega = 1. 
\]
Here, $A_0$ is a monodromy-invariant cycle (unique up to sign)  
and $A_1,\dots, A_r$ are such that they 
transforms under monodromy as 
$A_i \mapsto A_i + k_i A_0$. 
Thus, in Calabi-Yau case, \emph{flat co-ordinates are  
constructed as integral periods}. 

In this section, we consider integral periods 
in the A-model by choosing some integral structure on it.  
The integral structure in this section 
does not need to be the $\hGamma$-integral structure. 
We study relationships between integral periods and 
flat co-ordinates in the conformal limit (\ref{eq:conformal_limit}). 
Then we discuss why quantum parameters 
should be specialized to roots of unity in Y. Ruan's crepant 
resolution conjecture \cite{ruan-crc}. 
Throughout this section, we assume that $\cX$ is a 
weak Fano (\emph{i.e.} $\rho=c_1(\cX)$ is nef) 
Gorenstein projective orbifold without generic stabilizer. 

\subsection{Integral periods in the A-model} 
\label{subsec:intperiod_A} 
In what follows, we fix an integral lattice 
$\Sol(\cX)_\Z$ in the space $\Sol(\cX)$ of flat sections 
of the quantum $D$-module $QDM(\cX)$ satisfying 
\begin{itemize}
\item $\Sol(\cX)_\Z$ is invariant under the Galois action: 
$G^\Sol(\xi) \Sol(\cX)_\Z = \Sol(\cX)_\Z$, 

\item The pairing $(\cdot,\cdot)_{\Sol}$ restricts  
to a $\Z$-valued pairing: 
$\Sol(\cX)_\Z\times \Sol(\cX)_\Z \to \Z$.  
\end{itemize} 
An example is given by the $\hGamma$-integral structure 
$\Sol(\cX)_\Z = \cZ_K(K(\cX))$.  
(See Definition \ref{def:A-model_int}, 
Proposition \ref{prop:A-model_int}). 
An \emph{integral period} in the A-model is 
defined to be a pairing between 
a section of $QDM(\cX)$ 
and an element of $\Sol(\cX)_\Z$. 
The quantum cohomology central charge 
(\ref{eq:qc_centralcharge}) 
is an example of integral periods. 

We set up the notation. 
We set $\cV := H^*_{\rm orb}(\cX)$. 
This is identified with $\Sol(\cX)$ via 
the cohomology framing $\cZ_{\rm coh}
\colon \cV = H^*_{\rm orb}(\cX) \to \Sol(\cX)$ 
in (\ref{eq:coh_framing}). 
The integral structure $\Sol(\cX)_\Z$ 
induces an integral lattice 
$\cV_\Z$ in $\cV$: 
\[
\cV_\Z := \cZ_{\rm coh}^{-1} (\Sol(\cX)_\Z) 
\subset \cV = H_{\rm orb}^*(\cX).  
\] 
For $A \in \cV_\Z$ and 
$\alpha \in H^*_{\rm orb}(\cX)$, we put 
\begin{align}
\label{eq:Amodel_period}
\begin{split} 
\Pi_A^\alpha(\tau,z) &:= 
(\alpha, \cZ_{\rm coh}(A)(\tau,z))_{\rm orb}  \\ 
&= (L(\tau,-z)^{-1} \alpha, 
z^{-\mu} z^{\rho} A)_{\rm orb},   
\end{split} 
\end{align} 
where we used $L(\tau,-z)^{-1} = L(\tau,z)^\dagger$.  
The quantum cohomology central charge (\ref{eq:qc_centralcharge}) 
is given by $Z(V) = c(z) \Pi^{\unit}_{\Psi(V)}$. 
(We do not need the $\hGamma$-class 
to define $\Pi_A^\alpha$.)  

In order to consider the integral 
periods (\ref{eq:Amodel_period}) 
without $\log z$ terms,  
we introduce the sublattice 
$\cV_{\Z,\rho} \subset \cV_\Z $ by 
\[
\cV_{\Z,\rho} := \Ker(\rho) \cap \cV_\Z. 
\]
By the assumption that $\cX$ is Gorenstein, 
all the ages $\iota_v$ are integers and 
$H_{\rm orb}^*(\cX)$ is graded by even integers. 
Therefore, an element of $\cV_{\Z,\rho}$ corresponds 
to a flat section which is single-valued 
(when $n=\dim_\C\cX$ is even) 
or two-valued (when $n$ is odd) under $\cZ_{\rm coh}$. 
We write the integral period 
$\Pi^\alpha_A$ for $A\in \cV_{\Z,\rho}$ 
as a pairing on the ``two-valued Givental space" $\hcH$: 
\[
\hcH:= \cH \otimes_{\cO(\C^*)} \cO(\C^*_{z^{1/2}}),  
\]
where $\C^*_{z^{1/2}} \to \C^*$ denotes 
the double cover of the $z$-plane. 
The pairing (\ref{eq:pairing_cH}) 
on $\cH$ is naturally extended to 
$\hcH$ as 
\[
(\alpha(z^{1/2}),\beta(z^{1/2}))_{\cH} = 
(\alpha(\iu z^{1/2}), \beta(z^{1/2}))_{\rm orb}.  
\]
Then we have for $A\in \cV_{\Z,\rho}$ 
\begin{equation}
\label{eq:monodromyfree_integralperiod}
\Pi_A^\alpha(\tau,z) = (\J_\tau \alpha, z^{-\mu}A)_{\cH},    
\end{equation}  
where $\J_\tau \alpha = L(\tau,z)^{-1}\alpha$ 
is given in (\ref{eq:Linv}).  
Recall that $\J_\tau \alpha$ 
is lying on the semi-infinite Hodge structure 
$\F_\tau$ in Section \ref{subsec:Jfunct}. 

\subsection{Conformal limit and integral periods} 
\label{subsec:conflimit} 
By \emph{conformal limit} we mean the following 
limit sequence in $H^2_{\rm orb}(\cX)$:  
\begin{equation}
\label{eq:conformal_limit}
\tau - s \rho, \quad \Re(s) \to \infty  
\end{equation} 
with a fixed $\tau \in H^2_{\rm orb}(\cX)$.  
Using the assumption that $\rho=c_1(\cX)$ is nef, 
we can define the conformal limit of 
$\J_\tau = L(\tau,z)^{-1}$ as follows: 
\begin{align} 
\label{eq:conf_Linv}
\begin{split}
& \Jc_\tau \alpha  
:= \lim_{\Re(s) \to \infty} 
e^{s\rho/z} \J_{\tau- s \rho} \alpha \\ 
& = e^{\tau_{0,2}/z} \biggl( \alpha + 
\sum_{\substack{(d,l)\neq (0,0),\\ d\in \Ker(\rho)}} 
\sum_{i=1}^N \frac{1}{l!}
\corr{\alpha,\tau_{\rm tw},\dots,\tau_{\rm tw}, 
\frac{\phi_i}{z-\psi}}_{0,l+2,d}^\cX 
e^{\pair{\tau_{0,2}}{d}}\phi^i \biggr).  
\end{split} 
\end{align} 
Here we put $\tau = \tau_{0,2} + \tau_{\rm tw}$ with 
$\tau_{0,2} \in H^2(\cX)$ and 
$\tau_{\rm tw} \in \bigoplus_{\iota_v=1} H^0(\cX_v)$ 
and used (\ref{eq:Linv}) and the fact that 
$\pair{\rho}{d}\ge 0$ for all $d \in \Eff_\cX$. 
When $\alpha \in H^{2k}_{\rm orb}(\cX)$, 
$\Jc_\tau(\alpha)$ is homogeneous of degree $2k$ 
if we set $\deg(z) =2$. 

\begin{definition}
\label{def:CY_limit_VHS}
Assume that $\rho=c_1(\cX)$ is nef. 
Let $\tau\mapsto \F_\tau$ be the quantum cohomology 
\seminf VHS in Section \ref{subsec:Jfunct}. 
The \emph{conformal limit quantum cohomology \seminf VHS} 
is defined to be 
\[ 
\Fc_\tau := \lim_{\Re(s)\to \infty} 
e^{s\rho/z} \F_{\tau-s\rho}
= \Jc_\tau(H^*_{\rm orb}(\cX)\otimes \cO(\C^*)), 
\quad \tau 
\in H^2_{\rm orb}(\cX).  
\]
This satisfies $\Fc_{\tau+ a\rho} = e^{a\rho/z} \Fc_{\tau}$ 
and is homogeneous 
$(z\partial_z +\mu) \Fc_\tau \subset 
\Fc_\tau$. 
\end{definition}  
\begin{remark}
The new \seminf VHS $\Fc_\tau$ 
can be also defined in terms of 
the ``conformal quantum product"  
$\lim_{\Re(s)\to \infty} \circ_{\tau-s\rho}$ 
and the Dubrovin connection associated to it.  
This conformal limit of quantum cohomology 
is closely related to Y.\ Ruan's quantum corrected ring 
\cite{ruan-crc}, which is defined by 
counting rational curves contained 
in the exceptional locus (in the case of 
crepant resolution). 
The conformal limit of a \seminf VHS  
appears in the work of 
Sabbah \cite[Part I]{sabbah-hypergeometric}   
as the associated graded of a free 
$\C[z]$-module $G_k$ (an algebraization of $z^{-k}\F_\tau$) 
with respect to the  Kashiwara-Malgrange 
$V$-filtration at $z=\infty$. 
See also Hertling and Sevenheck 
\cite[Section 7]{hertling-sevenheck} for a review. 
\end{remark}

In the conformal limit, the \seminf VHS 
reduces to a \emph{finite dimensional VHS}. 
We define subspaces 
$H_0$, $\hbFc_\tau$ of $\hcH$ by 
\[
H_0:= \Ker(z\partial_z +\mu), 
\quad 
\hbFc_\tau := \Fc_\tau \otimes_{\cO(\C^*)} 
\cO(\C^*_{z^{1/2}}).  
\] 
The pairing $(\cdot,\cdot)_{\cH}$ 
on $\hcH$ induces a $(-1)^n$-symmetric 
$\C$-valued pairing $(\cdot,\cdot)_{H_0}$ on $H_0$. 
The semi-infinite flag 
$\cdots \supset z^{-1} \hbFc_\tau \supset 
\hbFc_\tau \supset z 
\hbFc_\tau \supset \cdots$ 
restricts to a finite dimensional flag 
$H_0= \rsF_\tau^0 \supset \rsF_\tau^1 
\supset \cdots \supset \rsF_\tau^n \supset 0$: 
\begin{align*} 
\rsF^{p}_\tau := z^{p-n/2}\hbFc_\tau \cap H_0  
= \Span_\C \left
\{ z^{p-n/2} \Jc_{\tau}(z^j \alpha)\;;\; \alpha 
\in H^{2n-2p-2j}_{\rm orb}(\cX),j\ge 0 
\right\}  
\end{align*} 
satisfying  
the Griffiths transversality and 
Hodge-Riemann bilinear relation: 
\[
\parfrac{}{t^i} \rsF_\tau^p \subset \rsF_\tau^{p-1}, \quad 
(\rsF_\tau^p, \rsF_\tau^{n-p+1})_{H_0} =0. 
\]
Conversely, the finite dimensional VHS 
$\rsF^\bullet_\tau$ recovers $\Fc_\tau$ by 
\[
\Fc_\tau = z^{-n/2} \rsF^n_\tau \otimes \cO(\C) + 
z^{-n/2+1} \rsF^{n-1}_\tau \otimes \cO(\C) + 
\dots  + z^{n/2} \rsF^0_\tau \otimes \cO(\C).  
\]
The integral structure on the A-model \seminf VHS 
does not induce a full integral lattice in $H_0$. 
One can see however that the lattice 
$\cV_{\Z,\rho}$ is naturally 
contained in $H_0$ by $A \mapsto  z^{-\mu} A$  
as a {\it partial} lattice. 
An integral period for $A\in \cV_{\Z,\rho}$ 
is related to a period of the finite dimensional 
VHS $\{\rsF^p_\tau\subset H_0\}$ as follows. 

\begin{lemma} 
For $A\in \cV_{\Z,\rho}$ and 
$\alpha \in H^{2n-2p}_{\rm orb}(\cX)$, 
the integral period $\Pi_A^\alpha(\tau,z)$ 
in (\ref{eq:monodromyfree_integralperiod}) 
converges to the period of the finite dimensional VHS 
${\rsF^p_\tau \subset H_0}$ in the conformal limit: 
\begin{equation}
\label{eq:limitperiod_c}
\lim_{\Re(s)\to \infty} 
z^{p-n/2} \Pi_A^\alpha(\tau-s \rho,z) = 
(z^{p-n/2} \Jc_{\tau} \alpha, z^{-\mu} A)_{H_0} 
\in  \C.    
\end{equation} 
Note that $z^{p-n/2}\Jc_\tau \alpha \in \rsF^p_\tau$ 
and that the limit 
depends only on $\tau \in H^2_{\rm orb}(\cX)/\C\rho$. 
\end{lemma} 

\begin{remark}
When the real structure $\Sol(\cX)_\Z \otimes \R$
makes $\Fc_\tau$ a pure and polarized \seminf VHS 
(see \cite[Section 2]{iritani-realint-preprint}), 
the finite dimensional VHS $\rsF^p_{\tau}$ 
satisfies the Hodge decomposition 
and Hodge-Riemann bilinear inequality:  
\[
H_0 = \rsF^p_\tau \oplus \kappa_{H_0}(\rsF^{n-p+1}_\tau), 
\quad 
(-\iu)^{2p-n} (\phi, \kappa_{H_0}(\phi))_{H_0} >0 
\]
where $\kappa_{H_0}$ is the real involution on $H_0$ and 
$\phi \in \rsF^{p}_\tau \cap \kappa_{H_0}(\rsF^{n-p}_\tau)$. 
\end{remark} 

\subsection{Co-ordinates on 
$H^2_{\rm orb}(\cX)$ via integral periods} 

We use periods for $\rsF^n_\tau \subset H_0$ 
to construct a co-ordinate system on $H^2_{\rm orb}(\cX)$.  
Note that $\rsF^n_\tau = z^{n/2} \hbFc \cap H_0 
=z^{n/2}\Jc_\tau(H^0_{\rm orb}(\cX))$ 
is one dimensional over $\C$. 
Using the Galois action, we take 
a good set of integral vectors in $\cV_{\Z,\rho}$ 
to measure $\rsF^n_\tau$. 

Choose an ample line bundle $L$ 
pulled back from the coarse moduli space $X$ of $\cX$. 
Then the Galois action $G^\Sol([L])$ 
is unipotent (since $f_v([L])=0$ in (\ref{eq:cohfr_property}))  
and its logarithm $\cN = 
\Log(\cZ_{\rm coh}^{-1} G^\Sol([L]) \cZ_{\rm coh}) 
= -2\pi\iu c_1(L)$   
defines a weight filtration $W_k$ on $\cV$.  
This is an increasing filtration 
characterized by the condition: 
\[
\cN W_k \subset W_{k-2}, \quad 
\cN^k \colon \Gr^W_k(\cV) \cong \Gr^W_{-k}(\cV),   
\]
where $\Gr^W_k(\cV) = W_k/W_{k-1}$. 
It is given by (independent of a choice of $L$)  
\begin{equation}
\label{eq:weightfiltr_ampleLoncoarse} 
W_k = \bigoplus_{v\in \sfT} H^{\ge n_v-k}(\cX_v).  
\end{equation} 
The weight filtration is defined over $\Q$ 
(with respect to the lattice $\cV_\Z$). 
Similarly, the subspace 
$\Ker(H^2(\cX)) := \{\alpha\in \cV=H^*_{\rm orb}(\cX) 
\;;\; \tau_{0,2}\cdot \alpha=0, \forall 
\tau_{0,2}\in H^2(\cX)\}$ 
is also characterized by Galois actions  
and is defined over $\Q$. 
These subspaces define 
the following filtration on $\cV_{\Z,\rho}$: 
\[
(W_{-n} \cap \cV_{\Z,\rho}) \subset 
(\Ker(H^2(\cX))\cap W_{-n+2} \cap \cV_{\Z,\rho}) \subset 
(W_{-n+2} \cap \cV_{\Z,\rho})  
\]
which are full lattices of the vector spaces: 
\[
H^{2n}(\cX)\subset 
H^{2n}(\cX) \oplus \bigoplus_{n_v=n-2} H^{2n_v}(\cX_v) 
\subset 
(H^{\ge 2n-2}(\cX)\cap \Ker(\rho)) 
\oplus \bigoplus_{n_v=n-2} H^{2 n_v}(\cX_v). 
\]
Since $\cX$ is Gorenstein, 
we have no $v\in \sfT$ satisfying $n_v=n-1$ 
and $\iota_v=1$ if $n_v=n-2$. 
Thus these subspaces are contained in 
$H^{\ge 2n-2}_{\rm orb}(\cX)$. 
We take integral vectors $A_0,A_1,\dots, A_\flat$, 
$A_{\flat+1},\dots,A_\sharp$ in $\cV_{\Z,\rho}$ 
compatible with this filtration: 
\begin{align*}
W_{-n} \cap \cV_{\Z,\rho} &= \Z A_0, \\
\Ker(H^2(\cX)) \cap W_{-n+2} \cap \cV_{\Z,\rho}
& = \Z A_0 + \textstyle\sum_{i=1}^\flat \Z A_i, \\ 
W_{-n+2} \cap \cV_{\Z,\rho} 
&= \Z A_0 + \textstyle\sum_{i=1}^\flat \Z A_i
+ \sum_{i=\flat+1}^\sharp \Z A_i. 
\end{align*} 
The vector $A_0\in H^{2n}(\cX)$ is unique up to sign 
and invariant under all Galois action. 
In analogy with the Calabi-Yau B-model,  
we normalize a generator $\Omega_\tau 
\in \rsF^n_\tau = z^{n/2}\Jc(H^0_{\rm orb}(\cX))$ 
by the condition 
\begin{equation}
\label{eq:normalization_Omega} 
(\Omega_\tau, z^{-\mu} A_0)_{H_0} = 1.  
\end{equation} 
Using the expression (\ref{eq:conf_Linv}), one can easily see that 
$\Omega_\tau = z^{n/2}\Jc_\tau((\iu^{n} a_0)^{-1} \unit)$
with $a_0 := (\unit,A_0)_{\rm orb}$.  
The \emph{normalized integral period} $\ov\Pi_A(\tau)$ 
of $A\in \cV_{\Z,\rho}$ is defined by  
(\emph{cf.} (\ref{eq:limitperiod_c})) 
\[
\ov\Pi_A(\tau) := (\Omega_\tau, z^{-\mu}A)_{H_0}, 
\quad \tau \in H^2_{\rm orb}(\cX).  
\] 
The filter $W_{-n+2} \cap \cV_{\Z,\rho}$ 
does not necessarily span  
$H^{\ge 2n-2}_{\rm orb}(\cX) \cap \Ker(\rho)$. 

\begin{proposition} 
\label{prop:integralperiods_basic} 
For $\tau \in H^2_{\rm orb}(\cX)$, 
we write $\tau = \tau_{0,2} + \tau_{\rm tw} = 
\tau_{0,2} + \tau_{\rm tw}' + \tau_{\rm tw}''$ 
with $\tau_{0,2} \in H^2(\cX)$, 
$\tau_{\rm tw} \in \bigoplus_{\iota_v=1}H^0(\cX_v)$, 
$\tau_{\rm tw}'\in \bigoplus_{n_v = n-2}H^0(\cX_v)$ 
and $\tau_{\rm tw}''\in 
\bigoplus_{n_v<n-2, \ \iota_v=1} H^0(\cX_v)$. 
Set $a_i:= (\unit, A_i)_{\rm orb}$. 
The normalized integral periods 
$\ov\Pi_{A_i}(\tau)$ 
give an affine co-ordinate system 
on the space 
$(H^2(\cX)/\C\rho) \oplus \bigoplus_{n_v=n-2} H^0(\cX_v)$:  
\begin{align*} 
\ov\Pi_{A_i}(\tau) &= 
a_0^{-1}a_i - (\tau_{\rm tw}', a_0^{-1}A_i)_{\rm orb}, \quad 
1\le i\le \flat,  \\ 
\ov\Pi_{A_i}(\tau) & =
a_0^{-1}a_i  - (\tau_{\rm tw}', a_0^{-1}A_i)_{\rm orb} 
- \frac{1}{2\pi\iu} \tau_{0,2}\cap [C_i],      
\quad \flat+ 1\le i \le \sharp 
\end{align*} 
where $[C_i]\in H_2(\cX)$ is the Poincar\'{e} dual 
of the $H^{2n-2}(\cX)$-component of $2\pi\iu a_0^{-1} A_i$ 
and 
\begin{equation}
\label{eq:integer_A_H2} 
[C_i]\in H_2(X,\Z)\cap \Ker\rho, \quad 
\text{where $X$ is the coarse moduli space of $\cX$}.  
\end{equation} 
Here, $[C_{\flat+1}],\dots,[C_{\sharp}]$ form a 
$\Q$-basis of $H_2(X,\Q)\cap \Ker\rho$. 
The period of a class 
$B \in \Ker(H^2(\cX))\cap \cV_{\Z,\rho}$ 
is possibly non-linear and has the asymptotic 
\[
\ov\Pi_B(\tau) \sim 
a_0^{-1}b - (\tau_{\rm tw}, a_0^{-1} B)_{\rm orb}, \quad 
b := (\unit, B)_{\rm orb}  
\]
in the large radius limit (\ref{eq:largeradiuslimit}).    
The constant term $a_0^{-1} a_i$ (resp.\ $a_0^{-1}b$)  
is a rational number if the following 
condition (\ref{eq:curveclass_overQ}) 
(resp.\ (\ref{eq:KerH2_rat})) holds.  
\begin{align} 
\label{eq:curveclass_overQ} 
& \text{The projection } 
W_{-n+2} \cap \Ker(\rho) 
\to H^{2n}(\cX) 
\text{ is defined over $\Q$}. \\ 
\label{eq:KerH2_rat} 
& \begin{cases} 
H^*(\cX)\text{ is generated by }H^2(\cX) \text{ and} \\   
\forall v \in \sfT \ 
(v \neq 0 \ \Longrightarrow \ \exists 
\xi \in H^2(\cX,\Z) \text{ such that } f_v(\xi)>0). 
\end{cases} 
\end{align} 
Here the projection in (\ref{eq:curveclass_overQ}) 
is to take the $H^{2n}(\cX)$-component. 
Recall that $W_{-n+2} \cap \Ker(\rho) 
= (H^{\ge 2n-2}(\cX) \cap \Ker(\rho)) 
\oplus \bigoplus_{n_v=n-2} H^{2n_v}(\cX_v)$.  
\end{proposition} 
\begin{proof} 
By (\ref{eq:conf_Linv}) and the string equation (see \cite{AGV}), 
$\Jc_{\tau} \unit$ can be written as follows:  
\begin{align*} 
\Jc_\tau \unit & = e^{\tau_{0,2}/z} 
\Biggl (1 + \frac{\tau_{\rm tw}}{z} + 
\sum_{\substack{d\in \Eff_\cX\cap \Ker(\rho),  
\\ l\ge 0,\\ 
d=0 \Rightarrow l\ge 2.}}
\sum_{i=1}^N 
\corr{\tau_{\rm tw},\dots,\tau_{\rm tw},  
\frac{\phi_i}{z(z-\psi)} }_{0,l+1d}^\cX 
e^{\pair{\tau_{0,2}}{d}} \phi^i \Biggr ) \\ 
& = 1 + \frac{\tau}{z} + 
z^{-2} H^{\ge 4}_{\rm orb}(\cX)\otimes \C[z^{-1}]   
\end{align*}  
The expressions for $\ov\Pi_{A_i}(\tau), \ov\Pi_B(\tau)$ 
easily follow from this. 

If $\xi \in H^2(X,\Z)$ 
is an integral class on the coarse moduli space, 
$G^{\Sol}(\xi)$ acts on $\cV$ by $e^{-2\pi\iu\xi}$ 
by (\ref{eq:cohfr_property}).  
Because the Galois action preserves the integral structure, 
$e^{-2\pi\iu\xi} A_i = A_i - m_i A_0$ 
for some integer $m_i$. 
Here, $2\pi\iu\xi A_i = m_i A_0$. 
Hence, 
$\xi\cap [C_i] = (\xi, 2\pi\iu a_0^{-1} A_i)_{\rm orb} = 
a_0^{-1} (\unit, 2\pi\iu\xi A_i)_{\rm orb} = m_i\in \Z$. 
This shows (\ref{eq:integer_A_H2}). 

Under the condition (\ref{eq:curveclass_overQ}), 
the $H^{2n}(\cX)$-component of $A_i$ 
is of the form $c_i A_0$ for $c_i\in \Q$.  
Hence $a_i = (\unit,A_i)_{\rm orb}
= c_i (\unit, A_0) = c_i a_0$ 
and $a_0^{-1}a_i$ is rational. 

Under the condition (\ref{eq:KerH2_rat}), 
we have the decomposition 
$\Ker(H^2(\cX)) = H^{2n}(\cX) \oplus 
(\Ker(H^2(\cX)) \cap \bigoplus_{v\in \sfT} H^*(\cX_v))$.  
By a consideration of the Galois action, 
we can easily see that this is defined over $\Q$.  
The rationality of $a_0^{-1}b$ follows similarly. 
\end{proof} 

\begin{remark} 
The rationality of $a_0^{-1} a_i$, $a_0^{-1}b$ 
are related to the rationality of  
specialization values in crepant resolution conjecture. 
The condition (\ref{eq:curveclass_overQ}) 
is satisfied by the $\hGamma$-integral structure. 
See Section \ref{subsec:hGamma} below. 
\end{remark} 

\subsection{Example: $\hGamma$-integral structure} 
\label{subsec:hGamma} 
Here we take $\Sol(\cX)_\Z$ to be the 
$\hGamma$-integral structure in 
Definition \ref{def:A-model_int} 
and compute some examples of integral periods. 
The lattice $\cV_\Z$ is given by $\Psi(K(\cX))$. 
By a natural map from the $K$-group of coherent sheaves 
to the $K$-group of topological orbifold vector bundles, 
we can regard a coherent sheaf as an element of $K(\cX)$. 
The integral vector $A_0 \in W_{-n}\cap \cV_{\Z,\rho}$ 
comes from the structure sheaf 
$\cO_x$ of a non-stacky point $x\in \cX$: 
\[
A_0 = \Psi(\cO_x) = 
\frac{(2\pi\iu)^n}{(2\pi)^{n/2}}
[\pt]. 
\] 
Here, we used the Poincar\'{e} duality to identify 
$[\pt]\in H_0(\cX)$ with an element in $H^{2n}(\cX)$.  
Hence we have 
$\Omega_\tau = (-1)^n 
(2\pi)^{-n/2} z^{n/2} \Jc_\tau \unit$. 

\subsubsection{A smooth curve}  
Let $\cX=X$ be a manifold and $C\subset X$ be a 
smooth curve of genus $g$ such that 
$[C]\cap c_1(\cX) =0$.  
Then the structure sheaf $\cO_C(g-1)$ 
defines an integral vector 
$A_C\in W_{-n+2} \cap \cV_{\Z,\rho}$ 
\[
A_C: = \Psi(\cO_C(g-1)) = 
\frac{(2\pi\iu)^{n-1}}{(2\pi)^{n/2}} [C] 
\]
and an integral period 
\[
\ov\Pi_{A_C}(\tau)  
= - \frac{1}{2\pi\iu} [C]\cap \tau. 
\]

\subsubsection{A general element in 
$W_{-n+2}\cap \cV_{\Z,\rho}$} 
Let $\Psi(V)\in W_{-n+2}\cap \cV_{\Z,\rho}$ 
be an arbitrary element. 
Using the fact that the untwisted sector 
of $\hGamma(T\cX)$ is of the form $1 - \gamma \rho + H^{\ge 4}(\cX)$  
($\gamma$ is the Euler constant) and 
that $\rho \cdot \tch(V)=0$, 
we can see that the $H^{2n}(\cX)$ component of 
$\Psi(V)$ belongs to 
$(2\pi)^{-n/2}(2\pi\iu)^n H^{2n}(\cX,\Q)=\Q A_0$. 
Therefore, the condition (\ref{eq:curveclass_overQ}) 
holds for the $\hGamma$-integral structure. 
We have 
\[
\ov\Pi_{\Psi(V)}(\tau) = \int_{\cX} \ch(V) - 
(\tau'_{\rm tw}, a_0^{-1}\Psi(V))_{\rm orb}
- \frac{1}{2\pi\iu} \tau_{0,2} \cap [C].  
\]
for some $[C]\in H_2(X,\Z)\cap \Ker\rho$ 
and $a_0 = (2\pi)^{-n/2}(2\pi\iu)^n$. 

\subsubsection{A stacky point} 
\label{subsubsec:stackypt} 
Let $y \in \cX$ be a possibly stacky point. 
Let $\varrho \colon \Aut(y) \to \End(V)$ be a 
finite dimensional representation of 
the automorphism group of $y$. 
This defines a coherent sheaf $\cO_y\otimes V$ 
supported on $y$ and an integral vector 
$A_{(y,V)} 
:= \Psi(\cO_y\otimes V) \in 
\Ker(H^2(\cX)) \cap \cV_{\Z,\rho}$.  
Using Toen's Riemann-Roch theorem \cite{toen}, 
one calculates 
\[
A_{(y,V)}  
= \frac{(2\pi\iu)^n}{(2\pi)^{n/2}}  
\sum_{(g)\subset \Aut(y)}  
\frac{(-1)^{n+n_{v(g)}+\iota_{v(g)}}\Tr(\varrho(g^{-1}))}
{|C(g)| \prod_{j=1}^{n-n_{v(g)}} \Gamma(f_{g,j})} 
[\pt]_{v(g)}, 
\]
where the sum is over all conjugacy classes 
$(g)$ of $g\in \Aut(y)$, 
$C(g)$ is the centralizer of $g$, 
$v(g)\in \sfT$ is the inertia component 
containing $(y,g) \in I\cX$, 
$[\pt]_{v(g)}$ is the homology class of a point on $\cX_{v(g)}$ 
(represented by a map $\pt \to \cX_v$ of stacks), 
$f_{g,1},\dots f_{g, n-n_{v(g)}}$ are 
rational numbers in $(0,1)$ such that 
$\{e^{2\pi\iu f_{g,j}}\}_j$ is a multi-set of 
the eigenvalues $\neq 1$ of 
the $g$ action on $T_y\cX$. 
The corresponding integral period behaves 
\begin{align*} 
\ov\Pi_{A_{(y,V)}}(\tau) 
 \sim  \frac{\dim(V)}{|\Aut(y)|} + 
\sum_{\substack{(g)\subset \Aut(y) \\ \iota_{v(g)}=1}} 
\frac{\Tr(\varrho(g))}
{|C(g)| \prod_{j=1}^{n-n_{v(g)}} \Gamma(1-f_{g,j})} 
\tau_{\rm tw} \cap [\pt]_{v(g)}   
\end{align*}  
in the large radius limit. 
This is an exact formula 
if $y\notin \cX_v$ for all $v$ with 
$\codim \cX_v = n-n_v \ge 3$ 
or equivalently, $A_{(y,V)} \in 
\Ker(H^2(\cX))\cap W_{-n+2}\cap \cV_{\Z,\rho}$. 

Note that the subspace $\cV_{\rm top}
:=\bigoplus_{v\in \sfT} H^{2n_v}(\cX_v)
\subset \cV$ is spanned by the integral vectors 
$A_{(y,V)}$ above, so is defined over $\Q$ 
for the $\hGamma$-integral structure. 
(This may not be true for an arbitrary 
integral structure.) 
For an integral vector $\Psi(V)$ in $\cV_{\rm top}$,  
the period $\ov\Pi_{\Psi(V)}(\tau)$ 
takes the rational value $\int_{\cX} \ch(V)$ 
at the large radius limit. 

\subsection{Crepant resolution conjecture 
with an integral structure} 

Yongbin Ruan's crepant resolution conjecture \cite{ruan-crc} 
states that when $Y$ is a crepant resolution of  
the coarse moduli space $X$ of a Gorenstein orbifold $\cX$, 
\[
\pi\colon Y \to X, \quad \pi^*(K_X) = K_Y,  
\]
the (orbifold) quantum cohomology of $\cX$ and $Y$ 
are related by analytic continuation
in quantum parameters. 
This conjecture was formulated more precisely 
by Bryan-Graber \cite{bryan-graber} as 
an isomorphism of Frobenius manifolds 
(under the Hard Lefschetz condition). 
In the joint work \cite{CIT:I} with Coates and Tseng, 
based on the toric mirror picture, 
we gave a conjecture that 
the A-model \seminf VHS of $\cX$ and $Y$ 
are related by an $\cO(\C^*)$-linear symplectic transformation 
$\U\colon \cH^\cX \to \cH^Y$ between the Givental spaces.  
(This does not need the Hard Lefschetz condition.)  
This symplectic transformation $\U$ encodes 
all the information on relationships between 
the genus zero Gromov-Witten theories of $\cX$ and $Y$. 
See \cite{coates-ruan,iritani-rims} for expositions 
and \cite{coates-II} for local examples.  

In this section, we incorporate integral structures 
into this picture and propose a possible relationship 
between the $K$-group McKay correspondence  
and the crepant resolution conjecture. 
We use a superscript to distinguish 
the spaces $\cX$, $Y$, \emph{e.g.} $\cH^\cX$, $\cH^Y$ etc.  

\begin{proposal} 
\label{propo:crc_int} 
{\rm (a)} 
For each smooth Deligne-Mumford stack $\cX$ 
with a projective coarse moduli space, 
the space $\Sol(\cX)$ of flat sections 
of the quantum $D$-module admits 
a $\Z$-lattice $\Sol(\cX)_\Z$ 
which is given by the image of 
the topological $K$-group under 
a \emph{$K$-group framing} $\cZ_K^\cX$: 
\[
\cZ^\cX_K \colon K(\cX) \to \Sol(\cX), \quad 
V \mapsto L(\tau,z)z^{-\mu}z^\rho \Psi^\cX (V),   
\]
where $\Psi^\cX$ is a map from $K(\cX)$ to $H^*_{\rm orb}(\cX)$ 
and $L(\tau,z)$ is the fundamental solution (\ref{eq:fundamentalsol_L}).  
We hope that $\Sol(\cX)_\Z$ 
is given by the $\hGamma$-integral structure, 
namely, $\Psi^\cX$ is given by (\ref{eq:Psi}).  
In the discussion below, we only need to 
assume that $\cZ^\cX_K$ satisfies 
the conclusions of Proposition \ref{prop:A-model_int}. 

{\rm (b)} 
Let $Y$ be a crepant resolution of the coarse moduli space 
$X$ of a Gorenstein orbifold $\cX$. 
The \emph{$K$-group McKay correspondence} 
predicts that we have an isomorphism of $K$-groups 
\[
\U_K \colon K(\cX) \cong K(Y) 
\]
which preserves the Mukai pairing 
(given in Proposition \ref{prop:A-model_int}) 
and commutes with the tensor by 
a topological line bundle $L$ 
on the coarse moduli space of $\cX$, 
$\U_K(L \otimes \cdot) = \pi^*(L)\otimes \U_K(\cdot)$. 

{\rm (c)} 
The quantum $D$-modules $QDM(\cX)$, $QDM(Y)$ 
with integral structures $\Sol(\cX)_\Z$, $\Sol(Y)_\Z$ 
become isomorphic under analytic continuation.  
The isomorphism of $\Sol(\cX)_\Z$
and $\Sol(Y)_\Z$ are induced from 
the $K$-group McKay correspondence 
$\U_K\colon K(\cX) \to K(Y)$ 
via the $K$-group framings. 
 
In terms of the \seminf VHS introduced 
in Section \ref{subsec:Jfunct},  
we have a degree-preserving\footnote{The grading 
on $\cH$ is given by $\deg z =2$ and the grading on 
orbifold cohomology.} $\cO(\C^*)$-linear 
symplectic isomorphism $\U\colon \cH^\cX \to \cH^Y$ 
and a map $\Upsilon$ from a subdomain of 
$H_{\rm orb}^*(\cX)$ to a subdomain of $H^*(Y)$ 
(where the quantum cohomology is analytically 
continued) such that 
the \seminf VHS of $\cX$ and $Y$ 
are identified by $\U$ 
\[
\U (\F^\cX_\tau) = \F^Y_{\Upsilon(\tau)} 
\]
and that $\U$ is induced from $\U_K$ by the 
commutative diagram: 
\begin{equation}
\label{eq:CD_UK_U}
\begin{CD}
K(\cX) @>{\U_K}>> K(Y) \\ 
@V{z^{-\mu}z^{\rho}\Psi^\cX}VV  
@VV{z^{-\mu}z^{\rho} \Psi^Y}V  \\
\cH^\cX\otimes_{\cO(\C^*)}\cO(\widetilde{\C^*}) 
 @>{\U}>> 
\cH^Y \otimes_{\cO(\C^*)}\cO(\widetilde{\C^*}).  
\end{CD} 
\end{equation} 
where $\mu,\rho$ in the left/right vertical 
arrow are those for $\cX$/$Y$. 
\qed 
\end{proposal} 

We hope that the isomorphism $\U_K$ in (b) 
arises from a geometric correspondence such as 
Fourier-Mukai transformations. 
In fact, Borisov-Horja \cite{borisov-horja-FM} 
showed that an analytic continuation of solutions to 
the GKZ-system corresponds to a Fourier-Mukai 
transformation between $K$-groups of 
toric Calabi-Yau orbifolds. 

\begin{remark} 
As formulated in \cite{CIT:I, coates-ruan}, 
the symplectic transformation $\U$ identifies 
the Givental's Lagrangian cone (\ref{eq:Giventalcone}), 
\emph{i.e.}  $\U \cL^\cX = \cL^Y$. 
Thus the relationship of the genus zero descendant 
potentials of $\cX$ and $Y$ 
is completely described by $\U$. 
\end{remark} 

We discuss what follows from this proposal   
assuming $\cX$ is weak Fano, \emph{i.e.} $c_1(\cX)$ is nef. 
As discussed in \cite{CIT:I}, this picture implies that 
quantum cohomology of $\cX$ and $Y$ 
are identified via $\Upsilon$ and $\U$ 
as a family of algebras  
(not necessarily as Frobenius manifolds). 
However, the large radius limit points for 
$\cX$ and $Y$ are not identified under $\Upsilon$, 
so we need actual analytic continuations.  
We refer the reader to \cite{CIT:I, coates-ruan, iritani-rims} 
for these things. 
Let us first observe that integral periods of $\cX$ and $Y$ 
in the conformal limit match under $\Upsilon$ and $\U$ 
(see (\ref{eq:matching_integralperiods}) below). 
Because $\U_K$ commutes with the tensor 
by a line bundle pulled back from $X$, 
it follows that $\U$ must commute with $H^2(\cX)$ 
((b), Section 5 in \cite{CIT:I}; 
(b), Conjecture 4.1 in \cite{coates-ruan}), 
\emph{i.e.} 
\begin{equation}
\label{eq:U_commutes_H2}
\U (\alpha\cup \cdot) = \pi^*(\alpha) \cup \U(\cdot), \quad 
\alpha \in H^2(\cX).   
\end{equation} 
Since $\cX$ is weak Fano, 
by the discussion leading to 
Theorem 8.2 in \cite{coates-ruan} 
(essentially using Lemma 5.1 \emph{ibid.}),  
we know that $\Upsilon$ should map $H^2_{\rm orb}(\cX)$ 
to $H^2(Y)$:  
\[
\Upsilon(H^2_{\rm orb}(\cX)) \subset H^2(Y). 
\]
The conformal limit $\tau \to \tau -s \rho$, $\Re(s)\to \infty$ 
on $H^2_{\rm orb}(\cX)$ should also be mapped to 
the conformal limit on $H^2(Y)$ under $\Upsilon$ 
because this flow is generated by the Euler vector field 
and the two Euler vector fields should match under $\Upsilon$  
(the Euler vector field is a part of the 
data of a quantum $D$-module).  
Therefore, by (\ref{eq:U_commutes_H2}) and 
$\pi^*c_1(\cX) = c_1(Y)$, 
the conformal limit of the \seminf VHSs  
(Definition \ref{def:CY_limit_VHS}) 
also match under $\U$: 
\[
\U(\F_{\tau}^{{\rm c}, \cX}) = 
\F_{\Upsilon(\tau)}^{{\rm c}, Y}.  
\] 
In particular, 
the finite dimensional VHS's 
$(\rsF^{\cX,\bullet}_\tau\subset H_0^\cX)$, 
$(\rsF^{Y,\bullet}_\tau \subset H_0^Y)$ 
associated with these also match:  
\[
\U(\rsF^{\cX,\bullet}_{\tau}) = \rsF^{Y,\bullet}_{\Upsilon(\tau)}, \quad 
\U \colon \hcH^\cX \supset H_0^{\cX} 
\to H_0^{Y} \subset \hcH^Y.  
\] 
We used the fact that $\U$ induces a map from 
$H_0^{\cX} = \Ker(z\partial_z + \mu^{\cX})$ 
to $H_0^Y = \Ker(z\partial_z + \mu^Y)$.  
Set $\cV^\cX := H^*_{\rm orb}(\cX)$, $\cV^Y := H^*(Y)$ 
and let $\U_\cV \colon \cV^\cX \to \cV^Y$ 
be the map induced from $\U_K$ (via $\Psi$) 
\[
\begin{CD}
K(\cX) @>{\U_K}>> K(Y) \\ 
@V{\Psi^\cX}VV        @VV{\Psi^Y}V  \\ 
\cV^\cX @>{\U_\cV}>> \cV^Y.  
\end{CD} 
\]
This again commutes with $H^2(\cX)$ and 
is related to $\U$ by 
\[
\U = z^{-\mu^Y} z^{\rho^Y} \U_\cV z^{-\rho^\cX} z^{\mu^\cX} 
= z^{-\mu^Y} \U_\cV z^{\mu^\cX}.  
\] 
The integral structures $\Sol(\cX)_\Z$, $\Sol(Y)_\Z$ 
induce the lattices $\cV^\cX_{\Z} = 
\cZ_{\rm coh}^{-1}(\Sol(\cX)_\Z) = \Psi^\cX(K(\cX))$, 
$\cV^Y_\Z =\cZ_{\rm coh}^{-1}(\Sol(Y)_\Z) 
= \Psi^Y(K(Y))$ as before. 
Let $L$ be an ample line bundle on $X$.  
Consider the weight filtration $W_k^\cX$ 
(\ref{eq:weightfiltr_ampleLoncoarse}) 
on $\cV^\cX$ defined by 
the Galois action logarithm 
$-2\pi\iu c_1(L)$. 
The first term $W_{-n}^\cX$ of the weight filtration 
is given by $\Image(c_1(L)^n)$. 
Thus $\U_\cV(W_{-n}^\cX) = \Image(\pi^*(c_1(L))^n) 
= H^{2n}(Y)$. Note that $\pi^*(c_1(L))^n$ is non-trivial   
since $\pi \colon Y\to X$ is birational. 
Therefore, for the weight filtration $W_k^Y$ on $\cV^Y$ 
(defined by an ample class on $Y$), 
we have 
\[
\U_{\cV}(W_{-n}^\cX) = W_{-n}^Y. 
\]
As we did before, we use an integral vector 
$A_0^\cX$ (unique up to sign)  
in $W_{-n}^\cX\cap \cV_{\Z,\rho}^\cX$ 
to normalize a generator $\Omega^{\cX}_\tau 
\in \rsF^{\cX,n}_\tau$ and then 
use $A_0^Y:=\U_\cV(A_0^\cX)\in W_{-n}^Y \cap \cV_{\Z,\rho}^Y$ 
to normalize $\Omega^Y_\tau \in \rsF^{Y,n}_\tau$  
(see (\ref{eq:normalization_Omega})).  
Because the $\U$ preserves the pairing, we have
\[
\U(\Omega_\tau^\cX) = \Omega_{\Upsilon(\tau)}^Y.  
\]
When $A^\cX \in \cV_{\Z,\rho}^\cX = \cV_{\Z}^\cX\cap \Ker(c_1(\cX))$, 
the corresponding vector $A^Y =\U_\cV(A^\cX)$ 
belongs to $\cV_{\Z}^Y \cap \Ker(\pi^*(c_1(\cX))) 
= \cV_{\Z,\rho}^Y$ and the integral periods match 
\begin{equation}
\label{eq:matching_integralperiods}
\ov\Pi^\cX_{A^\cX}(\tau) = 
(\Omega_\tau^\cX, z^{-\mu}A^\cX)_{H_0^\cX} = 
(\Omega_{\Upsilon(\tau)}^Y, z^{-\mu}A^Y)_{H_0^Y} 
= \ov\Pi^Y_{A^Y}(\Upsilon(\tau)). 
\end{equation} 

Now we can make predictions   
on the specialization values of quantum parameters. 
Note that $\Ker(\pi^*H^2(\cX))\subset \cV^Y$ 
is defined over $\Q$. 
Take a basis $A_0^Y,A_1^Y,\dots,A_\natural^Y$ of 
$\Ker(\pi^*H^2(\cX))\cap W_{-n+2}^Y \cap \cV^Y_{\Z,\rho}$.  
These generate a full lattice in $H^{2n}(Y) \oplus 
(H^{2n-2}(Y) \cap \Ker \pi_*)$ over $\C$.  
By Proposition \ref{prop:integralperiods_basic}, 
the integral periods for 
$A_1^Y,\dots,A_\natural^Y$ are of the form:
\begin{align}
\label{eq:exccurve_intperiod}
\ov\Pi^Y_{A_i^Y}(\tau) = 
a_0^{-1} a_i - \frac{1}{2\pi\iu} [C_i]\cap \tau,
\quad  a_i :=(A_i^Y, 1).   
\end{align} 
Here $[C_1],\dots,[C_\natural]\in H_2(Y,\Z)\cap \Ker\pi_*$ 
are a $\Q$-basis of $H_2(Y,\Q)\cap \Ker\pi_*$. 
So $\ov\Pi^Y_{A_i^Y}(\tau)$, $1\le i\le \natural$,   
form an affine co-ordinate system on $H^2(Y)/\Image\pi^*$. 
The integral vector $A_i^\cX$ corresponding to $A_i^Y$ 
belongs to $\Ker(H^2(\cX)) \cap \cV_{\Z,\rho}^\cX$. 
From (\ref{eq:matching_integralperiods}),  
Proposition \ref{prop:integralperiods_basic} 
and examples in Section \ref{subsec:hGamma},  
Proposal \ref{propo:crc_int} 
yields the following prediction: 

\begin{itemize}
\item[(i)] 
Assume that the condition (\ref{eq:KerH2_rat})
holds for $\cX$.  
Then the integral periods 
$\ov\Pi^Y_{A_i^Y}(\tau)$ (\ref{eq:exccurve_intperiod})
for $Y$ take rational values 
at the large radius limit point of $\cX$. 

\item[(ii)] 
Assume in addition to (i) that the condition 
(\ref{eq:curveclass_overQ}) (with $\cX$ there replaced with $Y$) 
holds for the rational structure on $\cV^Y$. 
Then $a_0^{-1} a_i$ in (\ref{eq:exccurve_intperiod}) 
is rational, so the ``quantum parameter" 
$q_C := \exp([C]\cap \tau)$ with 
$[C]\in H_2(Y,\Z)\cap \Ker\pi_*$ for $Y$ 
specializes to a root of unity at the 
large radius limit point of $\cX$. 

\item[(iii)] Assume that Proposal \ref{propo:crc_int}  
holds for the $\hGamma$-integral structures 
on $\cX$ and $Y$. 
Let $C\subset Y$ be a smooth rational curve in the 
exceptional set. 
If $\U_K^{-1}$ sends $[\cO_C(-1)]\in K(Y)$ to 
$[\cO_x\otimes V]\in K(\cX)$ for $x = \pi(C)$ and 
some representation $V$ of $\Aut(x)$,  
the quantum parameter $q_C$ specializes to 
$\exp(-2\pi\iu \dim V/|\Aut(x)|)$ 
at the large radius limit point of $\cX$.  
\end{itemize} 

For the $A_n$ surface singularity resolution, 
each irreducible curve in the exceptional set 
corresponds to a one-dimensional 
irreducible representation of $\Z/(n+1)\Z$ 
under the McKay correspondence. 
If we use this McKay correspondence as $\U_K$, 
the prediction of specialization values 
made in (iii) is true \cite{CCIT:An}.  
Also, under the McKay correspondence, (iii) gives 
the same prediction (up to complex conjugation)  
made by Bryan-Graber \cite{bryan-graber}, 
Bryan-Gholampour \cite{bryan-gholampour}  
for the ADE surface singularities   
and $\C^3/G$ with a finite subgroup $G\subset SO(3)$. 

The equality (\ref{eq:matching_integralperiods}) 
of integral periods can also predict  
the co-ordinate change $\Upsilon$.  
See \cite[Example 2.16, Section 3.8]{iritani-rims}
for local Calabi-Yau examples.

\section{Appendix} 

\subsection{Proof of Lemma \ref{lem:qsmall_kouchnirenko}} 
\label{subsec:qsmall} 
Let $\Delta$ be a face of $\hS$ 
($0\le \dim \Delta \le n-1$). 
Let $B_\Delta\subset (\C^*)^r$ be the 
discriminant locus of $W_{q,\Delta}(y)$, 
\emph{i.e.} the set of points $q=(q_1,\dots,q_r)$ 
such that $W_{q,\Delta}(y)$ has a critical point 
$y\in (\C^*)^r$. It suffices to show that 
the closure $\ov{B_\Delta}$ of $B_\Delta$ in $\C^r$ 
does not contain the origin. 
Suppose $0\in \ov{B_\Delta}$. 
Then there exists a curve 
$\alpha\colon \Spec \C[\![T]\!] \to \ov{B_\Delta}$ 
such that $\alpha(0) = 0$ and $\alpha$ 
restricts to $\alpha\colon \Spec \C(\!(T)\!) \to B_\Delta$. 
We can find a critical point $y(T)$ 
of $W_{q=\alpha(T),\Delta}(y)$ 
defined over the field 
$\ov{\C(\!(T)\!)} = \bigcup_{k\in \N} \C(\!(T^{1/k})\!)$ 
of Puiseux series. 
We take the leading terms of the $T$-expansions: 
\begin{align*}
\alpha_a (T) &= c_a T^{d_a} + \text{h.o.t.}, 
\quad c_a \neq 0, \quad 1\le a\le r, \\ 
y_i(T) &= s_i T^{f_i} + \text{h.o.t.}, \quad 
s_i \neq 0, \quad 1\le i\le n. 
\end{align*} 
Note that $d_a>0$ since $\alpha(0)=0$. 
Put $h_i := \sum_{a=1}^r\ell_{ia}d_a$. 
(See Section \ref{subsubsec:LGmodel_def} for $\ell_{ia}$.)
We claim that the piecewise linear function 
$h\colon \bN\otimes \R \to \R$ on the fan $\Sigma$ 
defined by $h(b_i) = h_i$ 
for $1\le i\le m'$ is strictly convex (with respect to $\Sigma$) 
and $h(b_j) < h_j$ for $m' < j \le m$. 
Since $\sum_{a=1}^r d_a p_a \in \tC_\cX$, for each ``anticone" 
$I \in \cA$, there exist $k_i>0$, $i\in I$ such that 
$\sum_{a=1}^r d_a p_a = \sum_{i\in I} k_i D_i$. 
Using $p_a = \sum_{i=1}^m D_i \ell_{ia}$ 
and the exact sequence dual to (\ref{eq:exactsequence_toric}), 
we have a linear function $\varphi \colon \bN\otimes \R 
\to \R$ such that 
$\varphi(b_i) = h_i - k_i$ for $i\in I$ 
and $\varphi(b_i) = h_i$ for $i\notin I$. 
Since $\varphi$ is a linear function 
which coincides with $h$ on the cone 
$\sum_{i\notin I} \R_{\ge 0} b_i$, 
the claim follows. 
Now consider the leading term of the critical point 
equation $d W_{\alpha(T),\Delta}(y) = 0$: 
\[
0 = \sum_{b_i \in \Delta} \alpha(T)^{\ell_i} y(T)^{b_i} b_i 
= \left( \sum 
c^{\ell_i} s^{b_i} b_i \right) T^g  + \text{h.o.t.} 
\]
where $g$ is the minimal exponent 
and the last summation is over $1\le i\le m$ such that 
$h_i + \sum_{j=1}^n b_{ij} f_j = g$ and $b_i \in \Delta$. 
The above claim shows that the 
$b_i$'s appearing in the leading term 
span a cone in $\Sigma$ and are linearly independent. 
This is a contradiction.

\subsection{Proof of Lemma \ref{lem:PScond}} 
\label{subsec:proof_PS} 
Let $B \subset \cMo\times \C^*$ be a compact set. 
We need to show that $B' = \{(q,z,y)\;;\; (q,z)\in B, \ y\in Y_q, \ 
\|df_{q,z}(y)\|\le \epsilon \}$ is compact. 
Assume that there exists a divergent 
sequence $\{(q_{(k)},z_{(k)}, y_{(k)})\}_{k=0}^\infty$ in $B'$, 
\emph{i.e.} any subsequence of it does not converge. 
Take an arbitrary Hermitian norm $\|\cdot\|$ on $\bN\otimes \C$. 
Note that we have 
\[
\|df_{q,z}(y)\| = \frac{1}{|z|} 
\|\sum_{i=1}^m q^{\ell_i} y^{b_i}b_i\|.   
\]
By passing to a subsequence and renumbering $b_1,\dots,b_m$, 
we can assume that $q_{(k)}$ and $z_{(k)}$ converge and that 
$|y_{(k)}^{b_1}| \ge |y_{(k)}^{b_2}| \ge \cdots \ge |y_{(k)}^{b_m}|$
for all $k$. 
Since $0$ is in the interior of $\hS$, 
there exist $c_i>0$ such that $\sum_{i=1}^m c_i b_i=0$. 
Hence $\prod_{i=1}^m |y_{(k)}^{b_i}|^{c_i} = 1$. 
Because $y_{(k)}$ diverges, 
we must have $\lim_{k\to\infty} |y_{(k)}^{b_1}| = \infty$. 
Since $\|df_{q_{(k)},z_{(k)}}(y_{(k)})\|$ is bounded, we have 
\[
0 = \lim_{k\to \infty} \frac{|z_{(k)}|}{|y_{(k)}^{b_1}|} 
\|df_{q_{(k)},z_{(k)}}(y_{(k)})\|
= \lim_{k\to \infty} \| \sum_{i=1}^m q_{(k)}^{\ell_i} y_{(k)}^{b_i-b_1} b_i \|.  
\]
Because $|y_{(k)}^{b_i-b_1}|\le 1$, 
by passing to a subsequence again, 
we can assume that $y_{(k)}^{b_i-b_1}$ converges to $\alpha_i\neq 0$ 
for all $1\le i\le l$ and $y_{(k)}^{b_i-b_1}$ goes to $0$ for 
$i> l$. Then we have 
\[
0= \sum_{i=1}^l \tilde{q}^{\ell_i} \alpha_i b_i, \quad 
\tilde{q}= \lim_{k\to\infty} q_{(k)} \in \cMo.  
\]
Put $\xi_{(k),i}:=\log y_{(k),i}$. 
By choosing a suitable branch of the logarithm, 
we can assume that 
$\lim_{k\to \infty} \pair{\xi_{(k)}}{b_i-b_1} = \log \alpha_i$ 
for $1\le i\le l$ 
and $\lim_{k\to \infty} \pair{\Re(\xi_{(k)})}{b_i-b_1} = -\infty$ 
for $i> l$. 
Let $V$ be the $\C$ subspace of $\bN\otimes \C$ 
spanned by $b_i-b_1$ with $1\le i\le l$. 
Take the orthogonal decomposition $\bN\otimes \C \cong V\oplus V^\perp$
and write $\xi_{(k)}=\xi_{(k)}'+\xi_{(k)}''$, 
where $\xi'_{(k)}\in V$ and $\xi''_{(k)}\in V^\perp$. 
Then $\xi_{(k)}'$ converges to some $\xi'\in V$. 
Putting $\tilde{y}_i = \exp(\xi'_i)$, we have 
$\tilde{y}^{b_i-b_1}=\alpha_i$ for $1\le i\le l$ and so 
\begin{equation}
\label{eq:crit_W_Delta}
\sum_{i=1}^l \tilde{q}^{\ell_i} \tilde{y}^{b_i} b_i  
= \tilde{y}^{b_1}
(\sum_{i=1}^l \tilde{q}^{\ell_i} \tilde{y}^{b_i-b_1} b_i) 
=0.    
\end{equation} 
On the other hand, for a sufficiently big $k$, 
$\pair{\Re(\xi''_{(k)})}{b_i-b_1}=0$ for $1\le i\le l$ 
and $\pair{\Re(\xi''_{(k)})}{b_i-b_1}<0$ for $i>l$. 
This means that $b_1,\dots,b_l$ are on some face $\Delta$ of $\hS$. 
But the equation (\ref{eq:crit_W_Delta}) shows that 
$\tilde{y}$ is a critical point of $W_{\tilde{q},\Delta}$.  
This contradicts to the assumption that $W_{\tilde{q}}$ 
is non-degenerate at infinity. 

\subsection{The pairings match under mirror symmetry}
\label{subsec:pairing} 
We give a proof of $(\cdot,\cdot)_{\cRz} = 
(\tau\times \id)^*(\cdot,\cdot)_F$ 
in Proposition \ref{prop:Dmoduleiso}. 
Firstly we show that $(\cdot,\cdot)_{\cRz}$ is a 
constant multiple of $(\tau\times \id)^*(\cdot,\cdot)_F$.  
The argument here follows the line of 
\cite[Proposition 3.6]{CIT:I}, where the case 
$\cX = \Proj(1,1,1,3)$ was discussed. 
We work on the (pulled back) A-model $D$-module 
via the identification $\Mir$. 
Let 
\[ 
(\cdot,\cdot)_{{\rm B},(\tau(q),z)}   
\colon F_{(\tau(q),-z)} \times F_{(\tau(q),z)}  
\to \C
\] 
be the pairing induced 
from the B-model pairing $(\cdot,\cdot)_{\cRz}$ 
via $\Mir$. 
Via the fundamental solution (\ref{eq:fundamentalsol_L}),  
this induces a pairing $\Pair{\cdot}{\cdot}_{\rm B}$ 
on the Givental space $\cH$ (see (\ref{eq:Giventalsp}) 
and (\ref{eq:Giv_flat})):  
\[
\Pair{\alpha(z)}{\beta(z)}_{\rm B} := 
\left( L(\tau(q),-z)\alpha(-z), L(\tau(q),z)\beta(z) 
\right)_{{\rm B},(\tau(q),z)} 
\]
for $\alpha(z),\beta(z) \in \cH$. 
Since $(\cdot,\cdot)_{\rm B}$ is $\nabla$-flat, 
$\Pair{\alpha(z)}{\beta(z)}_{\rm B}$ is independent of $q$. 
By the discussion after Conjecture \ref{conj:mirrorthm}, 
the monodromy over $\cM$ gives all the Galois actions 
of $H^2(\cX,\Z)$. Since the B-model pairing 
is monodromy-invariant, we have 
\begin{equation}
\label{eq:mon-inv_B}
\Pair{\alpha(z)}{\beta(z)}_{\rm B} = 
\Pair{G^{\cH}(\xi)\alpha(z)}{G^{\cH}(\xi)\beta(z)}_{\rm B}. 
\end{equation} 
Taking $\xi \in H^2(\cX,\Z)$ to be  
classes pulled-back from the coarse moduli space $X$ 
(so that $f_v(\xi) = 0$ for $v\in \sfT$) 
and using (\ref{eq:GaloisH}),  
one can deduce 
\begin{equation}
\label{eq:02_adj} 
\Pair{\tau_{0,2}\cdot \alpha(z)}{\beta(z)}_{\rm B} 
= \Pair{\alpha(z)}{\tau_{0,2}\cdot \beta(z)}_{\rm B}, \quad 
\tau_{0,2} \in H^2(\cX). 
\end{equation} 
By (\ref{eq:mon-inv_B}) and (\ref{eq:02_adj}), 
one can see that the semisimple part 
$\bigoplus_{v\in \sfT} e^{2\pi\iu f_v(\xi)}$ of 
$G^{\cH}(\xi)$ also preserves $\Pair{\cdot}{\cdot}_{\rm B}$. 
This implies that, 
for $\alpha\in H^*(\cX_v)$, $\beta\in H^*(\cX_{v'})$,  
\begin{equation}
\label{eq:orth_B}
\Pair{\alpha}{\beta}_{\rm B} = 0 \quad \text{if $v' \neq \inv(v)$}.  
\end{equation} 
Here we used the fact that $v'=\inv(v)$ if 
$f_v(\xi) + f_{v'}(\xi) \in \Z$ 
for all $\xi\in H^2(\cX,\Z)$.  
By the definition of $\Pair{\cdot}{\cdot}_{\rm B}$, 
one has for $\alpha,\beta \in H^*_{\rm orb}(\cX)$, 
\begin{align*}
(\alpha,\beta)_{{\rm B},(\tau(q),z)} 
&= \Pair{L(\tau(q),z)^{-1}\alpha}{L(\tau(q),z)^{-1}\beta}_{\rm B} \\ 
&\sim \Pair{e^{\sum_{a=1}^{r'} \ov{p}_a \log q_a/z} \alpha}
{e^{\sum_{a=1}^{r'} \ov{p}_a \log q_a/z} \beta}_{\rm B}  
= \Pair{\alpha}{\beta}_{\rm B} \quad 
\text{ as $q\to 0$,}    
\end{align*} 
where we used (\ref{eq:Linv}),  
(\ref{eq:mirrormap_exp}) and (\ref{eq:02_adj}). 
Since the left-hand side is regular at $z=0$,  
we know that $\Pair{\alpha}{\beta}_{\rm B}$ is regular 
at $z=0$. 
Moreover, since $(\cdot,\cdot)_{\rm B}$ 
is $\nabla_{z\partial_z}$-flat, we have 
\begin{equation}
\label{eq:homog_B} 
z\partial_z \Pair{\alpha}{\beta}_{\rm B}  
= \frac{1}{2}(\deg \alpha + \deg \beta - 2n)
\Pair{\alpha}{\beta}_{\rm B} 
+ \frac{1}{z} \Pair{\rho \cdot \alpha}{\beta}_{\rm B} 
- \frac{1}{z} \Pair{\alpha}{\rho\cdot \beta}_{\rm B}     
\end{equation} 
by the second equation of (\ref{eq:diffeq_L}). 
The last two terms cancel by (\ref{eq:02_adj}) 
and so $\Pair{\cdot}{\cdot}_{\rm B}$ is of degree $-2n$ 
when we set $\deg z = 2$. 

Now we claim that $\Pair{\alpha}{\beta}_{\rm B} \in \C$ 
for $\alpha,\beta\in H^*_{\rm orb}(\cX)$.  
To show the claim, 
by (\ref{eq:02_adj}), (\ref{eq:orth_B}) 
and the Lefschetz decomposition, 
it suffices to show that $\Pair{\alpha}{\omega^k \beta}\in \C$
for primitive classes $\alpha\in H^*(\cX_v)$, 
$\beta \in H^*(\cX_{\inv(v)})$  
with respect to a K\"{a}hler class $\omega$. 
By the homogeneity (\ref{eq:homog_B}) of 
$\Pair{\cdot}{\cdot}_{\rm B}$, 
we have $\Pair{\alpha}{\omega^k\beta} 
\in \C z^{k + \frac{1}{2}(\deg \alpha + \deg \beta -2n)}$. 
By the regularity at $z=0$, this is zero unless  
$2k + \deg \alpha + \deg \beta \ge 2n$. 
When $2k + \deg \alpha + \deg \beta >2n$, 
it follows from the Lefschetz decomposition  
that $\omega^k \alpha =0 $ or $\omega^k \beta =0$.

By this claim, one has 
$(\alpha,\beta)_{{\rm B},(\tau(q),z)} 
= \Pair{L(\tau(q),z)^{-1} \alpha}
{L(\tau(a),z)^{-1}\beta}_{\rm B} 
= \Pair{\alpha}{\beta}_{\rm B} + O(1/z)$ 
for $\alpha,\beta\in H^*_{\rm orb}(\cX)$. 
Because $(\alpha,\beta)_{{\rm B}, (\tau(q),z)}$ 
is regular at $z=0$, we have 
$(\alpha,\beta)_{{\rm B},(\tau(q),z)} 
= \Pair{\alpha}{\beta}_{\rm B}\in\C$ 
and this is independent of $q$ and $z$.  
Now the $\nabla$-flatness of $(\cdot,\cdot)_{\rm B}$ 
gives the Frobenius property 
\[ 
(\tau_*(\partial_a) \circ  \alpha,\beta)_{\rm B} 
= (\alpha, \tau_*(\partial_a) \circ \beta)_{\rm B}, 
\quad \partial_a = q_a (\partial/\partial q_a),    
\]
where we identify $\tau_*(\partial_a)$ with 
a section of $(\tau\times \id)^*F$ 
and the subscript $(\tau(q),z)$ is omitted. 
Since $\tau_*(\partial_a)\circ$ corresponds to 
the multiplication by $\sfp_a$ in the Batyrev ring 
(see Proposition \ref{prop:Jac-Bat}),  
$\tau_*(\partial_a)$ generates the quantum 
cohomology over $\unit$. 
Therefore, the pairing $(\cdot,\cdot)_{\rm B}$ 
is completely determined by the value 
$(\unit,\gamma)_{\rm B}\in \C$ 
for a top dimensional class $\gamma \in H^{2n}(\cX)$ 
and is proportional to $(\cdot,\cdot)_F$. 

Finally, we fix the constant ambiguity.  
Theorem \ref{thm:cc_match} implies that 
the $\Gamma_\R$ and $\Gamma_{\rm c}$ corresponds 
to the linear functions  
$\chi(- \otimes \cO_\cX^\vee), 
\chi(- \otimes \cO_{\pt}^\vee)$ 
on the $K$-group. 
(See the proof of Theorem \ref{thm:pulledbackintstr} 
in Section \ref{subsubsec:pr_main}.)
The pairings match under this correspondence 
$\sharp(\Gamma_\R \cap \Gamma_{\rm c}) 
= 1 = \chi(\cO_\cX^\vee\otimes \cO_\pt)$, 
so the proportionality constant is one. 
\qed 

\bibliographystyle{amsplain}

\end{document}